\documentclass[a4paper, 12pt]{article}

\usepackage{amsmath} 
\usepackage{amstext}
\usepackage{amsthm}
\usepackage{amscd} 
\usepackage{amsopn} 
\usepackage{verbatim} 
\usepackage{amssymb}
\usepackage{amsfonts}
\usepackage{mathtools}
\usepackage[mathscr]{euscript}
\usepackage{fullpage}
\usepackage{enumitem}
\usepackage{ulem}
\usepackage{nicematrix}

\usepackage[bbgreekl]{mathbbol}
\usepackage{tikz}
\usepackage{tikz-cd}
\usetikzlibrary{babel}
\usetikzlibrary{matrix}
\usetikzlibrary{shapes}
\usetikzlibrary{calc,3d}
\usetikzlibrary{through}
\usetikzlibrary{arrows,calc,matrix,topaths,positioning,scopes,shapes,decorations,decorations.markings,decorations.pathreplacing,decorations.pathmorphing} 

\usepackage{graphicx}

\usetikzlibrary{shapes.geometric}

\tikzset{
    edge/.style={draw, -}
}

\tikzset{
    bu/.style={circle, fill, inner sep=1.2pt},
    bbrace/.style={decorate, decoration={brace, amplitude=3pt}},
    mbrace/.style={decorate, decoration={brace, amplitude=3pt, mirror}}
}

\tikzset{ext/.style={circle, draw,inner sep=1pt},int/.style={circle,draw,fill,inner sep=1pt},nil/.style={inner sep=1pt}}
\tikzset{exte/.style={circle, draw,inner sep=3pt},inte/.style={circle,draw,fill,inner sep=3pt}}
\tikzset{diagram/.style={matrix of math nodes, row sep=3em, column sep=2.5em, text height=1.5ex, text depth=0.25ex}}
\tikzset{diagram2/.style={matrix of math nodes, row sep=0.5em, column sep=0.5em, text height=1.5ex, text depth=0.25ex}}
\tikzset{every picture/.append style={baseline=-.65ex}}
\tikzset{de/.style={-latex}} 
\tikzset{ed/.style={latex-}} 

\tikzstyle{int}=[circle, draw,fill=black,outer sep=0,minimum size=3pt, inner sep=0]
\tikzstyle{ext}=[circle, draw=black,outer sep=0,inner sep=1pt]
\tikzstyle{black}=[circle,fill=black!,inner sep=0pt,minimum size=2mm]
\tikzstyle{white}=[circle,inner sep=0pt,draw=black,minimum size=2mm]
\tikzstyle{invisible}=[coordinate,inner sep=0pt,minimum size=2mm]
\tikzstyle{biw}=[circle split,inner sep=1pt,draw=black,minimum size=8mm]
\tikzstyle{biwb}=[rectangle,inner sep=1pt,draw=black,minimum size=8mm]
\tikzstyle{uniw}=[circle,inner sep=1pt,draw=black,minimum size=8mm]
\tikzstyle{unimin}=[circle,inner sep=1pt,draw=black,minimum size=5mm]

\tikzstyle{sq}=[draw, fill, rectangle, minimum size=3pt, inner sep=0pt]		
\tikzstyle{colie}   = [circle,thin, minimum width=4pt, draw, inner sep=0pt, path picture={\draw (path picture bounding box.south east) -- (path picture bounding box.north west) (path picture bounding box.south west) -- (path picture bounding box.north east);}]
\tikzstyle{coliel}   = [circle,thin, minimum width=6pt, draw, inner sep=0pt, path picture={\draw (path picture bounding box.south east) -- (path picture bounding box.north west) (path picture bounding box.south west) -- (path picture bounding box.north east);}]
\tikzstyle{lie}=[draw,thin,fill, circle, minimum size=4pt, inner sep=0pt]	
\tikzstyle{prelie}=[draw,thin, fill=blue!50, circle, minimum size=4pt, inner sep=0pt]	
\tikzstyle{coprod}=[draw,thin, circle, minimum size=4pt, inner sep=0pt]	
\tikzstyle{prod}=[draw,thin, fill=black, circle, minimum size=4pt, inner sep=0pt]	
\tikzstyle{circblack}=[draw,thin,fill, circle, minimum size=4pt, inner sep=0pt]

\usepackage{hyperref}

\providecommand{\eprint}[2][]{\href{http://arxiv.org/abs/#2}{arXiv:#2}}

\usepackage[OT2, T1]{fontenc}
\newcommand\cyrillic[1]{{\fontencoding{OT2}\fontfamily{wncyr}\selectfont#1}}
\newcommand\mathcyr[1]{\text{\cyrillic{#1}}}
\newcommand\shuffle{\textbf{\mathcyr{Sh}}}

\newcommand{\cin}{\mathsf{in}}
\newcommand{\cout}{\mathsf{out}}
\newcommand{\shape}{\mathsf{Shape}}

\newcommand{\cB}{{\mathcal B}}

\newcommand{\cF}{{\mathcal F}}

\newcommand{\cK}{{\mathcal K}}
\newcommand{\caL}{{\mathcal L}}

\newcommand{\cP}{{\mathcal P}}

\newcommand{\Hom}{{\mathrm H\mathrm o\mathrm m}}

\newcommand{\Ass}{\mathcal{A} \mathit{ss}}

\newcommand{\QP}{\mathcal{QP}\mathit{ois}}

\newcommand{\LieBi}{{\caL \mathit{ieb}}}

\newcommand{\LB}{\mathcal{L}\mathit{ieb}}
\newcommand{\LieTriang}{\mathcal{L}\mathit{ieb}^{\triangle}}

\newcommand{\ad}{\mathsf{ad}}

\newcommand{\Frob}{{\cF rob}}

\newcommand{\gr}{\mathrm{gr}}

\newcommand{\ar}{\mathsf{ar}}

\newcommand{\ldot}{{\:\raisebox{1.5pt}{\selectfont\text{\circle*{1.5}}}}}
\newcommand{\udot}{{\:\raisebox{4pt}{\selectfont\text{\circle*{1.5}}}}}

\newcommand{\ttt}{\text{-}}

\let\leq\leqslant
\let\geq\geqslant

\newcommand{\calB}{\mathcal{B}}
\newcommand{\calC}{\mathcal{C}}
\newcommand{\calD}{\mathcal{D}}

\newcommand{\calP}{\mathcal{P}}
\newcommand{\calQ}{{\mathcal{Q}}}
\newcommand{\calF}{{\mathcal{F}}}
\newcommand{\calI}{{\mathcal{I}}}

\newcommand{\calV}{\mathcal V}

\newcommand{\VV}{{\mathbb{V}}}

\newcommand{\Sgn}{\mathsf{Sgn}}
\newcommand{\End}{\mathit{End}}

\newcommand\Com{\mathsf{Comm}}

\newcommand\Lie{\mathsf{Lie}}

\newcommand{\Id}{\mathsf{Id}}

\newcommand{\one}{\mathbb{1}}
\newcommand{\sgn}{\mathsf{Sgn}}

\newcommand{\Free}{\mathrm{Free}}
\newcommand{\coFree}{\mathrm{coFree}}

\DeclareRobustCommand{\rchi}{{\mathpalette\irchi\relax}}
\newcommand{\irchi}[2]{\raisebox{\depth}{$#1\chi$}}

  \newcommand{\bF}{{\mathbb F}}

\newcommand{\bS}{{\mathbb S}} 
  
  \newcommand{\bZ}{{\mathbb Z}}

\newcommand\quis{\mathsf{quis}}

\numberwithin{equation}{subsection}

\newtheorem{theorem}[equation]{Theorem}
\newtheorem*{theorem*}{Theorem}
\newtheorem{proposition}[equation]{Proposition}
\newtheorem*{proposition*}{Proposition}

\newtheorem*{statement*}{Statement}
\newtheorem{lemma}[equation]{Lemma}
\newtheorem*{lemma*}{Lemma}
\newtheorem{corollary}[equation]{Corollary}
\newtheorem*{corollary*}{Corollary}

\newtheorem{definition}[equation]{Definition}
\newtheorem*{definition*}{Definition}
\newtheorem{notation}[equation]{Notation}
\newtheorem{remark}[equation]{Remark}
\newtheorem*{remark*}{Remark}
\newtheorem{example}[equation]{Example}
\newtheorem*{example*}{Example}

\newcommand{\antish}{\text{\raisebox{\depth}{\textexclamdown}}} 

\newcommand{\op}{\mathbf{op}}
\newcommand{\diop}{\mathsf{diop}}

\tikzset{ver/.style={circle, draw,inner sep=2pt}}

\newcommand{\liePic}[2]
{
\begin{tikzpicture}[scale=0.1pt, baseline=(current bounding box.center)]
	\draw (0,-0.55) -- (0,-2.5);
	\draw (0.5,0.5) -- (2.2,2.2);
	\draw (-0.48,0.48) -- (-2.2,2.2);
	\node[ext] (A) at (0,0)   {};
	\node at (2.7,2.8) {$\scriptscriptstyle #2$};
	\node at (-2.7,2.8) {$\scriptscriptstyle #1$};
\end{tikzpicture}
}

\newcommand{\coliePic}[2]
{
	\begin{tikzpicture}[scale=0.1pt,yscale=-1, baseline=(current bounding box.center)]
		\draw (0,-0.55) -- (0,-2.5);
		\draw (0.5,0.5) -- (2.2,2.2);
		\draw (-0.48,0.48) -- (-2.2,2.2);
		\node[ext] (A) at (0,0)   {};
		\node at (2.7,2.8) {$\scriptscriptstyle #2$};
		\node at (-2.7,2.8) {$\scriptscriptstyle #1$};
	\end{tikzpicture}
}

\newcommand{\multPic}[2]
{
	\begin{tikzpicture}[scale=0.1pt , baseline=(current bounding box.center)]
		\draw (0,-0.55) -- (0,-2.5);
		\draw (0.5,0.5) -- (2.2,2.2);
		\draw (-0.48,0.48) -- (-2.2,2.2);
		\node[int] (A) at (0,0)   {};
		\node at (2.7,2.8) {$\scriptscriptstyle #2$};
		\node at (-2.7,2.8) {$\scriptscriptstyle #1$};
	\end{tikzpicture}
}

\newcommand{\comultPic}[2]
{
	\begin{tikzpicture}[scale=0.1pt,yscale=-1 , baseline=(current bounding box.center)]
		\draw (0,-0.55) -- (0,-2.5);
		\draw (0.5,0.5) -- (2.2,2.2);
		\draw (-0.48,0.48) -- (-2.2,2.2);
		\node[int] (A) at (0,0)   {};
		\node at (2.7,2.8) {$\scriptscriptstyle #2$};
		\node at (-2.7,2.8) {$\scriptscriptstyle #1$};
	\end{tikzpicture}
}

\newcommand{\LLiePic}[3]
{{
\begin{tikzpicture}[x=1.2mm,y=1.2mm , baseline=(current bounding box.center)]	
	\node[ext] (A) at (0,0) {};
	\draw (0,-0.49) -- (0,-3.0);
	\draw (0.49,0.49) -- (1.9,1.9);
	\draw (-0.5,0.5) -- (-1.9,1.9);	
	\node[ext] (B) at (-2.3,2.3) {};
	\draw (-1.8,2.8) -- (0,4.9);
	\draw (-2.8,2.9) -- (-4.6,4.9);	
	\node at (2.7,2.3) {\scriptsize $#3$};
	\node at (0.4,5.3) {\scriptsize $#2$};
	\node at (-5.1,5.3) {\scriptsize $#1$};
\end{tikzpicture}
}}

\newcommand{\LLiePicReverse}[3]
{{
\begin{tikzpicture}[x=1.2mm,y=1.2mm, xscale=-1 , baseline=(current bounding box.center)]	
	\node[ext] (A) at (0,0) {};
	\draw (0,-0.49) -- (0,-3.0);
	\draw (0.49,0.49) -- (1.9,1.9);
	\draw (-0.5,0.5) -- (-1.9,1.9);	
	\node[ext] (B) at (-2.3,2.3) {};
	\draw (-1.8,2.8) -- (0,4.9);
	\draw (-2.8,2.9) -- (-4.6,4.9);	
	\node at (2.7,2.3) {\scriptsize $#1$};
	\node at (0.4,5.3) {\scriptsize $#2$};
	\node at (-5.1,5.3) {\scriptsize $#3$};
\end{tikzpicture}
}}

\newcommand{\mLiePic}[3]
{{
		\begin{tikzpicture}[x=1mm,y=1mm , baseline=(current bounding box.center)]	
			\node[int] (A) at (0,0) {};
			\draw (0,-0.49) -- (0,-3.0);
			\draw (0.49,0.49) -- (1.9,1.9);
			\draw (-0.5,0.5) -- (-1.9,1.9);	
			\node[int] (B) at (-2.3,2.3) {};
			\draw (-1.8,2.8) -- (0,4.9);
			\draw (-2.8,2.9) -- (-4.6,4.9);	
			\node at (2.7,2.3) {\scriptsize $#3$};
			\node at (0.4,5.3) {\scriptsize $#2$};
			\node at (-5.1,5.3) {\scriptsize $#1$};
		\end{tikzpicture}
}}

\newcommand{\coLLiePic}[3]
{
	\begin{tikzpicture}[x=1mm,y=1mm, yscale=-1 , baseline=(current bounding box.center)] 	
		\node[ext] (A) at (0,0) {};
		
		\draw (0,-0.49) -- (0,-3.0);
		
		\draw (0.49,0.49) -- (1.9,1.9);
		\draw (-0.5,0.5) -- (-1.9,1.9);
		
		\node[ext] (B) at (-2.3,2.3) {};
		
		\draw (-1.8,2.8) -- (0,4.9);
		\draw (-2.8,2.9) -- (-4.6,4.9);
		
		\node at (2.7,2.3) {\scriptsize $#3$};
		\node at (0.4,5.3) {\scriptsize $#2$};
		\node at (-5.1,5.3) {\scriptsize $#1$};
	\end{tikzpicture}
}

\newcommand{\coLLiePicReverse}[3]
{
	\begin{tikzpicture}[x=1mm,y=1mm,yscale=-1,xscale=-1 , baseline=(current bounding box.center)]	
		\node[ext] (A) at (0,0) {};
		
		\draw (0,-0.49) -- (0,-3.0);
		
		\draw (0.49,0.49) -- (1.9,1.9);
		\draw (-0.5,0.5) -- (-1.9,1.9);
		
		\node[ext] (B) at (-2.3,2.3) {};
		
		\draw (-1.8,2.8) -- (0,4.9);
		\draw (-2.8,2.9) -- (-4.6,4.9);
		
		\node at (2.7,2.3) {\scriptsize $#1$};
		\node at (0.4,5.3) {\scriptsize $#2$};
		\node at (-5.1,5.3) {\scriptsize $#3$};
	\end{tikzpicture}
}

\newcommand{\mcoLLiePic}[3]
{
	\begin{tikzpicture}[x=1mm,y=1mm, yscale=-1 , baseline=(current bounding box.center)]	
		\node[int](A) at (0,0) {};
		
		\draw (0,-0.49) -- (0,-3.0);
		
		\draw (0.49,0.49) -- (1.9,1.9);
		\draw (-0.5,0.5) -- (-1.9,1.9);
		
		\node[int](B) at (-2.3,2.3) {};
		
		\draw (-1.8,2.8) -- (0,4.9);
		\draw (-2.8,2.9) -- (-4.6,4.9);
		
		\node at (2.7,2.3) {\scriptsize $#3$};
		\node at (0.4,5.3) {\scriptsize $#2$};
		\node at (-5.1,5.3) {\scriptsize $#1$};
	\end{tikzpicture}
}

\newcommand{\LiecoLiePic}[4]
{
\begin{tikzpicture}[x=1mm,y=1mm , baseline=(current bounding box.center)]
	\node[ext] (A) at (0,-0.8) {};
	\draw (0,-1.3) -- (0,-3.5);                 
	\draw (0.38,-0.2) -- (2.0,2.0);             
	\draw (-0.38,-0.2) -- (-2.2,2.2);           
	\node[ext] (B) at (2.4,2.4) {};
	\draw (2.77,2.0) -- (4.4,-0.8);             
	\draw (2.4,3.0) -- (2.4,5.2);               
	\node at (0,-5.3)  {\scriptsize $#1$};
	\node at (5.1,-2.6){\scriptsize $#2$};
	\node at (-2.8,2.5){\scriptsize $#3$};
	\node at (2.4,5.7) {\scriptsize $#4$};	
\end{tikzpicture}
}

\newcommand{\LiecoLiePicT}[4]
{
\begin{tikzpicture}[x=1mm,y=1mm,xscale=-1 , baseline=(current bounding box.center)]
	\node[ext] (A) at (0,-0.8) {};
	\draw (0,-1.3) -- (0,-3.5);                 
	\draw (0.38,-0.2) -- (2.0,2.0);             
	\draw (-0.38,-0.2) -- (-2.2,2.2);           
	\node[ext] (B) at (2.4,2.4) {};
	\draw (2.77,2.0) -- (4.4,-0.8);             
	\draw (2.4,3.0) -- (2.4,5.2);               
	\node at (0,-5.3)  {\scriptsize $#1$};
	\node at (5.1,-2.6){\scriptsize $#2$};
	\node at (-2.8,2.5){\scriptsize $#3$};
	\node at (2.4,5.7) {\scriptsize $#4$};	
\end{tikzpicture}
}

\newcommand{\mLiecoLiePic}[4]
{
	\begin{tikzpicture}[x=1mm,y=1mm , baseline=(current bounding box.center)]
		
		\node[int] (A) at (0,-0.8) {};
		
		\draw (0,-1.3) -- (0,-3.5);                 
		\draw (0.38,-0.2) -- (2.0,2.0);             
		\draw (-0.38,-0.2) -- (-2.2,2.2);           
		
		\node[int] (B) at (2.4,2.4) {};
		
		\draw (2.77,2.0) -- (4.4,-0.8);             
		\draw (2.4,3.0) -- (2.4,5.2);               
		
		\node at (0,-5.3)  {\scriptsize $#1$};
		\node at (5.1,-2.6){\scriptsize $#2$};
		\node at (-2.8,2.5){\scriptsize $#3$};
		\node at (2.4,5.7) {\scriptsize $#4$};
	\end{tikzpicture}
}

\newcommand{\mVPic}[2]
{
	\begin{tikzpicture}[x=1mm,y=1mm , baseline=(current bounding box.center)]
		\node[ext] (A) at (0,-0.8) {};	
		\draw (0,-1.3) -- (0,-3.5);                 
		\draw (0.38,-0.2) -- (2.0,2.0);             
		\draw (-0.38,-0.2) -- (-2.2,2.2);           
		\node[int] (B) at (2.4,2.4) {};
		
		\draw (2.77,2.0) -- (4.4,-0.8);             
		
		\node at (0,-5.3)  {\scriptsize $#1$};
		\node at (5.1,-2.6){\scriptsize $#2$};
		\end{tikzpicture}
}
\newcommand{\mVPicT}[2]
{
	\begin{tikzpicture}[x=1mm,y=1mm, xscale = -1 , baseline=(current bounding box.center)]
		\node[ext] (A) at (0,-0.8) {};	
		\draw (0,-1.3) -- (0,-3.5);                 
		\draw (0.38,-0.2) -- (2.0,2.0);             
		\draw (-0.38,-0.2) -- (-2.2,2.2);           
		\node[int] (B) at (2.4,2.4) {};
		
		\draw (2.77,2.0) -- (4.4,-0.8);             
		
		\node at (0,-5.3)  {\scriptsize $#2$};
		\node at (5.1,-2.6){\scriptsize $#1$};
	\end{tikzpicture}
}

\newcommand{\qpgen}[4]{
	\begin{tikzpicture}[scale=0.7]
		\node[int] (v) at (0,0) {};
		\node (v0) at (-.5,.5) {\tiny{#1}};
		\node (v1) at (.5,.5) {\tiny{#2}};
		\node (u0) at (-.5,-.5) {\tiny{#3}};
		\node (u1) at (.5,-.5) {\tiny{#4}};
		\draw (u0) -- (v);
		\draw (u1) -- (v);
		\draw (v) -- (v0);
		\draw (v) -- (v1);
	\end{tikzpicture}
}

\newcommand{\qprel}[6]{
	\begin{tikzpicture}[scale=0.4]
		\node[int] (v) at (0,0) {};
		\node[int] (u) at (1,.5) {};
		\node (v1) at (-.5,1) {\tiny{#1}};
		\node (v2) at (.5,1.5) {\tiny{#2}};
		\node (v3) at (1.5,1.5) {\tiny{#3}};
		\node (u1) at (-.5,-1) {\tiny{#4}};
		\node (u2) at (.5,-1) {\tiny{#5}};
		\node (u3) at (1.5, -0.5) {\tiny{#6}};
		\draw (u1) -- (v);
		\draw (u2) -- (v);
		\draw (u3) -- (u);
		\draw (v) -- (u);
		\draw (v) -- (v1);
		\draw (u) -- (v2);
		\draw (u) -- (v3);
	\end{tikzpicture}
}

\newcommand{\RLie}[2]
{
	\begin{tikzpicture}[scale=0.1pt, yscale =-1 , baseline=(current bounding box.center)]
		\draw (0.5,0.5) -- (2.2,2.2);
		\draw (-0.48,0.48) -- (-2.2,2.2);
		\node[circle,draw,inner sep=1.5pt, fill = black] (A) at (0,0)   {};
		\node at (2.7,2.8) {$\scriptscriptstyle #2$};
		\node at (-2.7,2.8) {$\scriptscriptstyle #1$};
	\end{tikzpicture}
}	
\newcommand{\RLieRel}[3]
{
	\begin{tikzpicture}[x=1mm,y=1mm, scale=0.65, yscale =-1 , baseline=(current bounding box.center)]
		\node at (-5,7) {\scriptsize $#1$};
		\node at (-12,1) {\scriptsize $#2$};
		\node at (2,1)   {\scriptsize $#3$};
		
		\node[draw, circle, inner sep=1pt, fill = white] (L) at (-5,1) {};
		\node[fill=black,circle,inner sep=1.5pt] (C) at (-8,-5) {};
		\node[fill=black,circle,inner sep=1.5pt] (D) at (-2,-5) {};
		
		\draw (D) -- (L);
		\draw (D) -- (2,1);
		\draw (C) -- (L);
		\draw (C) -- (-12,1);
		\draw (-5,6) -- (L);
		
	\end{tikzpicture}
}

\newcommand{\strt}{{|}}
\newcommand{\dott}{
	\text{{\bf\scalebox{0.5}{\scriptsize $\vdots$}}}
}

\newcommand{\ovaltwo}[2]{%
  \coordinate (mid) at ($(#1)!0.5!(#2)$); 
  
  \coordinate (p1) at ($(mid)!0.5!90:(#2)$);
  \coordinate (p2) at ($(mid)!0.5!-90:(#2)$);

  \coordinate (ext1) at ($(mid)!1.5!(#1)$);
  \coordinate (ext2) at ($(mid)!1.5!(#2)$);

  \draw[red, smooth cycle, tension=1] 
    plot coordinates {(ext1) (p1) (ext2) (p2)};
}

\newcommand{\ovalthree}[3]{%
  \coordinate (mid) at ($(#1)!0.5!(#3)$); 
  
  \coordinate (p1) at ($(mid)!0.6!-90:(#3)$);
  \coordinate (p2) at ($(mid)!0.6!90:(#3)$);

  \coordinate (ext1) at ($(mid)!1.3!(#1)$);
  \coordinate (ext2) at ($(mid)!1.3!(#3)$);

  \draw[red, smooth cycle, tension=1] 
    plot coordinates {(ext1) (p1) (ext2) (p2)};
}

\newcommand{\LwordRLRLR}[3]{
\begin{tikzpicture}[scale=0.22, xscale=0.55]
   	\node[int] (v0) at (0,-2) {};
    \node (l0) at (0,-3.8) {};
    \draw (v0) -- (l0);
	\node[ext] (v1) at (0,0) {};
    \draw[dotted] (v0) -- (v1);    
	\node[int] (v2) at (-2,2) {};
    \draw (v1) -- (v2);
    \node (l1) at (2,2) {$\scriptstyle #3$};
    \draw[dotted] (v1) -- (l1);
    \node[ext] (v3) at (-2,4) {};
    \draw[dotted] (v2) -- (v3);
    \node (l2) at (0,6) {$\scriptstyle #2$};
    \draw[dotted] (v3) -- (l2);
    \node[int] (v4) at (-4,6) {};
    \draw (v3) -- (v4);
    \node (l3) at (-4,8.3) {$\scriptstyle #1$};
    \draw[dotted] (v4) -- (l3);
\end{tikzpicture}
}

\newcommand{\LwordLRLRR}[3]{
\begin{tikzpicture}[scale=0.22,  xscale=0.55]
   	\node[ext] (v0) at (0,0) {};
    \node (l0) at (0,-1.8) {};
    \draw (v0) -- (l0);
	\node[int] (v1) at (2,2) {};
    \draw (v0) -- (v1);    
	\node[int] (v2) at (-2,2) {};
    \draw (v0) -- (v2);
    \node (l1) at (2,4.3) {$\scriptstyle #3$};
    \draw[dotted] (v1) -- (l1);
    \node[ext] (v3) at (-2,4) {};
    \draw[dotted] (v2) -- (v3);
    \node (l2) at (0,6) {$\scriptstyle #2$};
    \draw[dotted] (v3) -- (l2);
    \node[int] (v4) at (-4,6) {};
    \draw (v3) -- (v4);
    \node (l3) at (-4,8.3) {$\scriptstyle #1$};
    \draw[dotted] (v4) -- (l3);
\end{tikzpicture}
}

\newcommand{\LwordLRRLR}[3]{
\begin{tikzpicture}[scale=0.22,  xscale=0.55]
   	\node[ext] (v0) at (0,0) {};
    \node (l0) at (0,-1.8) {};
    \draw (v0) -- (l0);
	\node[int] (v1) at (2,2) {};
    \draw (v0) -- (v1);    
	\node[int] (v2) at (-2,2) {};
    \draw (v0) -- (v2);
    \node (l1) at (-2,4.3) {$\scriptstyle #1$};
    \draw[dotted] (v2) -- (l1);
    \node[ext] (v3) at (2,4) {};
    \draw[dotted] (v1) -- (v3);
    \node (l2) at (4,6) {$\scriptstyle #3$};
    \draw[dotted] (v3) -- (l2);
    \node[int] (v4) at (0,6) {};
    \draw (v3) -- (v4);
    \node (l3) at (0,8.3) {$\scriptstyle #2$};
    \draw[dotted] (v4) -- (l3);
\end{tikzpicture}
}

\newcommand{\LwordRLLRR}[3]{
\begin{tikzpicture}[scale=0.22, xscale=0.55]
   	\node[int] (v0) at (0,-2) {};
    \node (l0) at (0,-3.8) {};
    \draw (v0) -- (l0);
	\node[ext] (v1) at (0,0) {};
    \draw[dotted] (v0) -- (v1);    
	\node[ext] (v2) at (-2,2) {};
    \draw (v1) -- (v2);
    \node (l1) at (2,2) {$\scriptstyle #3$};
    \draw[dotted] (v1) -- (l1);
    \node[int] (v3) at (0,4) {};
    \draw (v2) -- (v3);
    \node (l2) at (0,6.3) {$\scriptstyle #2$};
    \draw[dotted] (v3) -- (l2);
    \node[int] (v4) at (-4,4) {};
    \draw (v2) -- (v4);
    \node (l3) at (-4,6.3) {$\scriptstyle #1$};
    \draw[dotted] (v4) -- (l3);
\end{tikzpicture}
}

\newcommand{\LwordLLRRR}[3]{
\begin{tikzpicture}[scale=0.22, xscale=0.55]
   	\node[ext] (v0) at (0,0) {};
    \node (l0) at (0,-1.8) {};
    \draw (v0) -- (l0);
	\node[int] (v1) at (2,2) {};
    \draw (v0) -- (v1);    
	\node[ext] (v2) at (-2,2) {};
    \draw (v0) -- (v2);
    \node (l1) at (2,4.3) {$\scriptstyle #3$};
    \draw[dotted] (v1) -- (l1);
    \node[int] (v3) at (0,4) {};
    \draw (v2) -- (v3);
    \node (l2) at (0,6.3) {$\scriptstyle #2$};
    \draw[dotted] (v3) -- (l2);
    \node[int] (v4) at (-4,4) {};
    \draw (v2) -- (v4);
    \node (l3) at (-4,6.3) {$\scriptstyle #1$};
    \draw[dotted] (v4) -- (l3);
\end{tikzpicture}
}

\newcommand{\RwordLRRLR}[3]{
\begin{tikzpicture}[scale=0.22,  xscale=-0.55]
   	\node[ext] (v0) at (0,0) {};
    \node (l0) at (0,-1.8) {};
    \draw (v0) -- (l0);
	\node[int] (v1) at (2,2) {};
    \draw (v0) -- (v1);    
	\node[int] (v2) at (-2,2) {};
    \draw (v0) -- (v2);
    \node (l1) at (2,4.3) {$\scriptstyle #1$};
    \draw[dotted] (v1) -- (l1);
    \node[ext] (v3) at (-2,4) {};
    \draw[dotted] (v2) -- (v3);
    \node (l2) at (0,6) {$\scriptstyle #2$};
    \draw[dotted] (v3) -- (l2);
    \node[int] (v4) at (-4,6) {};
    \draw (v3) -- (v4);
    \node (l3) at (-4,8.3) {$\scriptstyle #3$};
    \draw[dotted] (v4) -- (l3);
\ovalthree{v2}{v3}{v4};
\end{tikzpicture}
}

\newcommand{\RwordLRLRR}[3]{
\begin{tikzpicture}[scale=0.22,  xscale=-0.55]
   	\node[ext] (v0) at (0,0) {};
    \node (l0) at (0,-1.8) {};
    \draw (v0) -- (l0);
	\node[int] (v1) at (2,2) {};
    \draw (v0) -- (v1);    
	\node[int] (v2) at (-2,2) {};
    \draw (v0) -- (v2);
    \node (l1) at (-2,4.3) {$\scriptstyle #3$};
    \draw[dotted] (v2) -- (l1);
    \node[ext] (v3) at (2,4) {};
    \draw[dotted] (v1) -- (v3);
    \node (l2) at (4,6.3) {$\scriptstyle #1$};
    \draw[dotted] (v3) -- (l2);
    \node[int] (v4) at (0,6) {};
    \draw (v3) -- (v4);
    \node (l3) at (0,8.3) {$\scriptstyle #2$};
    \draw[dotted] (v4) -- (l3);
\ovalthree{v1}{v3}{v4};
\end{tikzpicture}
}
\newcommand{\RwordRLLRR}[3]{
\begin{tikzpicture}[scale=0.22, xscale=-0.55]
   	\node[int] (v0) at (0,-2) {};
    \node (l0) at (0,-3.8) {};
    \draw (v0) -- (l0);
	\node[ext] (v1) at (0,0) {};
    \draw[dotted] (v0) -- (v1);    
	\node[ext] (v2) at (-2,2) {};
    \draw (v1) -- (v2);
    \node (l1) at (2,2) {$\scriptstyle #1$};
    \draw[dotted] (v1) -- (l1);
    \node[int] (v3) at (0,4) {};
    \draw (v2) -- (v3);
    \node (l2) at (0,6.3) {$\scriptstyle #2$};
    \draw[dotted] (v3) -- (l2);
    \node[int] (v4) at (-4,4) {};
    \draw (v2) -- (v4);
    \node (l3) at (-4,6.3) {$\scriptstyle #3$};
    \draw[dotted] (v4) -- (l3);
    \ovaltwo{v1}{v2};
\end{tikzpicture}
}

\newcommand{\RLwordLRLRR}[3]{
\begin{tikzpicture}[scale=0.22, xscale=-0.55]
   	\node[ext] (v0) at (0,0) {};
    \node (l0) at (0,-1.8) {};
    \draw (v0) -- (l0);
	\node[int] (v1) at (2,2) {};
    \draw (v0) -- (v1);    
	\node[ext] (v2) at (-2,2) {};
    \draw (v0) -- (v2);
    \node (l1) at (2,4.3) {$\scriptstyle #1$};
    \draw[dotted] (v1) -- (l1);
    \node[int] (v3) at (0,4) {};
    \draw (v2) -- (v3);
    \node (l2) at (0,6.3) {$\scriptstyle #2$};
    \draw[dotted] (v3) -- (l2);
    \node[int] (v4) at (-4,4) {};
    \draw (v2) -- (v4);
    \node (l3) at (-4,6.3) {$\scriptstyle #3$};
    \draw[dotted] (v4) -- (l3);
    \ovaltwo{v0}{v2};
\end{tikzpicture}
}

\newcommand{\LRwordRLLRR}[3]{
\begin{tikzpicture}[scale=0.22,xscale=0.55]
   	\node[int] (v0) at (0,-2) {};
	\node[ext] (v1) at (0,0) {};
	\node[ext] (v2) at (-2,2) {};
    \node[int] (v3) at (2,2) {};
    \node[int] (v4) at (-4,4) {};
    
    \node (l0) at (0,-3.8) {};
    \node (l1) at (-4,6.3) {$\scriptstyle #1$};
    \node (l2) at (0,4.3) {$\scriptstyle #2$};
    \node (l3) at (3,4.3) {$\scriptstyle #3$};

    \draw (v0) -- (l0);
    \draw[dotted] (v0) -- (v1);    
    \draw[dotted] (v1) -- (v2);
    \draw (v1) -- (v3);
    \draw[dotted] (v2) -- (l2);
    \draw (v2) -- (v4);
    \draw[dotted] (v4) -- (l1);
    \draw[dotted] (v3) -- (l3);
    \ovalthree{v0}{v1}{v3};
\end{tikzpicture}
}

\newcommand{\LRwordRLRLR}[3]{
\begin{tikzpicture}[scale=0.22,xscale=0.55]
   	\node[int] (v0) at (0,-2) {};
	\node[ext] (v1) at (0,0) {};
	\node[ext] (v2) at (-2,2) {};
    \node[int] (v3) at (2,2) {};
    \node[int] (v4) at (0,4) {};
    
    \node (l0) at (0,-3.8) {};
    \node (l1) at (-4.1,4.1) {$\scriptstyle #1$};
    \node (l2) at (0,6.3) {$\scriptstyle #2$};
    \node (l3) at (3,4.3) {$\scriptstyle #3$};

    \draw (v0) -- (l0);
    \draw[dotted] (v0) -- (v1);    
    \draw[dotted] (v1) -- (v2);
    \draw (v1) -- (v3);
    \draw[dotted] (v2) -- (l1);
    \draw (v2) -- (v4);
    \draw[dotted] (v4) -- (l2);
    \draw[dotted] (v3) -- (l3);
    \ovalthree{v0}{v1}{v3};
\end{tikzpicture}
}

\newcommand{\RLwordRLRLR}[3]{
\begin{tikzpicture}[scale=0.22,xscale=0.55]
   	\node[int] (v0) at (0,-2) {};
	\node[ext] (v1) at (0,0) {};
	\node[int] (v2) at (-2,2) {};
    \node[ext] (v3) at (-2,4) {};
    \node[int] (v4) at (0,6) {};
    
    \node (l0) at (0,-3.8) {};
    \node (l1) at (-4.1,6.1) {$\scriptstyle #1$};
    \node (l2) at (0,8.3) {$\scriptstyle #2$};
    \node (l3) at (3,2.3) {$\scriptstyle #3$};

    \draw (v0) -- (l0);
    \draw[dotted] (v0) -- (v1);    
    \draw (v1) -- (v2);
    \draw[dotted] (v1) -- (l3);
    \draw[dotted] (v2) -- (v3);
    \draw (v3) -- (v4);
    \draw[dotted] (v3) -- (l1);
    \draw[dotted] (v4) -- (l2);
\ovalthree{v2}{v3}{v4};
\end{tikzpicture}
}

\newcommand{\RLwordRLrLR}[3]
{
\begin{tikzpicture}[scale=0.22,xscale=0.55]
   	\node[int] (v0) at (0,-2) {};
	\node[ext] (v1) at (0,0) {};
	\node[int] (v2) at (2,2) {};
    \node[ext] (v3) at (2,4) {};
    \node[int] (v4) at (0,6) {};
    
    \node (l0) at (0,-3.8) {};
    \node (l3) at (4.1,6.1) {$\scriptstyle #3$};
    \node (l2) at (0,8.3) {$\scriptstyle #2$};
    \node (l1) at (-3,2.3) {$\scriptstyle #1$};

    \draw (v0) -- (l0);
    \draw[dotted] (v0) -- (v1);    
    \draw (v1) -- (v2);
    \draw[dotted] (v1) -- (l1);
    \draw[dotted] (v2) -- (v3);
    \draw (v3) -- (v4);
    \draw[dotted] (v3) -- (l3);
    \draw[dotted] (v4) -- (l2);
    \ovalthree{v0}{v1}{v2};
\end{tikzpicture}
}

\newcommand{\LRwordRLRRL}[3]
{
\begin{tikzpicture}[scale=0.22,xscale=0.55]
   	\node[int] (v0) at (0,-2) {};
	\node[ext] (v1) at (0,0) {};
	\node[int] (v2) at (-2,2) {};
    \node[ext] (v3) at (2,2) {};
    \node[int] (v4) at (0,4) {};
    
    \node (l0) at (0,-3.8) {};
    \node (l3) at (4.1,4.3) {$\scriptstyle #3$};
    \node (l2) at (0,6.3) {$\scriptstyle #2$};
    \node (l1) at (-3,4.3) {$\scriptstyle #1$};

    \draw (v0) -- (l0);
    \draw[dotted] (v0) -- (v1);    
    \draw (v1) -- (v2);
    \draw[dotted] (v1) -- (v3);
    \draw[dotted] (v2) -- (l1);
    \draw (v3) -- (v4);
    \draw[dotted] (v3) -- (l3);
    \draw[dotted] (v4) -- (l2);
    \ovaltwo{v1}{v3};
\end{tikzpicture}
}

\newcommand{\LRwordRRLLR}[3]{
\begin{tikzpicture}[scale=0.22,xscale=0.55]
   	\node[int] (v0) at (0,-2) {};
	\node[ext] (v1) at (0,0) {};
	\node[int] (v2) at (-2,2) {};
    \node[ext] (v3) at (2,2) {};
    \node[int] (v4) at (4,4) {};
    
    \node (l0) at (0,-3.8) {};
    \node (l1) at (-3,4.3) {$\scriptstyle #1$};    
    \node (l2) at (0,4.3) {$\scriptstyle #2$};
    \node (l3) at (4.3,6.3) {$\scriptstyle #3$};

    \draw (v0) -- (l0);
    \draw[dotted] (v0) -- (v1);    
    \draw (v1) -- (v2);
    \draw[dotted] (v1) -- (v3);
    \draw[dotted] (v2) -- (l1);
    \draw (v3) -- (v4);
    \draw[dotted] (v3) -- (l2);
    \draw[dotted] (v4) -- (l3);
    \ovaltwo{v1}{v3};
\end{tikzpicture}
}

\newcommand{\compos}[2]{{{}_{#1}\circ_{#2}}}

\usepackage{color}

\newcommand\cbl[1]{{\color{blue} #1}}

\newcommand\cbrown[1]{{\color{brown} #1}}

\author{Anton Khoroshkin\thanks{Department of Mathematics, University of Haifa, Mount Carmel, 3498838, Haifa, Israel}
}
\title{Input/output coloring and Gr\"obner basis for dioperads}
\date{}

\begin{document}
\maketitle

\begin{abstract}
We introduce a functor $\Psi$ that associates to a dioperad $\mathcal{P}$ acting on a vector space $V$ a two-colored operad $\Psi(\mathcal{P})$ acting on the pair $(V, V^*)$. The construction is based on a simple pictorial idea: by selecting one input or output and dualizing, if necessary, the remaining ones, any dioperadic tree can be ``rerooted’’ as a standard colored operadic tree. This transformation allows one to apply the standard operadic machinery -- such as Gr\"obner bases and Hilbert series -- to the study of dioperads.

We illustrate the method with several examples and applications. (1) We compute the dimensions of the spaces of operations for the dioperad of Lie bialgebras. (2) We describe a Gr\"obner basis and construct a minimal resolution for the dioperad of triangular Lie bialgebras. (3) We perform explicit computations for the dioperad of ``algebraic string operations''. (4) We give a pictorial construction proving the existence of quadratic Gr\"obner bases and establishing the Koszul property for a broad class of dioperads arising from cyclic operads.
\end{abstract}

\small

\tableofcontents

\normalsize

\setcounter{section}{-1}

\renewcommand{\theequation}{\Alph{equation}}

\section{Introduction}

\subsection{Motivation}
Many algebraic structures across different areas of mathematics are governed by sets of multilinear operations and the identities they satisfy. The systematic description of these multilinear operations is a central theme in universal algebra. For classical algebraic structures such as associative, Lie, and Poisson algebras, these operations typically feature multiple inputs and a single output (i.e., they belong to $\Hom(V^{\otimes n}, V)$). The corresponding universal framework governing such operations is called an \emph{operad}. However, for many other natural algebraic structures, such as Hopf algebras and Frobenius algebras, one must deal with operations featuring multiple inputs and multiple outputs (i.e., belonging to $\Hom(V^{\otimes m}, V^{\otimes n})$).

While the compositions of multilinear operations with a single output are natively encoded by rooted trees, the compositions of operations with multiple outputs correspond to graphs that may have higher genus, making their description significantly more intricate. Depending on the class of graphs allowed for these compositions, one arrives at different categorical frameworks: \emph{PROPs}~\cite{Boarman_Vogt} and \emph{properads}~\cite{Vallette_Props} correspond to (connected) directed acyclic graphs, \emph{wheeled properads}~\cite{Merkulov_Markl_Wheeled} permit directed cycles, and \emph{modular operads}~\cite{Getzler_Kapranov} are governed by ribbon graphs embedded in surfaces. Despite their theoretical importance, these broader frameworks are notoriously difficult to handle computationally. The lack of a general computational theory often necessitates ad-hoc, case-by-case analyses. 

However, for many fundamental algebraic structures---including Lie bialgebras~\cite{Markl_Voronov}, double Poisson brackets~\cite{Double::Poisson}, coboundary Lie bialgebras~\cite{Merkulov_Properad_Twisting}, and balanced infinitesimal bialgebras~\cite{BIB}---their presentations involve only graphs of genus zero. This observation strongly motivates the study of \emph{dioperads}: algebraic structures that encode multilinear operations with multiple inputs and multiple outputs, but strictly restrict their compositions to directed trees.

In particular, enumerative and homological questions (such as Koszul duality) concerning the (wheeled) properadic envelopes of dioperads should naturally be examined first at the level of the underlying dioperads. Although this task is conceptually much simpler for dioperads than for general properads, to the best of our knowledge, it has not yet been systematically addressed. The primary goal of this paper is to bridge this gap by establishing a method to reduce the combinatorics and arithmetic of dioperads to the well-understood setting of (colored) operads. Crucially, this reduction allows us to pull back powerful, well-established tools for verifying the Koszul property---such as Gr\"obner bases and Hilbert series computations---directly into the realm of dioperads.

\subsection{Main Results}

Our approach is based on a simple yet powerful pictorial insight: by designating a specific input or output as a ``global root'' and dualizing the remaining leaves, any dioperadic tree can be unambiguously ``rerooted'' into a standard operadic tree with 2-colored edges. We illustrate this transformation below for the case of corollas --examining the scenarios where either an output or an input is selected as the root -- and refer the reader to Figure~\ref{pic::diop::oper::coloring} for a more detailed treatment.

\begin{equation}
\label{eq::picture::Psi}   
\begin{array}{ccccccc}
\text{\scriptsize dioperadic } & & \text{\scriptsize 2-colored operadic} & & \text{\scriptsize dioperadic } & & \text{\scriptsize 2-colored operadic} \\
\text{\scriptsize $(m,n)$-operation} & &\text{\scriptsize  $(m|n-1)$-operation} & & \text{\scriptsize $(m,n)$-operation} & & \text{\scriptsize  $(m-1|n)$-operation}\\  
\text{\scriptsize from} & & \text{\scriptsize from}  & & \text{\scriptsize from} & & \text{\scriptsize from} \\
{\scriptstyle \Hom(V^{\otimes m}, V^{\otimes n})}  & \simeq & {\scriptstyle  \Hom(V^{\otimes m}\otimes (V^{*})^{\otimes n-1}, V)} & & 
{\scriptstyle \Hom(V^{\otimes m}, V^{\otimes n})} 
& \simeq & 
{\scriptstyle  \Hom(V^{\otimes m-1}\otimes (V^{*})^{\otimes n}, V^*) } 
\\
\begin{tikzpicture}[scale=0.4]
	\tikzstyle{inner}=[circle,draw,inner sep=1.5pt]	
	\node[inner] (A) at (1,0)   {};	
	
	\node (in1) at (-1.5,2.05) {\scriptsize $1$};
	\node (in2) at (-0.5,2.05) {\scriptsize $2$};
	\node (in3) at (0.5,2.05)  {\scriptsize \dots};
	\node (in4) at (2.5,2.05)  {\scriptsize $m$};
	
	\node[draw,circle,inner sep=1.5pt] (out1) at (-1.5,-2.45) {\scriptsize $\bar{1}$};
	\node (out2) at (-0.5,-2.45) {\scriptsize $\bar{2}$};
	\node (out3) at (1.0,-2.45)  {\scriptsize \ldots};
	\node (out4) at (3.0,-2.45) {\scriptsize $\bar{n}$};
	
	\draw[->] (in1) --  (A);
	\draw[->] (in2) --  (A);
	\draw[->] (in3) --  (A);
	\draw[->] (in4) --  (A);

    \draw[->] (A) -- (out1);
    \draw[->] (A) -- (out2);
    \draw[->] (A) -- (out3);
    \draw[->] (A) -- (out4);
\end{tikzpicture}
& \stackrel{\Psi}{\mapsto} 
&
\begin{tikzpicture}[scale=0.4]
	\tikzstyle{inner}=[circle,draw,inner sep=1.5pt]	
	\node[inner] (A) at (1,0)   {};	
	\node (in1) at (-1.5,2.05) {\scriptsize $1$};
	\node (in3) at (-0.5,2.05)  {\scriptsize \ldots};
	\node (in4) at (0.5,2.05)  {\scriptsize $m$};
	
	\node[draw,circle,inner sep=1.5pt] (out1) at (1,-2.45) {\scriptsize $\bar{1}$};
	\node (out2) at (1.5,2.05) {\scriptsize $\bar{2}$};
	\node (out3) at (2.5,2.05)  {\scriptsize \ldots};
	\node (out4) at (3.5,2.05) {\scriptsize $\bar{n}$};
	
	\draw[->] (in1) --  (A);
	\draw[->] (in3) --  (A);
	\draw[->] (in4) --  (A);

    \draw (A) -- (out1);
    \draw[dotted] (A) -- (out2);
    \draw[dotted] (A) -- (out3);
    \draw[dotted] (A) -- (out4);
\end{tikzpicture}
& ; &
\begin{tikzpicture}[scale=0.4]
	\tikzstyle{inner}=[circle,draw,inner sep=1.5pt]	
	\node[inner] (A) at (1,0)   {};	
	
	\node[draw,circle,inner sep=1.5pt] (in1) at (-1.5,2.05) {${\scriptstyle 1}$};
	\node (in2) at (-0.5,2.05) {\scriptsize $2$};
	\node (in3) at (0.5,2.05)  {\scriptsize \dots};
	\node (in4) at (2.5,2.05)  {\scriptsize $m$};
	
	\node (out1) at (-1.5,-2.45) {\scriptsize $\bar{1}$};
	\node (out2) at (-0.5,-2.45) {\scriptsize $\bar{2}$};
	\node (out3) at (1.0,-2.45)  {\scriptsize \ldots};
	\node (out4) at (3.0,-2.45) {\scriptsize $\bar{n}$};
	
	\draw[->] (in1) --  (A);
	\draw[->] (in2) --  (A);
	\draw[->] (in3) --  (A);
	\draw[->] (in4) --  (A);

    \draw[->] (A) -- (out1);
    \draw[->] (A) -- (out2);
    \draw[->] (A) -- (out3);
    \draw[->] (A) -- (out4);
\end{tikzpicture}
& \stackrel{\Psi}{\mapsto} 
&
\begin{tikzpicture}[scale=0.4]
	\tikzstyle{inner}=[circle,draw,inner sep=1.5pt]	
	\node[inner] (A) at (1,0)   {};	
	\node (in1) at (-1.5,2.05) {\scriptsize $2$};
	\node (in3) at (-0.5,2.05)  {\scriptsize \ldots};
	\node (in4) at (0.5,2.05)  {\scriptsize $m$};
	
	\node[draw,circle,inner sep=1.5pt] (out1) at (1,-2.45) {\scriptsize ${1}$};
	\node (out2) at (1.5,2.05) {\scriptsize $\bar{1}$};
	\node (out3) at (2.5,2.05)  {\scriptsize \ldots};
	\node (out4) at (3.5,2.05) {\scriptsize $\bar{n}$};
	
	\draw (in1) --  (A);
	\draw (in3) --  (A);
	\draw (in4) --  (A);
    \draw[dotted] (A) -- (out1);
    \draw[dotted] (A) -- (out2);
    \draw[dotted] (A) -- (out3);
    \draw[dotted] (A) -- (out4);
\end{tikzpicture}
\end{array}
\end{equation}
Our first main result can be summarized as follows:
\begin{theorem}[see Theorem~\ref{thm::diop->op}]
The pictorial map~\eqref{eq::picture::Psi} defines a faithful, exact functor
$$\Psi : \textbf{Dioperads} \longrightarrow \textbf{2-colored operads},$$
such that for a dioperad $\calP$ acting on $V$ there is a natural $\Psi(\calP)$ action on the pair $(V,V^{*})$.
Moreover, $\Psi$ maps free dioperads to free $2$-colored operads and commutes with the bar and cobar constructions.

Consequently, a quadratic dioperad $\calP$ is Koszul if and only if its image $\Psi(\calP)$ is a Koszul $2$-colored operad.
\end{theorem}

By exploiting the functor $\Psi$, we derive a functional equation relating the generating series of any Koszul dioperad to that of its Koszul dual:
\begin{theorem}[see Corollary~\ref{cor::sereis::diop::Koszul}]
\label{thm::series::intro}
If the dioperad $\cP$ is Koszul, then the graded generating series
$$
\rchi_{\cP}(x,y;q) := \sum_{m,n} \left(\sum_{r=0}^{\infty} q^{r}\dim \cP(m,n)^{(r)}\right) \frac{x^m}{m!}\frac{y^n}{n!}
= \sum_{m,n} \dim_q \cP(m,n) \frac{x^m}{m!}\frac{y^n}{n!}.
$$
of $\cP$ and the corresponding generating series $\rchi_{\cP^{!}}(x,y;q)$ of its Koszul dual $\cP^{!}$ satisfy the following functional identity:
\begin{equation*}
\left(\frac{\partial \rchi_{\cP^{!}}(x,y;-q)}{\partial y}; \frac{\partial \rchi_{\cP^{!}}(x,y;-q)}{\partial x}\right) \circ    
\left(\frac{\partial \rchi_{\cP}(x,y;q)}{\partial y}; \frac{\partial \rchi_{\cP}(x,y;q)}{\partial x}\right) = (x;y).
\end{equation*}
Here, $f\circ g$ denotes the composition of power series maps with respect to the formal variables $x$ and $y$.   
\end{theorem}

Applying this to the dioperad of Lie bialgebras ($\mathcal{L}ieb$), we obtain a striking closed-form formula for the dimensions of its spaces of operations.
\begin{theorem}[see Proposition~\ref{prp::Liebi::genseries}]
The dimension of the space of $(m,n)$-ary operations in the dioperad of Lie bialgebras is given by:
\[
\dim \mathcal{L}ieb(m,n) = \frac{(m+n-2)!^2}{(m-1)!(n-1)!}.
\]
\end{theorem}
Furthermore, Theorem~\ref{thm::series::intro} enables us to disprove the Koszul property for the dioperad $\mathcal{W}^{(d)}$ introduced in~\cite{Tradler}, resolving an open question regarding its Koszulness.

The core \textbf{theoretical contribution} of this paper is the adaptation of the (colored) shuffle operad \textbf{Gr\"obner basis} theory to the setting \textbf{of dioperads} (Section~\ref{sec::Shuffle::Grobner}).
We illustrate this machinery by providing explicit computations for several foundational dioperads, including those governing Frobenius algebras, Lie bialgebras, quadratic Poisson structures, and algebraic string operations. Moreover, we resolve a conjecture regarding the dioperad of triangular Lie bialgebras ($\mathcal{L}ieb^{\triangle}$).
\begin{theorem*}
\begin{itemize}[itemsep=0pt,topsep=0pt]
\item  {\rm(see Theorem~\ref{thm::triangular::Grobner})} 
The spanning (quadratic and cubic) Relations~\eqref{eq::triang::LieBi} for the dioperad of triangular Lie bialgebras form a Gr\"obner basis. 
\item {\rm(see Theorem~\ref{thm::Triang::Lieb::Anick}) }
The Anick-type resolution constructed from this basis is minimal, confirming the conjectural resolution proposed by Merkulov in~\cite{Merkulov_Properad_Twisting}.
\end{itemize}
\end{theorem*}

We also establish a systematic way to generate Koszul dioperads from cyclic operads. By partitioning the legs of a cyclic operad into inputs and outputs following the given rule $\mathfrak{c}$, we define a pictorial functor $\Theta_{\mathfrak{c}}$ that maps cyclic operads to dioperads. 

\begin{theorem}[see Corollary~\ref{cor::coloring::Koszul}]
\label{thm::cycl->diop::intro}
If a cyclic operad $\mathcal{P}$ admits a quadratic Gr\"obner basis with caterpillar normal forms,\footnote{A caterpillar operadic monomial is represented by a tree that has no vertices with nontrivial branches growing in different directions. Right and left combs are particular examples of caterpillar monomials. See Definition~\ref{def::caterpilar} for details.} 
then the induced dioperad $\Theta_{\mathfrak{c}}(\mathcal{P})$ also admits caterpillar normal forms and a quadratic convergent rewriting system.\footnote{A convergent rewriting system is a slightly weaker notion than a Gr\"obner basis, yet it retains the same powerful properties for working with monomials. See \S~\ref{sec::rewriting::systems} for details.} 
\end{theorem}
In particular, Theorem~\ref{thm::cycl->diop::intro} provides a purely combinatorial proof of the Koszul property for a broad family of dioperads.

\subsection{Structure of the paper}
Following a brief survey of the theory of dioperads in Section \ref{sec::dioperads}, we introduce the pictorial functor $\Psi$ from dioperads to colored operads in the subsequent Section~\ref{sec::Psi}. Its applications to generating series and the Koszul property are discussed in 
Section~\ref{sec::Hilbert}.
The theory of shuffle operads and Gr\"obner bases, adapted to the framework of the 2-colored operads arising from $\Psi$ is recalled in Section~\ref{sec::Shuffle::Grobner}. Finally, in Section~\ref{sec::Anick} we outline the homological applications of Gr\"obner bases.

Section~\ref{sec::coloring::cyclic} is devoted to the functor $\Theta$, which constructs a dioperad from a cyclic operad. Theorem~\ref{thm::cycl->diop} establishes that a rewriting system for a dioperad can be derived from a Gr\"obner basis of the underlying cyclic operad under specific assumptions.

Finally, Section~\ref{sec::examples} provides various illustrative examples. Notable applications include computing the dimensions of operations for Lie bialgebras (\S\ref{example::Lie::Bialg}), as well as confluence computations and the construction of Anick resolutions for triangular Lie bialgebras (\S\ref{example::Lie::Triang}) and the dioperad of Tradler and Zeinalian (\S\ref{example::Tradler}).

\renewcommand{\theequation}{\thesubsection.\arabic{equation}}

\subsubsection*{Acknowledgments}
This project grew from a simple idea that originally spanned only a few pages and remained an unpublished draft for several years. I am grateful to several people whose interest motivated me to finally complete it. First, I would like to thank Sergei Merkulov for posing the question regarding the resolution of the dioperad of triangular Lie bialgebras. Second, I am grateful to Pedro Tamaroff for his interest in disproving the Koszul property for the operad $\mathcal{W}^{(d)}$ introduced by K.\,Poirier and T.\,Tradler in~\cite{Tradler}. Third, I thank my former student Vladislav Kharitonov, who promised to assist with computer programs for searching for quadratic Gr\"obner bases for colored operads. Finally, I am grateful to Vladimir Dotsenko for his valuable feedback on an earlier version of this manuscript.

\section{Dioperads -- recollection}
\label{sec::dioperads}

We assume the reader is familiar with the notion of an algebraic operad and its corresponding Koszul duality theory (see, e.g.,~\cite{Ginzburg::Kapranov, Loday::Vallette} for comprehensive details). Recall that a standard operad consists of a collection of operations featuring multiple inputs and a single output. This paper, however, is devoted to developing computational methods for algebraic structures of an operadic nature that possess multiple inputs and multiple outputs.

When working with such structures, an element of $\calP(m,n)$ is typically depicted as a corolla with $m$ incoming edges and $n$ outgoing edges. While more complex categorical frameworks (such as PROPs, properads, wheeled properads, and modular operads) exist to capture compositions along arbitrary graphs or surfaces, dioperads govern the simplest, most fundamental subset of these compositions: those restricted exclusively to directed, genus-zero trees. In particular, the category of dioperads sits inside all of the aforementioned mathematical universes. 

One of the main goals of this paper is to establish an "arithmetic" of monomials that enables systematic answers to basic enumerative and homological questions specifically for dioperads.
Let us recall the formal definition of a \emph{dioperad} (we refer the reader to~\cite{Gan::dioperads} for further details):
\begin{definition}
	A dioperad $\calP$ consists of:
	\begin{itemize}[itemsep=0pt,topsep=0pt]
		\item a collection of operations $\calP(m,n)$ for $m$ inputs and $n$ outputs, equipped with a right action of $\bS_m$ on the inputs and a left action of $\bS_n$ on the outputs.
		\item infinitesimal composition operations: $\compos{i}{j} : \calP(m,n) \otimes \calP(m',n') \rightarrow \calP(m+m'-1,n+n'-1)$ that insert the $i$-th output of $\calP(m,n)$ into the $j$-th input of $\calP(m',n')$.
	\end{itemize}
These composition operations must be \emph{associative} and equivariant with respect to the $\bS \times \bS^{\op}$ actions. 
\end{definition}

To formulate iterated infinitesimal compositions, one assigns a global composition map $\circ_{T}$ to any \emph{dioperadic tree} $T$:
$$\circ_{T} : \bigotimes_{v\in V(T)} \calP(\cin(v),\cout(v)) \rightarrow \calP(\cin(T),\cout(T)).$$
Here, a dioperadic tree $T$ is a directed tree equipped with a global labeling/numbering of its external incoming edges $\cin(T)$ and its external outgoing edges $\cout(T)$. The associativity constraint dictates that the result of an iterated composition does not depend on the order in which the internal edges of the dioperadic tree are contracted.

For example, given the tree:
$$
T:=
\begin{tikzpicture}[scale=0.5, baseline=-0.5ex]
	\tikzstyle{inner}=[circle,draw,inner sep=1.5pt]
	
	\node[inner] (A) at (0,0)   {};
	\node[inner] (B) at (3,0)   {};
	\node[inner] (C) at (1.5,-1.8) {};
	
	\node (in1) at (-1.5,1.8) {};
	\node (in2) at (-0.5,1.8) {};
	\node (in3) at (0.5,1.8) {};
	\node (in4) at (2.5,1.8) {};
	\node (in5) at (3.5,1.8) {};
	
	\node (out1) at (-1.5,-3.4) {};
	\node (out2) at (-0.5,-3.4) {};
	\node (out4) at (1.0,-3.4)  {};
	\node (out5) at (2.0,-3.4)  {};
	\node (out6) at (3.0,-3.4)  {};
	\node (out3) at (4.0,-3.4)  {};
	
	\draw[->] (in1) --  (A);
	\draw[->] (in2) --  (A);
	\draw[->] (in3) --  (A);
	\draw[->] (in4) --  (B);
	\draw[->] (in5) --  (B);
	
	\draw[->] (A) -- (C);
	\draw[->] (B) -- (C);
	
	\draw[->] (A) -- (out1);
	\draw[->] (A) -- (out2);
	
	\draw[->] (B) --  (out3);
	
	\draw[->] (C) --  (out4);
	\draw[->] (C) --  (out5);
	\draw[->] (C) --  (out6);
\end{tikzpicture}
\quad \substack{\text{ we have the }\\ \text{composition map }} \circ_{T}: \calP(3,3)\otimes \calP(2,2)\otimes \calP(2,3) \rightarrow \calP(5,6).
$$

It is worth noting that the free dioperad $\Free(\Upsilon)$ generated by an $\bS_m\times\bS_n^{\op}$-collection $\Upsilon(m,n)$ is linearly spanned by all dioperadic trees $T$ whose vertices $v$ are decorated by elements of $\Upsilon(\cin(v),\cout(v))$.

\begin{definition}
\label{def::reps::dioperad}
	A structure of a $\calP$-algebra (also called a representation of a dioperad $\calP$) on a vector space $V$ is a morphism of dioperads $\rho:\calP\to \End_V$, where $\End_{V}$ is the endomorphism dioperad defined by $\End_V(m,n):=\Hom(V^{\otimes m},V^{\otimes n})$ equipped with the natural composition rules.
\end{definition}

As with operads, one can define the bar and cobar constructions for dioperads. The bar construction functor $\cB$ maps a dioperad $\cP$ to the free dg-co-dioperad generated by the suspended space $s\cP$, while the cobar construction functor $\Omega$ associates to a co-dioperad $\calQ$ a quasi-free dg-dioperad generated by $s^{-1}\calQ$. For any dioperad $\cP$ and codioperad $\calQ$, there are canonical quasi-isomorphism of (co)dioperads: 
$$\Omega(\cB(\cP))\rightarrow \cP, \quad \calQ\rightarrow \cB(\Omega(\calQ)).$$ 

However, this canonical free resolution is generally far from minimal. For quadratic dioperads, one extends the standard Koszul duality theory to construct minimal resolutions for "good" (i.e., Koszul) quadratic presentations. Koszul duality theory, which is classically known for associative algebras~\cite{Priddy, PP_Koszul}, was generalized to quadratic operads by Ginzburg and Kapranov~\cite{Ginzburg::Kapranov}, and has been subsequently extended to various operadic-type structures. In this paper, we are primarily interested in the extensions to colored operads~\cite{Laan} and dioperads~\cite{Gan::dioperads}. Our notation aligns closely with the standard operadic literature (e.g.,~\cite{Loday::Vallette}), and we adopt the following equivalent definition of Koszulness.

\begin{definition}
	A quadratic dioperad $\cP := \Free(\Upsilon)/(R)$ generated by $\bS\times\bS^{\op}$-collection $\Upsilon$ and subject to quadratic relations $R$ is called \emph{Koszul} if the canonical projection from the cobar construction of its quadratic dual co-dioperad $\cP^{\antish} := \coFree(s\Upsilon|sR^{\perp})$,
	\[
	\Omega_{\diop}(\cP^{\antish}) \stackrel{\quis}{\longrightarrow} \cP,
	\]
	is a quasi-isomorphism.	
\end{definition}

\section{From {\it dioperads} to {\it colored (shuffle) operads}}

This section details the central pictorial intuition that allows us to derive a theory of Gr\"obner bases for dioperads from the established machinery for colored operads. We omit a formal survey of colored operads here and instead refer the reader to~\cite{vanderLaan}, as well as our preceding paper~\cite{colored_Grobner} and the references therein, for a comprehensive treatment of Gr\"obner bases in the colored setting.

\subsection{The pictorial map $\Psi$}
\label{sec::Psi}
Suppose $T$ is a dioperadic tree. All edges of $T$ are directed; we refer to this as the \emph{inner} orientation.
Choose a leaf $r$ of $T$ and declare it the \emph{global root}. This choice defines a unique path from each vertex to $r$. We then impose a new orientation on the edges such that $r$ becomes the unique sink, which we call the \emph{outer} orientation.

The outer orientation transforms $T$ into an operadic tree $T_r$, where each vertex has exactly one outgoing edge, and all other incident edges are incoming.
This resulting operadic tree carries a two-coloring of its edges: an edge is drawn as a \emph{straight line} if its outer and inner orientations \emph{coincide}, and as a \emph{dotted line} if they are \emph{opposite}.
Note that the outer orientation is entirely determined by the choice of the root, and the inner orientation can be recovered from the coloring.

Figure~\eqref{pic::diop::oper::coloring} below illustrates the procedure for converting a dioperadic tree into a two-colored operadic tree. The first example shows the transformation with outgoing edge $\bar{6}$ as the root; the second uses incoming edge $5$ as the root.
\begin{equation}
\label{pic::diop::oper::coloring}
\begin{array}{ccccc}
\substack{\text{ a dioperadic tree } \\ \text{ with a chosen root}} & & 
\text{ edge coloring } & & 
\substack{\text{ the corresponding } \\ 
\text{colored operadic tree}} \\
\begin{tikzpicture}[scale=0.5]
	\tikzstyle{inner}=[circle,draw,inner sep=1.5pt]
	
	\node[inner] (A) at (0,0)   {};
	\node[inner] (B) at (3,0)   {};
	\node[inner] (C) at (1.5,-1.8) {};
	
	\node (in1) at (-1.5,1.8) {};
	\node (in2) at (-0.5,1.8) {};
	\node (in3) at (0.5,1.8) {};
	\node (in4) at (2.5,1.8) {};
	\node (in5) at (3.5,1.8) {};
	
	\node at (-1.5,2.05) {\scriptsize $1$};
	\node at (-0.5,2.05) {\scriptsize $2$};
	\node at (0.5,2.05)  {\scriptsize $3$};
	\node at (2.5,2.05)  {\scriptsize $4$};
	\node at (3.5,2.05)  {\scriptsize $5$};
	
	\node (out1) at (-1.5,-3.4) {};
	\node (out2) at (-0.5,-3.4) {};
	\node (out4) at (1.0,-3.4)  {};
	\node (out5) at (2.0,-3.4)  {};
	\node (out6) at (3.0,-3.4)  {};
	\node (out3) at (4.0,-3.4)  {};
	
	\node at (-1.5,-3.65) {\scriptsize $\bar{1}$};
	\node at (-0.5,-3.65) {\scriptsize $\bar{2}$};
	\node at (4.0,-3.65)  {\scriptsize $\bar{3}$};
	\node at (1.0,-3.65)  {\scriptsize $\bar{4}$};
	\node at (2.0,-3.65)  {\scriptsize $\bar{5}$};
	
	\node[draw,circle,inner sep=1.5pt] at (3.0,-3.65) {\scriptsize $\bar{6}$};
	
	\draw[->] (in1) --  (A);
	\draw[->] (in2) --  (A);
	\draw[->] (in3) --  (A);
	\draw[->] (in4) --  (B);
	\draw[->] (in5) --  (B);
	
	\draw[->] (A) -- (C);
	\draw[->] (B) -- (C);
	
	\draw[->] (A) -- (out1);
	\draw[->] (A) -- (out2);
	\draw[->] (B) -- (out3);
	\draw[->] (C) -- (out4);
	\draw[->] (C) -- (out5);
	\draw[->] (C) -- (out6);
\end{tikzpicture}
&
\mapsto 
&
\begin{tikzpicture}[scale=0.5]
	\tikzstyle{inner}=[circle,draw,inner sep=1.5pt]
	
	\node[inner] (A) at (0,0)   {};
	\node[inner] (B) at (3,0)   {};
	\node[inner] (C) at (1.5,-1.8) {};
	
	\node (in1) at (-1.5,1.8) {};
	\node (in2) at (-0.5,1.8) {};
	\node (in3) at (0.5,1.8) {};
	\node (in4) at (2.5,1.8) {};
	\node (in5) at (3.5,1.8) {};
	
	\node at (-1.5,2.05) {\scriptsize $1$};
	\node at (-0.5,2.05) {\scriptsize $2$};
	\node at (0.5,2.05)  {\scriptsize $3$};
	\node at (2.5,2.05)  {\scriptsize $4$};
	\node at (3.5,2.05)  {\scriptsize $5$};
	
	\node (out1) at (-1.5,-3.4) {};
	\node (out2) at (-0.5,-3.4) {};
	\node (out4) at (1.0,-3.4)  {};
	\node (out5) at (2.0,-3.4)  {};
	\node (out6) at (3.0,-3.4)  {};
	\node (out3) at (4.0,-3.4)  {};
	
	\node at (-1.5,-3.65) {\scriptsize $\bar{1}$};
	\node at (-0.5,-3.65) {\scriptsize $\bar{2}$};
	\node at (4.0,-3.65)  {\scriptsize $\bar{3}$};
	\node at (1.0,-3.65)  {\scriptsize $\bar{4}$};
	\node at (2.0,-3.65)  {\scriptsize $\bar{5}$};
	\node[draw,circle,inner sep=1.5pt] at (3.0,-3.65) {\scriptsize $\bar{6}$};
	
	\draw[->] (in1) --  (A);
	\draw[->] (in2) --  (A);
	\draw[->] (in3) --  (A);
	\draw[->] (in4) --  (B);
	\draw[->] (in5) --  (B);
	
	\draw[->] (A) -- (C);
	\draw[->] (B) -- (C);
	\draw[->] (C) -- (out6); 
	
	\draw[->,dotted] (A) -- (out1);
	\draw[->,dotted] (A) -- (out2);
	\draw[->,dotted] (B) -- (out3);
	\draw[->,dotted] (C) -- (out4);
	\draw[->,dotted] (C) -- (out5);
\end{tikzpicture}
&
=
&
\begin{tikzpicture}[scale=0.5,>=stealth]
	\tikzstyle{inner}=[circle,draw,inner sep=1.5pt]
	
	\node[inner] (A) at (0,0)   {};   
	\node[inner] (B) at (4,0)   {};   
	
	\node[inner] (C) at (2,-2)  {};   
	
	\node (in1) at (-1.5,1.5) {};
	\node (in2) at (-0.5,1.5) {};
	\node (in3) at ( 0.5,1.5) {};
	
	\node at (-1.5,1.8) {\scriptsize $\bar{2}$};
	\node at (-0.5,1.8) {\scriptsize $1$};
	\node at ( 0.5,1.8) {\scriptsize $2$};
	
	\node (bin1) at (-2.5,1.5) {};
	\node (bin2) at ( 1.5,1.5) {};
	
	\node at (-2.5,1.8) {\scriptsize $\bar{1}$};
	\node at ( 1.5,1.8) {\scriptsize $3$};
	
	\draw[-]            (bin2)  -- (A);
	\draw[-]            (in2)  -- (A);
	\draw[-]            (in3)  -- (A);
	\draw[-,dotted]    (bin1) -- (A);
	\draw[-,dotted]    (in1) -- (A);
	
	\node (rin1) at (3.0,1.5) {};
	\node (bin3) at (5.0,1.5) {};
	
	\node at (3.0,1.8) {\scriptsize $4$};
	\node at (5.0,1.8) {\scriptsize $\bar{3}$};
	
	\node (rin2) at (4.0,1.5) {};
	\node at (4.0,1.8) {\scriptsize $5$};
	
	\draw[-]         (rin1) -- (B);
	\draw[-]         (rin2) -- (B);
	\draw[-,dotted] (bin3) -- (B);
	
	\node (cin1) at (1.5,-0.5) {};
	\node (cin2) at (2.5,-0.5) {};
	
	\node at (1.5,-0.2) {\scriptsize $\bar{4}$};
	\node at (2.5,-0.2) {\scriptsize $\bar{5}$};
	
	\draw[-,dotted] (cin1) -- (C);
	\draw[-,dotted] (cin2) -- (C);
	
	\draw[-] (A) -- (C);
	\draw[-] (B) -- (C);
	
	\node (outroot) at (2,-3.5) {};
	\draw[-] (C) -- (outroot);
	
\end{tikzpicture}
\\
& & & & \\
\begin{tikzpicture}[scale=0.5]
	\tikzstyle{inner}=[circle,draw,inner sep=1.5pt]
	
	\node[inner] (A) at (0,0)   {};
	\node[inner] (B) at (3,0)   {};
	\node[inner] (C) at (1.5,-1.8) {};
	
	\node (in1) at (-1.5,1.8) {};
	\node (in2) at (-0.5,1.8) {};
	\node (in3) at (0.5,1.8) {};
	\node (in4) at (2.5,1.8) {};
	\node (in5) at (3.5,1.8) {};
	
	\node at (-1.5,2.05) {\scriptsize $1$};
	\node at (-0.5,2.05) {\scriptsize $2$};
	\node at (0.5,2.05)  {\scriptsize $3$};
	\node at (2.5,2.05)  {\scriptsize $4$};
	\node[draw,circle,inner sep=1.5pt] at (3.5,2.05)  {\scriptsize $5$};
	
	\node (out1) at (-1.5,-3.4) {};
	\node (out2) at (-0.5,-3.4) {};
	\node (out4) at (1.0,-3.4)  {};
	\node (out5) at (2.0,-3.4)  {};
	\node (out6) at (3.0,-3.4)  {};
	\node (out3) at (4.0,-3.4)  {};
	
	\node at (-1.5,-3.65) {\scriptsize $\bar{1}$};
	\node at (-0.5,-3.65) {\scriptsize $\bar{2}$};
	\node at (4.0,-3.65)  {\scriptsize $\bar{3}$};
	\node at (1.0,-3.65)  {\scriptsize $\bar{4}$};
	\node at (2.0,-3.65)  {\scriptsize $\bar{5}$};
	
	\node at (3.0,-3.65) {\scriptsize $\bar{6}$};
	
	\draw[->] (in1) --  (A);
	\draw[->] (in2) --  (A);
	\draw[->] (in3) --  (A);
	\draw[->] (in4) --  (B);
	\draw[->] (in5) --  (B);
	
	\draw[->] (A) -- (C);
	\draw[->] (B) -- (C);
	
	\draw[->] (A) -- (out1);
	\draw[->] (A) -- (out2);
	\draw[->] (B) -- (out3);
	\draw[->] (C) -- (out4);
	\draw[->] (C) -- (out5);
	\draw[->] (C) -- (out6);
\end{tikzpicture}
&
\mapsto 
&
\begin{tikzpicture}[scale=0.5]
	\tikzstyle{inner}=[circle,draw,inner sep=1.5pt]
	
	\node[inner] (A) at (0,0)   {};
	\node[inner] (B) at (3,0)   {};
	\node[inner] (C) at (1.5,-1.8) {};
	
	\node (in1) at (-1.5,1.8) {};
	\node (in2) at (-0.5,1.8) {};
	\node (in3) at (0.5,1.8) {};
	\node (in4) at (2.5,1.8) {};
	\node (in5) at (3.5,1.8) {};
	
	\node at (-1.5,2.05) {\scriptsize $1$};
	\node at (-0.5,2.05) {\scriptsize $2$};
	\node at (0.5,2.05)  {\scriptsize $3$};
	\node at (2.5,2.05)  {\scriptsize $4$};
	\node[draw,circle,inner sep=1.5pt] at (3.5,2.05)  {\scriptsize $5$};
	
	\node (out1) at (-1.5,-3.4) {};
	\node (out2) at (-0.5,-3.4) {};
	\node (out4) at (1.0,-3.4)  {};
	\node (out5) at (2.0,-3.4)  {};
	\node (out6) at (3.0,-3.4)  {};
	\node (out3) at (4.0,-3.4)  {};
	
	\node at (-1.5,-3.65) {\scriptsize $\bar{1}$};
	\node at (-0.5,-3.65) {\scriptsize $\bar{2}$};
	\node at (4.0,-3.65)  {\scriptsize $\bar{3}$};
	\node at (1.0,-3.65)  {\scriptsize $\bar{4}$};
	\node at (2.0,-3.65)  {\scriptsize $\bar{5}$};
	\node at (3.0,-3.65) {\scriptsize $\bar{6}$};
	
	\draw[->] (in1) --  (A);
	\draw[->] (in2) --  (A);
	\draw[->] (in3) --  (A);
	\draw[->] (in4) --  (B);
	\draw[->] (in5) --  (B);
	
	\draw[->] (A) -- (C);
	
	\draw[->,dotted] (A) -- (out1);
	\draw[->,dotted] (A) -- (out2);
	\draw[->,dotted] (B) -- (out3);
    \draw[->,dotted] (B) -- (C);
	\draw[->,dotted] (C) -- (out4);
	\draw[->,dotted] (C) -- (out5);
    \draw[->,dotted] (C) -- (out6);
\end{tikzpicture}
&
=
&
\begin{tikzpicture}[scale=0.4,>=stealth]
	\tikzstyle{inner}=[circle,draw,inner sep=1.5pt]	
	\node[inner] (A) at (0,0)   {};   
	\node[inner] (B) at (4,-3.5)   {};   
	
	\node[inner] (C) at (2,-2)  {};   
	
	\node (in1) at (-1.5,1.5) {};
	\node (in2) at (-0.5,1.5) {};
	\node (in3) at ( 0.5,1.5) {};
	
	\node at (-1.5,1.8) {\scriptsize $\bar{2}$};
	\node at (-0.5,1.8) {\scriptsize $1$};
	\node at ( 0.5,1.8) {\scriptsize $2$};
	\node (bin1) at (-2.5,1.5) {};
	\node (bin2) at ( 1.5,1.5) {};
	\node at (-2.5,1.8) {\scriptsize $\bar{1}$};
	\node at ( 1.5,1.8) {\scriptsize $3$};
	\draw[-]            (bin2)  -- (A);
	\draw[-]            (in2)  -- (A);
	\draw[-]            (in3)  -- (A);
	\draw[-,dotted]    (bin1) -- (A);
	\draw[-,dotted]    (in1) -- (A);
	\node (rin1) at (3.5,-.5) {};
	\node at (3.5,-.2) {\scriptsize $\bar{6}$};
	\node at (5.5,-1.3) {\scriptsize $\bar{3}$};
	\node (bin3) at (5.5,-1.5) {};
	\node (rin2) at (4.5,-1.5) {};
	\node at (4.5,-1.3) {\scriptsize $4$};
	\draw[-,dotted]         (rin1) -- (C);
	\draw[-]         (rin2) -- (B);
	\draw[-,dotted] (bin3) -- (B);
	
	\node (cin1) at (1.5,-0.5) {};
	\node (cin2) at (2.5,-0.5) {};
	
	\node at (1.5,-0.2) {\scriptsize $\bar{4}$};
	\node at (2.5,-0.2) {\scriptsize $\bar{5}$};
	
	\draw[-,dotted] (cin1) -- (C);
	\draw[-,dotted] (cin2) -- (C);
	
	\draw[-] (A) -- (C);
	\draw[-,dotted] (B) -- (C);
	
	\node (outroot) at (4,-5.5) {};
	\draw[-,dotted] (B) -- (outroot);
\end{tikzpicture}
\end{array}
\end{equation}
The result depends strongly on which input or output is selected as the root. Specifically, to each dioperadic corolla with $m$ inputs and $n$ outputs, we associate two types of operadic corollas: one with $m$ straight inputs and $n-1$ dotted inputs, and another with $m-1$ straight inputs and $n$ dotted inputs. The operadic composition of these colored trees corresponds exactly to the dioperadic composition in the original dioperad, leading to the following observation:

\begin{proposition}
\label{prp::Psi:diop->op}
	The collection of all possible root choices in a dioperad defines a faithful exact functor
	\[
	\Psi: \text{Dioperads} \longrightarrow \text{$2$-colored Operads}.
	\]
The space of operations $\Psi(\cP)^{\strt}(m,n-1)$ with a straight output, $m$ straight inputs, and $n-1$ dotted inputs, and the space $\Psi(\cP)^{\dott}(m-1,n)$ with a dotted output, $m-1$ straight inputs, and $n$ dotted inputs, are both canonically isomorphic to the space $\cP(m,n)$ of operations in the underlying dioperad $\cP$:
	\[
	\Psi(\cP)^{\strt}(m,n-1)
	=
	\Psi(\cP)^{\dott}(m-1,n)
	:=
	\cP(m,n).
	\]
Moreover, operadic composition in the colored operad agrees with composition in the dioperad.
\end{proposition}

\begin{proof}
	This follows directly from the pictorial description.
\end{proof}

Recall from Definition~\ref{def::reps::dioperad} that a representation $(V,\rho)$ of a dioperad $\calP$ assigns to each element $\gamma\in \calP(m,n)$ a multilinear operation $\rho(\gamma): V^{\otimes m} \to V^{\otimes n}$. By dualizing either an output or an input, we can reinterpret these maps as operations ending in $V$ or $V^*$, respectively. This observation allows us to relate dioperad representations to those of the associated colored operad:

\begin{proposition}
Any finite-dimensional representation $(V,\rho)$ of a dioperad $\calP$ induces a representation $\bar{\rho}$ of the colored operad $\Psi(\calP)$ on the pair of vector spaces $(V, V^*)$, where $V$ corresponds to the "straight" color and the dual space $V^*$ to the "dotted" one.
\end{proposition}

\begin{proof}
By Definition~\ref{def::reps::dioperad}, the fact that $V$ is a representation of $\calP$ is equivalent to the existence of a dioperad morphism $\rho:\calP \to \End_{V}$. 
Each $(m,n)$-ary operation $\gamma \in \calP(m,n)$ defines a multilinear map $\rho(\gamma) \in \Hom(V^{\otimes m}, V^{\otimes n})$. Since $V$ is finite-dimensional, we have the following canonical isomorphisms:
$$
\begin{array}{ccccc}
\Hom(V^{\otimes m}, V^{\otimes n})  & \simeq & \Hom(V^{\otimes m}\otimes (V^{*})^{\otimes n-1}, V) & \simeq &
\Hom(V^{\otimes m-1}\otimes (V^{*})^{\otimes n}, V^*)      \\
\rotatebox[origin=c]{90}{$\in$} &  & \rotatebox[origin=c]{90}{$\in$} &  & \rotatebox[origin=c]{90}{$\in$}
\\
\rho(\gamma) & \leftrightarrow & \bar{\rho}(\Psi^{\strt}(\gamma)) & \leftrightarrow & \bar{\rho}(\Psi^{\dott}(\gamma))
\end{array}
$$
where $\Psi^{\strt}(\gamma)$ and $\Psi^{\dott}(\gamma)$ are the operations corresponding to the two possible target colors in $\Psi(\calP)$.
\end{proof}

\begin{example}
\label{ex::diop->op}
Let $\calF_{\diop}(\mu)$ be the free dioperad generated by a single element $\mu\in\calF(2,2)$, symmetric with respect to inputs and skew-symmetric with respect to outputs:
	\[
\mu:=	\qpgen{1}{2}{1}{2} = - \qpgen{2}{1}{1}{2} = \qpgen{1}{2}{2}{1} = -\qpgen{2}{1}{2}{1}.
	\] 
Then the $2$-colored operad $\Psi(\calF_{\diop}(\mu))$ is the free $2$-colored operad generated by the following two operations, which satisfy the following symmetry conditions:    
$$
\mu_+:=
\begin{tikzpicture}[scale =1.2]
            \node[invisible] (r) at (0,-.25) {};
			\node[int] (c) at (0,0) {};
            \draw[-] (r) -- (c);
			\node[invisible] at (0,0.25) {}
			edge [-] (c);
			\node[invisible] (r) at (-0.25,0.25) {}
			edge [-] (c);
			\node[invisible] at (0.25,0.25) {}
			edge[dotted] [-] (c);
			\node[] at (-0.25,0.4) {$\scriptstyle 1$};
			\node[] at (0,0.4) {$\scriptstyle 2$};
			\node[] at (0.25,0.4) {$\scriptstyle \bar{1}$};
\end{tikzpicture} = 
\begin{tikzpicture}[scale =1.2]
            \node[invisible] (r) at (0,-.25) {};
			\node[int] (c) at (0,0) {};
            \draw[-] (r) -- (c);
			\node[invisible] at (0,0.25) {}
			edge [-] (c);
			\node[invisible] (r) at (-0.25,0.25) {}
			edge [-] (c);
			\node[invisible] at (0.25,0.25) {}
			edge[dotted] [-] (c);
			\node[] at (-0.25,0.4) {$\scriptstyle 2$};
			\node[] at (0,0.4) {$\scriptstyle 1$};
			\node[] at (0.25,0.4) {$\scriptstyle \bar{1}$};
\end{tikzpicture}; \ \quad
\mu_{-}:=
\begin{tikzpicture}[scale =1.2]
            \node[invisible] (r) at (0,-.25) {};
			\node[int] (c) at (0,0) {};
            \draw[-,dotted] (r) -- (c);
			\node[invisible] at (0,0.25) {}
			edge[dotted] [-] (c);
			\node[invisible] (r) at (-0.25,0.25) {}
			edge [-] (c);
			\node[invisible] at (0.25,0.25) {}
			edge[dotted] [-] (c);
			\node[] at (-0.25,0.4) {$\scriptstyle 1$};
			\node[] at (0,0.4) {$\scriptstyle \bar{1}$};
			\node[] at (0.25,0.4) {$\scriptstyle \bar{2}$};
\end{tikzpicture} = -
\begin{tikzpicture}[scale =1.2]
            \node[invisible] (r) at (0,-.25) {};
			\node[int] (c) at (0,0) {};
            \draw[-,dotted] (r) -- (c);
			\node[invisible] at (0,0.25) {}
			edge[dotted] [-] (c);
			\node[invisible] (r) at (-0.25,0.25) {}
			edge [-] (c);
			\node[invisible] at (0.25,0.25) {}
			edge[dotted] [-] (c);
			\node[] at (-0.25,0.4) {$\scriptstyle 1$};
			\node[] at (0,0.4) {$\scriptstyle \bar{2}$};
			\node[] at (0.25,0.4) {$\scriptstyle \bar{1}$};
\end{tikzpicture}.
$$
\end{example}

The functor $\Psi$ generalizes straightforwardly to the functor $\Psi^{\op}$ for co-dioperads and colored co-operads, leading to the following observation:
\begin{theorem}
\label{thm::diop->op}
	The functors $\Psi$ and $\Psi^{\op}$ 
    $$
\begin{array}{cccc}
\Psi: &  \textbf{Dioperads} &\rightarrow & \textbf{$2$-colored Operads} \\
    \Psi^{\op}: & \textbf{co-Dioperads} & \rightarrow & \textbf{$2$-colored co-Operads}
\end{array}    
    $$ 
    are exact, maps (co)free objects to (co)free objects and commutes with the bar and cobar constructions $\cB, \Omega$:
\begin{equation}
	\label{eq::cobar::op::diop}
		\Psi \circ \Omega_{\mathsf{Dioperads}} = \Omega_{\mathsf{Operads}} \circ \Psi^{\op}, \quad
	\Psi^{\op} \circ \cB_{\mathsf{Dioperads}} = \cB_{\mathsf{Operads}} \circ \Psi.
\end{equation}		
\end{theorem}
\begin{proof}
The proof relies on a direct comparison of dioperadic and operadic trees and their coloring.
Each dioperadic tree is uniquely identified by replacing the inner and outer directions with coloring information, establishing a bijection between these classes of trees.
Consequently, $\Psi$ maps a free dioperad to a free $2$-colored operad.
The (co)bar construction $\Omega_{\mathsf{Dioperads}}$ assigns to a co-dioperad $\cP$ the free dg-dioperad generated by $s\cP$, where $s$ denotes the appropriate shift of the collection (cf. \cite{Loday::Vallette}).
Similarly, the construction $\Omega_{\mathsf{Operads}}$ assigns to a colored cooperad $\Psi^{\op}(\cP)$ the free colored dg-operad generated by $s\Psi^{\op}(\cP)$ with the same shift. The equivalences in \eqref{eq::cobar::op::diop} then follow from the isomorphism of the underlying free objects.
\end{proof}

\subsection{Koszul property and Hilbert series}
\label{sec::Hilbert}
\begin{corollary}
\label{cor::Psi::Koszul}
The dioperad $\calP$ is quadratic (respectively Koszul) if and only if the associated $2$-colored operad $\Psi(\calP)$ is quadratic (respectively Koszul).
\end{corollary}

\begin{proof}
The functor $\Psi$ sends a free dioperad generated by a $\bS\times\bS^{\op}$-collection $\Upsilon$ to the free colored operad generated by the collection $\Psi(\Upsilon)$. More generally, a presentation of a dioperad by generators and relations is mapped by $\Psi$ to a presentation of a colored operad whose generators and relations are the images under $\Psi$ of the original ones. In particular, $\Psi$ preserves quadratic presentations.
Furthermore, thanks to Theorem~\ref{thm::diop->op}, a morphism of dg-dioperads is a quasi-isomorphism if and only if its image under $\Psi$ is a quasi-isomorphism. Therefore, the canonical morphism defining Koszulness satisfies
$$
	\Omega_{\diop}(\cP^{\antish}) \stackrel{\quis}{\longrightarrow} \cP
	\ \Leftrightarrow\ 
    \Omega_{\mathsf{Operad}}\bigl(\Psi(\cP^{\antish})\bigr)
    =
    \Psi\bigl(\Omega_{\diop}(\cP^{\antish})\bigr)
    \stackrel{\quis}{\longrightarrow}
    \Psi(\cP),
$$
which proves the claim.
\end{proof}

In particular, one can apply the theory of Gr\"obner bases and rewriting systems developed for colored operads in~\cite{colored_Grobner} in order to prove the Koszulness of quadratic dioperads. We briefly recall this theory in the next section. We also state Corollary~\ref{cor::sereis::diop::Koszul} concerning the generating series of Koszul dual dioperads.

Suppose that the dioperad $\cP$ is graded,
$\cP=\bigoplus_{r=0}^{\infty}\cP^{(r)}.$
We then define its graded generating series by
$$
\rchi_{\cP}(x,y;q):=\sum_{m,n} \left(\sum_{r=0}^{\infty} q^{r}\dim \cP(m,n)^{(r)}\right) \frac{x^m}{m!}\frac{y^n}{n!}
= \sum_{m,n} \dim_q \cP(m,n) \frac{x^m}{m!}\frac{y^n}{n!}.
$$
For example, if the dioperad $\cP=\Free(\Upsilon,R)$ is quadratic, then $\cP^{(0)}=\cP(1,1)=\Id$, $\cP^{(1)}=\Upsilon$, and
\[
\cP^{(2)}=\frac{\Upsilon\circ'\Upsilon}{R},
\]
where $\circ'$ denotes the infinitesimal composition. Moreover, in Examples~\S\ref{example::Frobenius}, \S\ref{example::Lie::Bialg}, \S\ref{example::QP} considered here, the total number of inputs and outputs of the generators $\Upsilon$ is fixed. As a result, the parameter $q$ is uniquely determined by the $x$- and $y$-gradings, since the sum of the numbers of inputs and outputs is constant on each graded component.

\begin{corollary}
\label{cor::sereis::diop::Koszul}
If the dioperad $\cP$ is Koszul, then the graded generating series of $\cP$ and its Koszul dual $\cP^{!}$ satisfy the following algebraic identity:
\begin{equation}
\label{eq::generating::Koszul}
\left(\frac{\partial \rchi_{\cP^{!}}(x,y;-q)}{\partial y}; \frac{\partial \rchi_{\cP^{!}}(x,y;-q)}{\partial x}\right) \circ	
\left(\frac{\partial \rchi_{\cP}(x,y;q)}{\partial y}; \frac{\partial \rchi_{\cP}(x,y;q)}{\partial x}\right) = (x;y).
\end{equation}
Here $f\circ g$ denotes the composition of formal diffeomorphisms of a disk at the origin. Namely, 
\begin{multline*}
f\circ g:=(f_1(g_1(x,y),g_2(x,y));f_2(g_1(x,y),g_2(x,y)), \\ \text{ where } f:=(f_1(x,y);f_2(x,y)) \text{ and } g:=(g_1(x,y);g_2(x,y)).
\end{multline*}
\end{corollary}

\begin{proof}
The argument follows the standard proof for operads (see, for example,~\cite[\S7.5]{Loday::Vallette}). The Koszulness of the dioperad $\cP$ is equivalent to the Koszulness of the colored operad $\Psi(\cP)$. In turn, this is equivalent to the acyclicity of the Koszul complex $\cK(\Psi(\cP))$, which, as a graded vector space, is isomorphic to $\Psi(\cP^{\antish})\circ'\Psi(\cP)$. The homological grading is determined by the degree of $\cP^{\antish}$.

It therefore remains to compute the generating series of $\Psi(\cP)$, as is done for colored operads (see, for example,~\cite{colored_Grobner}). The generating series of a colored operad is given by a collection of generating series corresponding to operations with a fixed output color. In the present situation, there are two such series, corresponding to the ``$\strt$''-(straight) and ``$\dott$''-(dotted) colors:
\begin{gather*}
F_{\cP}^{\strt}(x,y)
:= \sum_{m,n} \dim_q \Psi(\cP)^{\strt}(m,n) \frac{x^m}{m!}\frac{y^n}{n!}
= \sum_{m,n}\dim_q\cP(m,n+1)\frac{x^m}{m!}\frac{y^n}{n!}
= \frac{\partial \rchi_{\cP}(x,y;q)}{\partial y}, \\
F_{\cP}^{\dott}(x,y)
:= \sum_{m,n} \dim_q \Psi(\cP)^{\dott}(m,n) \frac{x^m}{m!}\frac{y^n}{n!}
= \sum_{m,n}\dim_q\cP(m+1,n)\frac{x^m}{m!}\frac{y^n}{n!}
= \frac{\partial \rchi_{\cP}(x,y;q)}{\partial x}.
\end{gather*}
Finally, the generating series of the operad $\Psi(\cP^{!})$ coincides with that of the cooperad $\Psi(\cP^{\antish})$, and replacing $q$ by $-q$ accounts for the homological grading, which contributes signs to the Euler characteristic.
\end{proof}

We illustrate this theory with the computation of the generating series for the dioperad of Lie bialgebras (\S\ref{example::Lie::Bialg}).

Notably, one may also define generating series for colored operads that account for the action of the symmetric group (see e.g.~\cite[\S1.6]{colored_Grobner}). The corresponding functional equation is expressed via plethystic substitutions of symmetric functions, a framework that generalizes readily to the dioperadic setting.

\subsection{Colored shuffle operads and Gr\"obner bases for dioperads}
\label{sec::Shuffle::Grobner}

Gr\"obner basis theory originates from the study of monomials, their divisibility properties, and the existence of a monomial ordering compatible with multiplication (see, for example,~\cite{Grobner::book} for a general introduction in the commutative setting).
Such an approach cannot be directly applied to symmetric operads, since there is no monomial ordering compatible with the action of the symmetric group. This difficulty is resolved by the notion of shuffle operads, introduced in~\cite{DK_Grobner}, where the symmetric group action is forgotten, while keeping the set of all generating compositions. In~\cite{colored_Grobner}, this idea was extended to the colored setting by introducing colored shuffle operads.

From a pictorial point of view, the passage from dioperads to colored shuffle operads is straightforward. However, it leads to a significant increase in the number of generators, since one must consider all possible choices of colorings. In what follows, we recall the main definitions from~\cite{colored_Grobner} and~\cite{DK_Grobner}, adapting them to the case of $2$-colored operads.

\begin{remark}
During the preparation of this manuscript, we found a description of PBW bases for dioperads in the doctoral thesis~\cite{PBW_Dioperads}. In our view, the construction presented there contains several gaps that limit its utility. Our primary concern is that in~\cite[\S3.03]{PBW_Dioperads}, the author introduces a \emph{pointed shuffle} and claims it can be used to generate a basis for the free dioperad via tree monomials. However, unlike in the operadic case, the pointed shuffle composition is not associative. Furthermore, the description of dioperadic $(m,n)$-trees does not specify the ordering of the inner vertices. Consequently, the description of the basis for a free dioperad lacks necessary detail and, in its current form, does not appear to be correct.
\end{remark}

\subsubsection{Shuffle (colored) operads}

\begin{definition}
 A \emph{shuffle $2$-colored tree} $T$ is a planar rooted tree with edges oriented from leaves toward the root, satisfying:
    \begin{itemize}[itemsep=0pt,topsep=0pt]
        \item Each edge is assigned one of two specified colors.
        \item The set of leaves of $T$ is equipped with a total (linear) order.
        \item For every internal vertex $v \in T$, the relative order of the minimal labels of its incoming subtrees is compatible with the planar structure.
    \end{itemize}
\end{definition}    
\begin{notation}    
\begin{itemize}[itemsep=2pt,topsep=0pt]
    \item A \emph{shuffle corolla} is a shuffle tree possessing exactly one internal vertex.
    \item For each internal vertex $v$ of a shuffle tree $T$, we denote by $\mathrm{ar}(v)$ the corresponding corolla that preserves the coloring of the inputs and output of $v$. 
    \item For a shuffle tree $T$, we denote by $\mathrm{ar}(T)$ the total arity and coloring of its leaves and root.
\end{itemize}
\end{notation}

We represent shuffle trees in the plane with the root at the bottom and all edges oriented downward. The planar structure is fixed by the local order of minima, which increases from left to right. Consequently, in a shuffle corolla, the leaves are arranged from left to right according to their labels.

The following figure displays three shuffle corollas, followed by three colored planar trees with two internal vertices. The first two are valid shuffle trees; the third is not, as the compatibility condition is violated at one vertex. The indices causing the violation are circled.

$$
\begin{array}{ccc}
\begin{tikzpicture}[scale =1.2]
            \node[invisible] (r) at (0,-.25) {};
			\node[int] (c) at (0,0) {};
            \draw[-] (r) -- (c);
			\node[invisible] at (0,0.25) {}
			edge[dotted] [-] (c);
			\node[invisible] (r) at (-0.25,0.25) {}
			edge [-] (c);
			\node[invisible] at (0.25,0.25) {}
			edge[dotted] [-] (c);
			\node[] at (-0.25,0.4) {$\scriptstyle 1$};
			\node[] at (0,0.4) {$\scriptstyle 2$};
			\node[] at (0.25,0.4) {$\scriptstyle 3$};
\end{tikzpicture}, \ 
\begin{tikzpicture}[scale =1.2]
            \node[invisible] (r) at (0,-.25) {};
			\node[int] (c) at (0,0) {};
            \draw[-,dotted] (r) -- (c);
			\node[invisible] at (0,0.25) {}
			edge[dotted] [-] (c);
			\node[invisible] (r) at (-0.25,0.25) {}
			edge [-] (c);
			\node[invisible] at (0.25,0.25) {}
			edge [-] (c);
			\node[] at (-0.25,0.4) {$\scriptstyle 1$};
			\node[] at (0,0.4) {$\scriptstyle 2$};
			\node[] at (0.25,0.4) {$\scriptstyle 3$};
\end{tikzpicture}, \ 
\begin{tikzpicture}[scale =1.2]
            \node[invisible] (r) at (0,-.25) {};
			\node[int] (c) at (0,0) {};
            \draw[-] (r) -- (c);
			\node[invisible] at (0,0.25) {}
			edge[dotted] [-] (c);
			\node[invisible] (r) at (-0.25,0.25) {}
			edge[dotted] [-] (c);
			\node[invisible] at (0.25,0.25) {}
			edge [-] (c);
			\node[] at (-0.25,0.4) {$\scriptstyle 1$};
			\node[] at (0,0.4) {$\scriptstyle 2$};
			\node[] at (0.25,0.4) {$\scriptstyle 3$};
\end{tikzpicture}; \ &
		\begin{tikzpicture}[scale=1.2,xscale=-1]
            \node[invisible] (b) at (0,-.25) {};
			\node[int] (c) at (0,0) {};
            \draw[-,dotted] (b) -- (c);
			\node[invisible] at (0,0.25) {}
			edge [-] (c);
			\node[int] (r) at (0.25,0.25) {}
			edge [-] (c);
			\node[invisible] at (-0.25,0.25) {}
			edge [-] (c);
			\node[invisible] at (0.15,0.55) {}
			edge [-,dotted] (r);
			\node[invisible] at (0.35,0.55) {}
			edge [-,dotted] (r);
			\node[] at (-0.25,0.4) {$\scriptstyle 3$};
			\node[] at (0,0.4) {$\scriptstyle 2$};
			\node[] at (0.15,0.7) {$\scriptstyle 4$};
			\node[] at (0.35,0.7) {$\scriptstyle 1$};
		\end{tikzpicture}, 
		\begin{tikzpicture}[scale=1.2,xscale=-1]
            \node[invisible] (b) at (0,-.25) {};
			\node[int] (c) at (0,0) {};
            \draw[-] (b) -- (c);
			\node[invisible] at (0,0.25) {}
			edge [-,dotted] (c);
			\node[int] (r) at (0.25,0.25) {}
			edge [-,dotted] (c);
			\node[invisible] at (-0.25,0.25) {}
			edge [-] (c);
			\node[invisible] at (0.15,0.55) {}
			edge [-,dotted] (r);
			\node[invisible] at (0.35,0.55) {}
			edge [-] (r);
			\node[] at (-0.25,0.4) {$\scriptstyle 4$};
			\node[] at (0,0.4) {$\scriptstyle 2$};
			\node[] at (0.15,0.7) {$\scriptstyle 3$};
			\node[] at (0.35,0.7) {$\scriptstyle 1$};
		\end{tikzpicture}; &
		\begin{tikzpicture}[scale=1.2]
            \node[invisible] (b) at (0,-.25) {};
			\node[int] (c) at (0,0) {};
            \draw[-,dotted] (b) -- (c);
			\node[invisible] at (0,0.25) {}
			edge [-] (c);
			\node[int] (r) at (0.25,0.25) {}
			edge [-] (c);
			\node[invisible] at (-0.25,0.25) {}
			edge [-] (c);
			\node[invisible] at (0.15,0.55) {}
			edge [-] (r);
			\node[invisible] at (0.35,0.55) {}
			edge [-] (r);
			\node[] at (-0.25,0.4) {\textcircled{$\scriptstyle 3$}};
			\node[] at (0,0.4) {$\scriptstyle 4$};
			\node[] at (0.15,0.7) {\textcircled{$\scriptstyle 1$}};
			\node[] at (0.35,0.7) {$\scriptstyle 2$};
            \draw[-,red] (-0.3,0.8) -- (0.4,-0.3);
            \draw[-,red] (0.4,0.8) -- (-0.3,-0.3);          
		\end{tikzpicture}. \\
\text{Shuffle corollas} & \text{Shuffle trees} & \text{Non-shuffle tree}
\end{array}
$$

\begin{definition}
A \emph{(colored) shuffle operad} $\calP$ in a monoidal category $\calC$ consists of:
\begin{itemize}[itemsep=2pt, topsep=0pt]
    \item An assignment of an object $\calP(\mathrm{ar}(v))$ in $\calC$ to each colored shuffle corolla $v$.
    \item For each colored shuffle tree $T$, an operadic composition map
    $$
    \circ_{T} : \bigotimes_{v\in T} \calP(\mathrm{ar}(v)) \longrightarrow \calP\bigl(\mathrm{ar}(T)\bigr),
    $$
\end{itemize}
where the tensor product is taken over the internal vertices of $T$.

The maps $\circ_T$ must be associative: for any shuffle subtree $T' \subset T$, performing composition within $T'$ followed by composition in the contracted tree $T/T'$ must coincide with the direct composition over $T$:
$$
\circ_{T} \simeq \circ_{T/T'} \circ \circ_{T'}.
$$
\end{definition}

The defining characteristic of shuffle operads is that free shuffle operads possess a monomial basis and admit an ordering compatible with operadic compositions.

\begin{proposition}(\cite[\S3.3-3.4]{colored_Grobner})
Let $\{\alpha_i\}_{i\in\mathsf{I}}$ be an alphabet of generators with prescribed arities and colors.
Let $\calB_{\shuffle}^{\{\alpha_i\}}(\ar(w))$ be the set of all shuffle trees or arity $\ar(w)$ whose vertices are labeled by elements of the alphabet such that the label at vertex $v$ matches $\mathrm{ar}(v)$. Then 
$$\left\{\calB_{\shuffle}^{\{\alpha_i\}}(\ar(w))\colon w \text{ is a corolla}  \right\}$$ 
forms the basis of the free colored shuffle operad generated by 
$\{\alpha_i\}$. 
\end{proposition}

\begin{definition}
In a free colored shuffle operad $\mathcal{F}ree(\{\alpha_i\})$, a  total order $\leq$ on the tree monomials of $\calB_{\shuffle}^{\{\alpha_i\}}(\ar(w))$ of the same (colored) arity $\ar(w)$ is \emph{admissible} if it is compatible with composition: 
$$ \forall \alpha\leq\alpha' \in \calB_{\shuffle}^{\{\alpha_i\}}(\ar(v)), \ \forall \beta\leq \beta'\in \calB_{\shuffle}^{\{\alpha_i\}}(\ar(w)) \quad \Rightarrow \quad \alpha\circ \alpha' \leq \beta\circ \beta' \in \calB_{\shuffle}^{\{\alpha_i\}}(\ar(v\circ w)).
$$
\end{definition}

It is well known that admissible orderings for (colored) shuffle operads exist (see~\cite[\S4.1]{colored_Grobner}). The most widely used is the \emph{path-lexicographical ordering} introduced by E.\,Hoffbeck in~\cite{Hoffbeck_PBW} (see also~\cite[\S3.2.1]{DK_Grobner}); for our computations, we find the generalizations suggested by V.\,Dotsenko in~\cite{Dotsenko_QP-order} particularly useful. We do not provide the technical details of these orderings here, as we will primarily employ the framework of rewriting systems in most of our examples.

\begin{definition}
Let $\mathcal{M}$ be an operadic ideal in a free colored shuffle operad $\mathcal{F}$. Given an admissible ordering of monomials, a set $G$ of generators of $\mathcal{M}$ is a \emph{Gr\"obner basis} if for every $f \in \mathcal{M}$, the leading term of $f$ is divisible by the leading term of some element in $G$.
\end{definition}

A key result from~\cite{DK_Grobner} that enables the application of Gr\"obner basis theory to symmetric operads is the construction of the forgetful functor

$$\shuffle : \text{(colored) symmetric operads} \longrightarrow \text{(colored) shuffle operads}.$$

This functor forgets the action of the symmetric group but
leaves the spaces of $n$-ary operations unchanged and preserves enough of the composition rules to ensure that the image of a free symmetric operad remains free in the shuffle setting.

\begin{remark}
We intend to apply this theory to $2$-colored operads derived from dioperads. Specifically, we seek a Gr\"obner basis for the shuffle $2$-colored operad ${\shuffle}(\Psi(\calP))$, where $\calP$ is a dioperad in the category of vector spaces. The primary drawback of this approach is the significant increase in the number of generators and defining relations within ${\shuffle}(\Psi(\calP))$.
\end{remark}

For instance, consider $\calP$ as the free dioperad from Example~\ref{ex::diop->op} generated by a single element $\mu$ with two inputs and two outputs:
	\[
\mu := \qpgen{1}{2}{1}{2} = - \qpgen{2}{1}{1}{2} = \qpgen{1}{2}{2}{1} = -\qpgen{2}{1}{2}{1}.
	\] 
The corresponding colored shuffle operad contains the following six generators, which are distinguished by the permutation of their colorings (which we also refer to as their arity):

$$
\begin{array}{cccccc}
\mu_+^{1}:=
\begin{tikzpicture}[scale =1.2]
            \node[invisible] (r) at (0,-.25) {};
			\node[int] (c) at (0,0) {};
            \draw[-] (r) -- (c);
			\node[invisible] at (0,0.25) {}
			edge [-] (c);
			\node[invisible] (r) at (-0.25,0.25) {}
			edge[dotted] [-] (c);
			\node[invisible] at (0.25,0.25) {}
			edge [-] (c);
			\node[] at (-0.25,0.4) {$\scriptstyle 1$};
			\node[] at (0,0.4) {$\scriptstyle 2$};
			\node[] at (0.25,0.4) {$\scriptstyle 3$};
\end{tikzpicture}; \ &
\mu_+^2:=
\begin{tikzpicture}[scale =1.2]
            \node[invisible] (r) at (0,-.25) {};
			\node[int] (c) at (0,0) {};
            \draw[-] (r) -- (c);
			\node[invisible] at (0,0.25) {}
			edge[dotted] [-] (c);
			\node[invisible] (r) at (-0.25,0.25) {}
			edge [-] (c);
			\node[invisible] at (0.25,0.25) {}
			edge [-] (c);
			\node[] at (-0.25,0.4) {$\scriptstyle 1$};
			\node[] at (0,0.4) {$\scriptstyle 2$};
			\node[] at (0.25,0.4) {$\scriptstyle 3$};
\end{tikzpicture}, \ &
\mu_+^{3}:=
\begin{tikzpicture}[scale =1.2]
            \node[invisible] (r) at (0,-.25) {};
			\node[int] (c) at (0,0) {};
            \draw[-] (r) -- (c);
			\node[invisible] at (0,0.25) {}
			edge [-] (c);
			\node[invisible] (r) at (-0.25,0.25) {}
			edge [-] (c);
			\node[invisible] at (0.25,0.25) {}
			edge[dotted] [-] (c);
			\node[] at (-0.25,0.4) {$\scriptstyle 1$};
			\node[] at (0,0.4) {$\scriptstyle 2$};
			\node[] at (0.25,0.4) {$\scriptstyle 3$};
\end{tikzpicture}, \ &
\mu_{-}^1:=
\begin{tikzpicture}[scale =1.2]
            \node[invisible] (r) at (0,-.25) {};
			\node[int] (c) at (0,0) {};
            \draw[-,dotted] (r) -- (c);
			\node[invisible] at (0,0.25) {}
			edge[dotted] [-] (c);
			\node[invisible] (r) at (-0.25,0.25) {}
			edge [-] (c);
			\node[invisible] at (0.25,0.25) {}
			edge[dotted] [-] (c);
			\node[] at (-0.25,0.4) {$\scriptstyle 1$};
			\node[] at (0,0.4) {$\scriptstyle 2$};
			\node[] at (0.25,0.4) {$\scriptstyle 3$};
\end{tikzpicture}, \ &
\mu_{-}^{2}:=
\begin{tikzpicture}[scale =1.2]
            \node[invisible] (r) at (0,-.25) {};
			\node[int] (c) at (0,0) {};
            \draw[-,dotted] (r) -- (c);
			\node[invisible] at (0,0.25) {}
			edge [-] (c);
			\node[invisible] (r) at (-0.25,0.25) {}
			edge[dotted] [-] (c);
			\node[invisible] at (0.25,0.25) {}
			edge[dotted] [-] (c);
			\node[] at (-0.25,0.4) {$\scriptstyle 1$};
			\node[] at (0,0.4) {$\scriptstyle 2$};
			\node[] at (0.25,0.4) {$\scriptstyle 3$};
\end{tikzpicture}, \ &
\mu_{-}^{3}:=
\begin{tikzpicture}[scale =1.2]
            \node[invisible] (r) at (0,-.25) {};
			\node[int] (c) at (0,0) {};
            \draw[-,dotted] (r) -- (c);
			\node[invisible] at (0,0.25) {}
			edge[dotted] [-] (c);
			\node[invisible] (r) at (-0.25,0.25) {}
			edge[dotted] [-] (c);
			\node[invisible] at (0.25,0.25) {}
			edge [-] (c);
			\node[] at (-0.25,0.4) {$\scriptstyle 1$};
			\node[] at (0,0.4) {$\scriptstyle 2$};
			\node[] at (0.25,0.4) {$\scriptstyle 3$};
\end{tikzpicture}. 
\end{array}
$$

\subsubsection{Rewriting systems for Shuffle operads}
\label{sec::rewriting::systems}

As observed, the functor $\Psi$ significantly increases the number of generators, occasionally making the standard Gr\"obner bases machinery too restrictive. Consequently, we propose a generalization via the notion of a \emph{rewriting system}. Unlike Gr\"obner bases, which require a global total ordering on all monomials, a rewriting system only requires a specific choice of a \emph{leading term} for each relation to define a directed reduction.

For the general theory of rewriting systems, we refer to~\cite{rewriting}, and for their adaptation to Gr\"obner bases for algebras and operads, we refer to~\cite[\S2.6]{Bremner_Dotsenko}. Below, we provide the definitions tailored to shuffle operads.

\begin{definition}
 A \emph{rewriting system} $S$ for a shuffle operad $\mathcal{P} = \mathcal{F}(\Upsilon) / J$, generated by a shuffle collection $\Upsilon$ subject to an ideal of relations $J$, is a set of rules $S = \{ (\tau_i, f_i) \}_{i \in I}$ where:
\begin{itemize}[itemsep=0pt,topsep=0pt]
    \item $\tau_i$ is a \emph{shuffle tree monomial} (designated as the leading monomial).
    \item $f_i \in \mathcal{F}(\Upsilon)$ is a linear combination of shuffle tree monomials.
\end{itemize}
Each pair defines a reduction rule $\tau_i \longrightarrow f_i$, such that the relations $\{ \tau_i - f_i \}_{i \in I}$ generate the ideal $J$.
\end{definition}

A shuffle tree monomial $T$ is called \emph{reducible} if it contains a subtree isomorphic to some $\tau_i$. Specifically, if $T = a \circ \tau_i \circ (b_1, \dots, b_k)$, where $\circ$ denotes operadic composition, the \emph{reduction step} is defined as:
\[ T \longrightarrow a \circ f_i \circ (b_1, \dots, b_k). \]
This operation extends linearly to the free shuffle operad $\mathcal{F}(\Upsilon)$. We write $g \to_S h$ if $h$ is obtained from $g$ by a finite sequence of such reductions.

\begin{definition}
A shuffle tree monomial is called $S$-\emph{irreducible} if it does not contain any $\tau_i$ as a subtree. The vector space of \emph{normal forms} is the subspace spanned by the set of all irreducible shuffle tree monomials:
\[ \mathcal{B}_{irr}^S = \{ T \in \text{Shuffle Tree Monomials} \colon \forall i \in I, \, \tau_i \text{ is not a subtree of } T \}. \]
\end{definition}

\begin{definition}
    The system $S$ is called \emph{terminating} if there is no infinite sequence of reductions $m \to_S m_1 \to_S m_2 \to_S \dots$ for any shuffle tree monomial $m$. 
\end{definition}    

In the shuffle operad setting, assuming termination, we define the following:
\begin{definition}
\begin{itemize}[itemsep=0pt, topsep=0pt]
\item \emph{Ambiguities} (overlaps) occur when two leading trees $\tau_i$ and $\tau_j$ are realized as subtrees of a shuffle tree $m$ such that: (1) they share at least one common internal vertex, and (2) all internal vertices of $m$ belong to the union of these two subtrees.
\item
For each such ambiguity, we form the \emph{$S$-polynomial} (the difference of possible reductions):
    \[ S(i, j) = (\text{reduction of } m \text{ using rule } i) - (\text{reduction of } m \text{ using rule } j). \]
\item The system is \emph{confluent} if for every ambiguity, the corresponding $S$-polynomial reduces to zero:
    \[ S(i, j) \rightarrow_S \ldots \rightarrow_S 0. \]
\end{itemize}    
\end{definition}

\begin{definition}
    A terminating and confluent rewriting system is called \emph{convergent}.
\end{definition}

The following theorem is the analogue of Bergman's Diamond Lemma for shuffle operads:

\begin{theorem}
\label{thm::rewrite::system}
A terminating rewriting system $S$ is confluent if and only if the set $\mathcal{B}_{irr}^S$ forms a vector space \emph{basis} for the shuffle operad $\mathcal{P}$.

In this case, every element in $\mathcal{P}$ possesses a \emph{unique normal form} in $\mathsf{Span}(\mathcal{B}_{irr}^S)$.
\end{theorem}

It is important to note that the rewriting system $\{(\mathrm{lm}(f), f) \colon f \in G\}$ associated with a Gr\"obner basis $G$ is always convergent (i.e. terminating and confluent). However, there exist rewriting systems that do not originate from a Gr\"obner basis.

\subsection{Koszulness and inclusion-exclusion resolution}
\label{sec::Anick}

One of the most prominent homological applications of noncommutative Gr\"obner bases is the \emph{Anick resolution} of the trivial module. This construction is based on deforming the differential in the minimal resolutions of monomial algebras, as introduced in~\cite{Anick} (for a detailed combinatorial treatment of the Anick resolution and Anick chains, we recommend the exposition in~\cite[\S 3]{ufnarovski}). 

Unfortunately, a similar description for the minimal resolution of a shuffle operad with monomial relations is currently unknown. However, in~\cite[\S2.1]{DK_Quillen}, we defined a dg-operad $\calI^{\udot}_{\calQ}$, termed the \emph{"inclusion-exclusion" operad}, which provides a quasi-free resolution of a monomial shuffle operad; this resolution proves to be minimal in many specific cases. We recall this construction below, as we are interested in its applications (see \S\ref{example::Lie::Triang}).

\begin{definition}
\label{def::Anick::chains}
The differential-graded \emph{inclusion-exclusion operad} $(\calI_{\calQ}^{\udot},d_{\calI})$ associated with a shuffle operad $\calQ=\calF(E;R)$ -- generated by the set $E$ subject to an ideal of relations generated by the set of monomials $R$ -- is the dg-operad whose monomial basis is given by:
$$\left\{
\begin{array}{c} m\otimes g_1\wedge\ldots\wedge g_k \\
\deg(m\otimes g_1\wedge\ldots\wedge g_k) = k
\end{array}
\ \left| \
\begin{array}{rcl}
m &- & \text{a shuffle monomial in }\calF(E),\\
g_i &- & 
\begin{array}{l}
{ \text{a divisor of } m \text{ isomorphic to }}\\
\text{a monomial relation in } R
\end{array}
\end{array}
\right.
\right\}$$ 
The differential in $\calI_{\calQ}^{\udot}$ omits one of the divisors while leaving the underlying monomial $m$ unchanged:
$$
d_{\calI}(m\otimes g_1\wedge\ldots\wedge g_k):= \sum_{i=1}^{k} (-1)^{i-1} m\otimes g_1\wedge\ldots\wedge \widehat{g_i} \wedge\ldots\wedge g_k
$$
The generators of the free shuffle operad $\calI_{\calQ}^{\udot}$ are monomials with divisors that cannot be factored as a product of two such monomials; in particular, the divisors $\{g_{\ldot}\}$ must form a non-split covering of the underlying operadic tree.
\end{definition}

Starting from a shuffle operad $\calP$ with a chosen Gr\"obner basis (or a convergent rewriting system), we consider the inclusion-exclusion quasi-free operad $\calI^{\udot}_{\gr\calP}$ associated with the corresponding monomial shuffle operad $\gr\calP$. By deforming the differential, one obtains a resolution of $\calP$, denoted by $\calI_{\calP}^{\udot}$ (see~\cite{DK_Quillen} for details). 

\begin{corollary}
\label{cor::Grobner->Koszul}
    If the $2$-colored operad $\Psi(\calD)$ associated with a dioperad $\calD$ admits a \emph{quadratic convergent rewriting system}, then:
\begin{itemize}[itemsep=0pt,topsep=0pt]
    \item the shuffle operad $\Psi(\calD^{!})$ associated with the quadratic dual dioperad $\calD^{!}$ also admits a quadratic convergent rewriting system;
    \item the dioperads $\calD$ and $\calD^{!}$ are Koszul.
\end{itemize}     
\end{corollary}

\begin{proof}
Consider a (colored) shuffle operad $\calQ$ with quadratic monomial relations. Its quadratic dual shuffle operad $\calQ^{!}$ possesses complementary quadratic monomial relations (i.e., a quadratic monomial is zero in $\calQ^{!}$ if and only if it is non-zero in $\calQ$). Furthermore, there is a one-to-one correspondence between the monomial generators of the inclusion-exclusion operad $\calI_{\calQ}^{\udot}$ and the non-zero shuffle monomials in $\calQ^!$. For grading reasons, this implies that $\calQ$ is Koszul and that the inclusion-exclusion operad $\calI_{\calQ}^{\udot}$ is indeed the minimal resolution of $\calQ$.

Now, suppose that the (colored) shuffle operad $\calP$ admits a quadratic convergent rewriting system such that the associated monomial quadratic shuffle operad is isomorphic to $\calQ$. By similar degree considerations (where the cohomology grading coincides with the syzygy grading), the corresponding inclusion-exclusion resolution $\calI^{\udot}_{\calP}$ with the deformed differential remains minimal. 
Consequently, $\calP$ is Koszul, and the monomials of $\calQ^{!}$ define a basis for $\calP^{!}$. Thus, for any quadratic monomial $g \in \calF(E)$ that vanishes in $\calQ^{!}$, one can assign a rewriting rule $g \mapsto f$, where $f$ is the linear combination of non-zero monomials in $\calQ^{!}$ representing the class of $g$ in $\calP^{!}$. The resulting rewriting system is clearly terminating and confluent because the monomials from $\calQ^{!}$ constitute a basis.
\end{proof}

We refer to \S\ref{example::Lie::Triang} for an explicit example of the Inclusion-Exclusion resolution applied to a dioperad with non-quadratic relations.

\section{Coloring of inputs/outputs in a cyclic operad}
\label{sec::coloring::cyclic}

\subsection{The pictorial maps $\Theta$ with examples $\Theta(\Lie)$ and $\Theta(\Ass)$}
\label{sec::Theta}
In this section, we suggest more pictorial maps governed by the coloring of inputs and outputs. As an application, we provide a way of producing rewriting systems for certain dioperads.

\begin{definition}
	A \emph{cyclic operad} $\calP$ consists of
	\begin{itemize}[itemsep=0pt, topsep=0pt]
		\item a collection of $m$-ary operations $\calP(m+1)$ with $m\geq 1$ equipped with $\bS_{m+1}$ action;
		\item infinitesimal composition rules: $\compos{i}{j} : \calP(m+1) \otimes \calP(n+1) \rightarrow \calP(m+n)$ that connects $i$'th and $j$'th inputs correspondingly.  
	\end{itemize}
such that composition rules are \emph{associative} and $\bS$-equivariant.
\end{definition}
In particular, after fixing an output in $n$-ary operations $\calP$ became an ordinary symmetric operad.

\begin{definition}
\begin{itemize}[itemsep=0pt, topsep=0pt]
    \item 
We call the subset $\mathfrak{c}\in \bZ_{\geq 0}\times\bZ_{\geq 0}$ to be the \emph{coloring rule} if it is closed under the following operation:
$$
\text{ if }(m,n), (m',n') \in \mathfrak{c} \ \text{ then }  \ (m+m'-1,n+n'-1)\in \mathfrak{c}.
$$
\item 
The $\mathfrak{c}$-coloring of the inputs/outputs by the rule:
$$\Theta_{\mathfrak{c}}(\calP)(m,n):=\begin{cases}
    \calP(m+n), \text{ if } (m,n)\in\mathfrak{c},\\
    0, \text{ otherwise. }
\end{cases}$$
defines an exact functor from the category of \textbf{Cyclic operads} to the category of \textbf{dioperads}. Where the composition rules in a dioperad coincide with the restriction of the composition rule in the cyclic operad $\calP$.
\end{itemize}
\end{definition}
The examples of the coloring rules include:
$$ \bZ_{>0}\times \{1\}; \quad \bZ_{>0}\times \bZ_{>0}; \quad \bZ_{\geq 0}\times \bZ_{>0}; \quad \{(n,n)\colon n\in\bZ_{>0}\}; \quad \{(k n+1, ln+1) \colon n\in\bZ_{>0}\}. $$
In particular,
\begin{itemize}[itemsep=0pt,topsep=0pt]
   \item The coloring with respect to the $\mathfrak{c}=\bZ_{>0}\times \{1\}$ defines an ordinary symmetric operad by forgetting the cyclic structure;
   \item The coloring  $\Theta_{\bZ_{>0}\times\bZ_{>0}}$ applied to the commutative operad $\Com$ gives the Frobenius dioperad (see \S\ref{example::Frobenius});
   \item The coloring $\Theta_{\mathfrak{c}}$ with respect to the set $\bZ_{\geq 0}\times \bZ_{>0}$ applied to the same operad $\Com$ of commutative algebra gives the dioperad that is Koszul dual to the dioperad of \emph{quasi-Lie bialgebras} and to the coloring $\bZ_{\geq 0}\times\bZ_{\geq 0}$ is isomorphic to the dioperad Koszul dual to the dioperad of \emph{pseudo Lie bialgebras} (see \S\ref{sec::PseudoLie}); 
   \item the coloring with an equal number of inputs/outputs of the same commutative operad provides the dioperad Koszul dual to the dioperad of quadratic Poisson structures (see \S\ref{example::QP} for details).
\end{itemize}

Note that the functor $\Theta_{\mathfrak{c}}$ may significantly alter the sets of generators and relations. However, Theorem~\ref{thm::cycl->diop} demonstrates that, under certain assumptions, this functor enables the construction of a Gr\"obner basis for the image and facilitates proofs of Koszulness for various natural dioperads.

\begin{example}
\label{ex::Lie::dioperad}
    The dioperad $\Theta_{\bZ_{>0}\times\bZ_{>0}}(\Lie)$ is generated by a pair of skew-symmetric generators (Lie bracket and Lie cobracket):
	$$
	\begin{array}{c}
		{
			\Theta_{\bZ_{>0}\times\bZ_{>0}}(\Lie)(2,1):=
\mbox{span}\left\langle
			\liePic{1}{2}
			=-
			\liePic{2}{1}
			\right\rangle
		},\\
        {
			\Theta_{\bZ_{>0}\times\bZ_{>0}}(\Lie)(1,2):=
			\mbox{span}\left\langle
			\coliePic{1}{2} = - 
            \coliePic{2}{1}
			\right\rangle
		}
	\end{array}
	$$
	by the $\bS\times\bS^{\op}$-invariant ideal generated by the following linear combinations
	\begin{equation*}
		\left\{
		\begin{array}{ccc}
			\LLiePic{1}{2}{3} + 	\LLiePic{2}{3}{1} + \LLiePic{3}{1}{2}; & \quad &
			\coLLiePic{1}{2}{3} + 	\coLLiePic{2}{3}{1} + \coLLiePic{3}{1}{2}; \\ 
			\begin{tikzpicture}[x=1mm,y=1mm]
				\node[draw,circle,inner sep=1pt,fill=white] (A) at (0,3) {};
				\draw (0,2.47) -- (0,0.12);                  
				\draw (0.5,3.5) -- (2.2,5.2);                
				\draw (-0.48,3.48) -- (-2.2,5.2);            
				\node[draw,circle,inner sep=1pt,fill=white] (B) at (0,-0.8) {};
				\draw (-0.39,-1.2) -- (-2.2,-3.5);           
				\draw (0.39,-1.2) -- (2.2,-3.5);             
				\node at (2.8,5.7) {\scriptsize $2$};
				\node at (-2.8,5.7) {\scriptsize $1$};
				\node at (-2.7,-5.2) {\scriptsize $1$};
				\node at (2.7,-5.2) {\scriptsize $2$};
			\end{tikzpicture}
			- \LiecoLiePic{1}{2}{1}{2}   
            +\LiecoLiePic{2}{1}{1}{2}; & \quad & 
            \LiecoLiePic{1}{2}{1}{2}  
           - \LiecoLiePicT{2}{1}{2}{1}
		\end{array}
		\right\}.
	\end{equation*}    
\end{example}

\begin{example}
\label{ex::Assoc::dioperad}
The dioperad $\Theta_{\bZ_{>0}\times\bZ_{>0}}(\Ass)$ assigned to the cyclic operad of associative algebras is generated by a pair of nonsymmetric multiplication and comultiplication:
	$$
	\begin{array}{c}
		{
			\Theta_{\bZ_{>0}\times\bZ_{>0}}(\Ass)(2,1):=
\mbox{span}\left\langle
			\liePic{1}{2},
			\liePic{2}{1}
			\right\rangle
		},
\quad
        {
			\Theta_{\bZ_{>0}\times\bZ_{>0}}(\Ass)(1,2):=
			\mbox{span}\left\langle
			\coliePic{1}{2}, 
            \coliePic{2}{1}
			\right\rangle
		}
	\end{array}
$$
Subject to the following list of relations (and all their symmetries):
	\begin{equation*}
		\label{eq::Theta:Lie}
		\left\{
		\begin{array}{ccc}
			\LLiePic{1}{2}{3} = \LLiePicReverse{1}{2}{3}; & 
			\coLLiePic{1}{2}{3} = 
            \coLLiePicReverse{1}{2}{3}; 
            &
			\begin{tikzpicture}[x=1mm,y=1mm]
				\node[draw,circle,inner sep=1pt,fill=white] (A) at (0,3) {};
				\draw (0,2.47) -- (0,0.12);                  
				\draw (0.5,3.5) -- (2.2,5.2);                
				\draw (-0.48,3.48) -- (-2.2,5.2);            
				\node[draw,circle,inner sep=1pt,fill=white] (B) at (0,-0.8) {};
				\draw (-0.39,-1.2) -- (-2.2,-3.5);           
				\draw (0.39,-1.2) -- (2.2,-3.5);             
				\node at (2.8,5.7) {\scriptsize $2$};
				\node at (-2.8,5.7) {\scriptsize $1$};
				\node at (-2.7,-5.2) {\scriptsize $1$};
				\node at (2.7,-5.2) {\scriptsize $2$};
			\end{tikzpicture}
			= \LiecoLiePic{1}{2}{1}{2}   
            = \LiecoLiePicT{2}{1}{2}{1}
		\end{array}
		\right\}.
	\end{equation*}    
\end{example}

\subsection{Defining relations with Gr\"obner bases for dioperads $\Theta_{\mathfrak{c}}(\calP)$}
\label{sec::Theta::Grobner}

The results in this section can be readily adapted to various other settings; however, to avoid unnecessary technical complications, we restrict our attention to the specific case $\mathfrak{c}:= \bZ_{>0} \times \bZ_{>0}$. In this context,
to each cyclic generator $\gamma \in \calP(n+1)$, we assign $n$ generators of different arities $\gamma_{i,j} \in \Theta_{\mathfrak{c}}(\calP)(i,j)$, where $i+j=n+1$ and $i,j > 0$. 
In particular,
one must exclude all operations of arities $(n+1,0)$ and $(0,n+1)$. 

\begin{proposition}
\label{prp::cyclic::coloring}
    Suppose that the cyclic operad $\calP$ admits a presentation $\calF(E;R)$ as a symmetric operad. In particular, assume that $E$ and $R$ possess cyclic symmetry. Then, the dioperad $\Theta_{\mathfrak{c}}(\calP)$ is generated by $\Theta_{\mathfrak{c}}(E)$ -- obtained by coloring the generators as inputs/outputs -- subject to the union of the relations $\Theta_{\mathfrak{c}}(R)$ (the set of all possible colorings of inputs, outputs, and internal edges of the relations in $R$) and the following quadratic relations, which admit the pictorial description below:
\begin{equation}
    \label{eq::input/output:rel}
    \forall \alpha,\beta \in E \qquad
    \begin{tikzpicture}[scale = 0.3]
        \node[ext] (v) at (0,0) {$\scriptscriptstyle \alpha$};
        \node (ov1) at (-1,-2) {};
        \node (ov2) at (0,-2) {$\scriptstyle \ldots$};
        \node (ov3) at (1,-2) {};
        \node (iv0) at (-2,2) {};
        \node (iv1) at (-1,2) {$\scriptstyle \ldots$};        
        \node (iv2) at (0,2.2) {};
        \draw (v) -- (ov1);
        \draw (v) -- (ov3);
        \draw (v) -- (iv0);
        \draw (v) -- (iv2);
        \node[ext] (w) at (4,4)  {$\scriptscriptstyle \beta$};
        \node (ow1) at (3,2) {};
        \node (ow2) at (4,2) {$\scriptstyle \ldots$};
        \node (ow3) at (5,2) {};
        \node (iw0) at (3.5,6) {};
        \node (iw1) at (5,6) {$\scriptstyle \ldots$};        
        \node (iw2) at (6,6) {};
        \node (iw3) at (4,6) {};
        \draw (w) -- (ow1);
        \draw (w) -- (ow3);
        \draw (w) -- (iw0);
        \draw (w) -- (iw2);
        \draw (w) -- (iw3);
        \draw[->] (w) -- (v);
    \end{tikzpicture}
    =
    \begin{tikzpicture}[scale = 0.3]
        \node[ext] (v) at (0,4) {$\scriptscriptstyle \alpha$};
        \node (ov1) at (-1,4-2) {};
        \node (ov2) at (0,4-2) {$\scriptstyle \ldots$};
        \node (ov3) at (1,4-2) {};
        \node (iv0) at (-2,4+2) {};
        \node (iv1) at (-1,4+2) {$\scriptstyle \ldots$};        
        \node (iv2) at (0,4+2.2) {};
        \draw (v) -- (ov1);
        \draw (v) -- (ov3);
        \draw (v) -- (iv0);
        \draw (v) -- (iv2);
        \node[ext] (w) at (4,4-4)  {$\scriptscriptstyle \beta$};
        \node (ow1) at (3,2-4) {};
        \node (ow2) at (4,2-4) {$\scriptstyle \ldots$};
        \node (ow3) at (5,2-4) {};
        \node (iw0) at (3.5,6-4) {};
        \node (iw1) at (5,6-4) {$\scriptstyle \ldots$};        
        \node (iw2) at (6,6-4) {};
        \node (iw3) at (4,6-4) {};
        \draw (w) -- (ow1);
        \draw (w) -- (ow3);
        \draw (w) -- (iw0);
        \draw (w) -- (iw2);
        \draw (w) -- (iw3);
        \draw[->] (v) -- (w);
    \end{tikzpicture}    
\end{equation}        
\end{proposition}

\begin{proof}
Directly from the definitions of the pictorial functors $\Theta_{\mathfrak{c}}$ and $\Psi$, it follows that any element in the $2$-colored operad $\Psi(\Theta_{\mathfrak{c}}(\calP))$ can be represented as a linear combination of operadic trees underlying $\calP$, equipped with a coloring of all edges in one of two colors. This coloring is subject to a mild constraint: no vertex exists where all incoming edges are of one color and the outgoing edge is of a different color.
In particular, $\Theta_{\mathfrak{c}}(\calP)$ is clearly generated by $\Theta_{\mathfrak{c}}(E)$, since each corolla is obtained by assigning an appropriate coloring to the adjacent edges -- effectively determining which are inputs and which are outputs. Relation~\eqref{eq::input/output:rel} implies that the coloring of internal edges in the operadic tree is immaterial.  
This ensures that all aforementioned relations are satisfied in $\Theta_{\mathfrak{c}}(\calP)$.

Conversely, one can easily show by induction on the size of a coloured monomial $M$ that any two different admissible colorings of internal edges of $M$ can be connected by a sequence of changing the color in one edge. 
Therefore, the size of the dioperad generated by $\Theta_{\mathfrak{c}}(E)$, subject to the coloring of relations $\Theta_{\mathfrak{c}}(R)$ and the invariance of internal colors, cannot exceed the size of the dioperad $\Theta_{\mathfrak{c}}(\calP)$, which concludes the proof.
\end{proof}

In particular, we observe that the functor $\Theta_{\mathfrak{c}}$ maps quadratic cyclic operads to quadratic dioperads. It is not currently known whether $\Theta_{\mathfrak{c}}$ maps Koszul cyclic operads to Koszul dioperads, though we suspect this is not generally the case. However, we provide below sufficient conditions on the Gr\"obner basis of a cyclic operad $\calP$ that ensure the Koszulness of the corresponding dioperad $\Theta_{\mathfrak{c}}(\calP)$. To this end, we must analyze the effect of the composition of functors $\Psi \circ \Theta_{\mathfrak{c}}$ on shuffle monomials.

Indeed, suppose we choose a set of shuffle generators $\{\gamma_i : i \in \mathsf{I}\}$ for the free cyclic operad $\calF$ generated by a cyclic collection $E$. In other words, $\{\gamma_i : i \in \mathsf{I}\}$ defines a basis for the $\bS$-collection $E$. Then the $2$-colored shuffle operad $\shuffle(\Psi(\Theta_{\mathfrak{c}}(\calF)))$ is generated by the following set:
\begin{equation}
\label{eq::colored::generators}
\{\gamma_{i,\bar{c}} : i \in \mathsf{I}, \ \bar{c} \in \bF_2^{m_i+1} \setminus \{(0;1,\ldots,1),(1;0,\ldots,0)\}, \text{ where } \ar(\gamma_i)=m_i\}.
\end{equation}
Specifically, to each generator $\gamma_i$ of cyclic arity $m_i+1$, we assign a $(2^{m_i+1}-2)$-tuple of $m_i$-ary generators in the $2$-colored shuffle operad. The vector $\bar{c}$ denotes the coloring of the inputs and outputs: $c_0=0$ if the output is a straight edge, while for $j=1,\ldots,m$, the $j$-th coordinate $c_j$ equals $0$ (resp. $1$) if the $j$-th input of $\gamma_{i,\bar{c}}$ is colored as a straight (resp. dotted) edge.

Furthermore, the quadratic relation in Equation~\eqref{eq::input/output:rel} leads to the following set of quadratic equations in the shuffle operad:
\begin{equation}
\label{eq::rewrite::recolor}
    \gamma_{i,(\dots,1,\dots)} \circ_a \gamma_{j(1;\dots)} = \gamma_{i,(\dots,0,\dots)} \circ_a \gamma_{j,(0;\dots)}.
\end{equation}
That we call the \emph{"recoloring relation"} because they change the color of an inner edge in the shuffle colored operad $\shuffle(\Psi(\Theta_{\mathfrak{c}}(\calF)))$.

The corresponding rewriting rule, when applied to any colored shuffle monomial, does not change the underlying uncolored monomial. Thus, to any colored shuffle monomial $T$ in the free $2$-colored shuffle operad generated by~\eqref{eq::colored::generators}, we assign the set:
$$[T]_{\mathfrak{c}} := \{ T' \in \calF(\gamma_{i,\bar{c}}) : \shape(T') = \shape(T) \},$$
where $\shape(T)$ denotes the underlying shuffle monomial obtained by forgetting the colors of all internal edges while preserving the colors of the inputs and outputs. Let us define a partial order $\prec$ on this set by declaring the covering relations to be given by recoloring a single dotted internal edge to a straight one.
As illustrated in the following diagrammatic Example~\eqref{ex::shape::monomial}, this poset may lack a supremum, and instances without an infimum also exist.
\begin{equation}
\label{ex::shape::monomial}
\left[
\begin{tikzpicture}[scale=0.2,baseline=(current bounding box.center)]
    \node[ext] (v0) at (0,0) {};
    \node[ext] (v1) at (-2,2) {};
    \node[ext] (v2) at (2,2) {};
    \node (l0) at (0,-2) {};
    \node (l1) at (-4,4.2) {};
    \node (l2) at (-0.3,4.2) {};
    \node (l3) at (0.3,4.2) {};
    \node (l4) at (4,4.2) {};
    \draw[dotted] (l0) -- (v0);
    \draw[dotted] (v0) -- (v1);
    \draw[dotted] (v0) -- (v2);
    \draw[dotted] (v1) -- (l1);
    \draw (v1) -- (l2);
    \draw (v2) -- (l3);
    \draw[dotted] (v2) -- (l4);
\end{tikzpicture}   
\right]_{\mathfrak{c}}:=
\left\{
\begin{tikzpicture}[scale=0.2, baseline=(current bounding box.center)]
    \node[ext] (v0) at (0,0) {};
    \node[ext] (v1) at (-2,2) {};
    \node[ext] (v2) at (2,2) {};
    \node (l0) at (0,-2) {};
    \node (l1) at (-4,4.2) {};
    \node (l2) at (-0.3,4.2) {};
    \node (l3) at (0.3,4.2) {};
    \node (l4) at (4,4.2) {};
    \draw[dotted] (l0) -- (v0);
    \draw (v0) -- (v1);
    \draw[dotted] (v0) -- (v2);
    \draw[dotted] (v1) -- (l1);
    \draw (v1) -- (l2);
    \draw (v2) -- (l3);
    \draw[dotted] (v2) -- (l4);
\end{tikzpicture}   
\succ
\begin{tikzpicture}[scale=0.2, baseline=(current bounding box.center)]
    \node[ext] (v0) at (0,0) {};
    \node[ext] (v1) at (-2,2) {};
    \node[ext] (v2) at (2,2) {};
    \node (l0) at (0,-2) {};
    \node (l1) at (-4,4.2) {};
    \node (l2) at (-0.3,4.2) {};
    \node (l3) at (0.3,4.2) {};
    \node (l4) at (4,4.2) {};
    \draw[dotted] (l0) -- (v0);
    \draw[dotted] (v0) -- (v1);
    \draw[dotted] (v0) -- (v2);
    \draw[dotted] (v1) -- (l1);
    \draw (v1) -- (l2);
    \draw (v2) -- (l3);
    \draw[dotted] (v2) -- (l4);
\end{tikzpicture}   
\prec
\begin{tikzpicture}[scale=0.2, baseline=(current bounding box.center)]
    \node[ext] (v0) at (0,0) {};
    \node[ext] (v1) at (-2,2) {};
    \node[ext] (v2) at (2,2) {};
    \node (l0) at (0,-2) {};
    \node (l1) at (-4,4.2) {};
    \node (l2) at (-0.3,4.2) {};
    \node (l3) at (0.3,4.2) {};
    \node (l4) at (4,4.2) {};
    \draw[dotted] (l0) -- (v0);
    \draw[dotted] (v0) -- (v1);
    \draw (v0) -- (v2);
    \draw[dotted] (v1) -- (l1);
    \draw (v1) -- (l2);
    \draw (v2) -- (l3);
    \draw[dotted] (v2) -- (l4);
\end{tikzpicture}   
\right\}.
\end{equation}
 However, such cases do not occur for \emph{caterpillar} monomials, which we define as follows:
\begin{definition}
\label{def::caterpilar}
\begin{itemize}[itemsep=0pt,topsep=0pt]
    \item An operadic tree is called a \emph{caterpillar tree} if it contains no vertex with more than one incoming internal edge. Equivalently, the subtree spanned by the internal vertices is a \emph{path} (also referred to as the \emph{spine} or \emph{trunk}), and all leaves are attached directly to this path. Pictorially, the following subtree is forbidden:
    
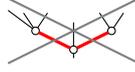
\[
    \begin{tikzpicture}[scale=0.25, baseline=-0.5ex]
      \node[ext] (v0) at (0,0) {};
      \node[ext] (v1) at (-2,1) {};
      \node[ext] (v2) at (2,1) {};
      \node (l0) at (0,-1) {};
      \node (l1) at (-4,2.5) {};
      \node (l2) at (-3,2.5) {};
      \node (l3) at (-1,2.5) {};
      \node (l4) at (0,2) {};
      \node (l5) at (1,2.5) {};
      \node (l6) at (3,2.5) {};
      \draw (l0) -- (v0);
      \draw[very thick, red] (v0) -- (v1);
      \draw[very thick, red] (v0) -- (v2);
      \draw (v0) -- (l4);
      \draw (v1) -- (l1);
      \draw (v1) -- (l2);
      \draw (v1) -- (l3);
      \draw (v2) -- (l5);
      \draw (v2) -- (l6);
      \node (a1) at (-4, -1) {};
      \node (a2) at (-4, 3) {};
      \node (b1) at (4, -1) {};
      \node (b2) at (4, 3) {};  
      \draw[thick, gray] (a1) -- (b2);
      \draw[thick, gray] (a2) -- (b1);
    \end{tikzpicture}
    \]
    \item A shuffle tree monomial $T$ is called a \emph{caterpillar monomial} if its underlying operadic tree is a caterpillar tree.
\end{itemize}
\end{definition}

\begin{lemma}
\label{lem::caterpillar::coloring}
For any coloring of the inputs and outputs of a caterpillar shuffle monomial $T$, the poset $([T]_{\mathfrak{c}}, \prec)$ of all colored shuffle monomials of a given shape possesses a unique supremum $T^{\max}$. This supremum is characterized by having the maximal possible number of internal edges colored as straight.
\end{lemma}

\begin{proof}
Recall that for a caterpillar monomial, the set of internal edges forms a spine (a central path). For a given coloring of the leaves (inputs and outputs), certain internal edges may have forced colors, while others may admit different colorings. We claim that there exists a unique maximal monomial $T^{\max}$ in which all such "flexible" internal edges are colored as straight.

Consider any monomial $T \in [T]_{\mathfrak{c}}$. 
If the color of the edge adjacent to the root is fixed, we can cut this edge and proceed by induction on the size of the caterpillar monomial. So, without loss of generality, we assume that the first edge is flexible. This may happen whenever the color of the output coincides with the color of one of the non-spine inputs of the root vertex.
If $T$ is not $T^{\max}$, there must exist at least one internal edge $e$ colored as dotted that can be recolored as straight without violating the consistency relations at the vertices. To see this, consider the set of nonrooted vertices where one of the inputs is colored as straight but the output is colored as dotted. Among these, let $v$ be a vertex that is closest to the root, and let $v'$ be the adjacent vertex immediately closer to the root. By the minimality of $v$, the vertex $v'$ must either have a straight output, or all of its inputs must be dotted, or $v'$ is a root (otherwise, $v'$ would have been chosen instead of $v$). 
Consequently, the configuration allows for the internal edge connecting $v$ and $v'$ to be changed from dotted to straight. By iteratively applying this recoloring process, we move upward in the poset $([T]_{\mathfrak{c}}, \prec)$ until all possible internal edges are straight. This process terminates at a unique element $T^{\max}$, which serves as the supremum of the poset.
\end{proof}

\begin{definition}
A convergent rewriting system $S$ for a (colored) shuffle operad is called \emph{caterpillar} if every irreducible shuffle monomial in $\calB_{irr}^{S}$ is a caterpillar monomial.
\end{definition}

\begin{example}
\begin{itemize}[itemsep=0pt,topsep=0pt]
    \item The standard quadratic Gr\"obner basis of the commutative operad $\Com$ (with respect to the path-lexicographical ordering) defines a caterpillar rewriting system. This is because all irreducible monomials are \emph{right combs}, where the internal vertices grow exclusively to the right.
    \item The quadratic Gr\"obner basis of the Lie operad $\Lie$ (with respect to the reverse path-lexicographical ordering) is caterpillar, as its normal forms consist of \emph{left combs}, where the internal vertices grow exclusively to the left.
    \item There exists a quadratic Gr\"obner basis for the associative operad $\mathsf{Assoc}$ such that all normal forms are left-normalized combs (see e.g.~\cite[Example 2.14]{Khor::Univ::Envelope}) and thus $\mathsf{Assoc}$ admits a quadratic caterpillar rewriting system.
    \item Following results on generating series (see~\cite[Example 5.11]{Khor::Univ::Envelope}), one can show that the Koszul-self dual cyclic operad $\mathsf{Pois}$ of Poisson algebras does not admit a (quadratic) caterpillar rewriting system.
\end{itemize}
\end{example}

Suppose that the shuffle operad $\shuffle(\calP)$ associated with a given cyclic operad $\calP$ admits a caterpillar rewriting system. That is, we are given a set of generators $\{\gamma_i : i \in \mathsf{I}\}$ of arity $m_i$ and a terminating, confluent rewriting system $S = \{ (\tau_j, f_j) \}_{j \in \mathsf{J}}$ such that all normal forms are caterpillar. In particular, each $f_j$ is spanned by caterpillar monomials.

Then, by Lemma~\ref{lem::caterpillar::coloring}, for each $j \in \mathsf{J}$, any coloring $\bar{c}$ of the inputs/outputs of $\tau_j$, and any coloring $\bar{b}$ of its internal edges, we can assign the following rewriting rule:
\begin{equation}
\label{eq::rewrite::caterpillar}
(\tau_{j})_{\bar{c},\bar{b}} \rightarrow [f_j]_{\bar{c}}^{\max}.
\end{equation}
In other words, we rewrite every colored shuffle monomial $(\tau_j)_{\bar{c},\bar{b}}$ as a sum of caterpillar monomials possessing the maximal number of straight internal edges. We denote the union of the "recoloring" rules from~\eqref{eq::rewrite::recolor} and the rules defined in~\eqref{eq::rewrite::caterpillar} as $\Theta_{\mathfrak{c}}(S)$.

We are now ready to state the main result of this section.

\begin{theorem}
\label{thm::cycl->diop}
Suppose that a convergent rewriting system $S$ for the shuffle operad associated with a cyclic operad $\calP$ is caterpillar. Then $\Theta_{\mathfrak{c}}(S)$ is a convergent caterpillar rewriting system for the dioperad $\Theta_{\mathfrak{c}}(\calP)$.
\end{theorem}

\begin{proof}
First, by Proposition~\ref{prp::cyclic::coloring}, the set of generators described in~\eqref{eq::colored::generators} generates the shuffle operad $\shuffle(\Psi(\Theta_{\mathfrak{c}}(\calF(E))))$. Second, it is straightforward to verify that these rewriting rules generate the ideal of relations for the dioperad.
Third, to prove termination, we observe that each recoloring rule~\eqref{eq::rewrite::recolor} increases the number of straight internal edges. Since the number of internal edges is finite for any given monomial, the recoloring process must terminate. The overall termination of $\Theta_{\mathfrak{c}}(S)$ then follows from the termination of the underlying system $S$.
Fourth, similarly, for confluence, it suffices to notice that after applying a sufficient number of recoloring rules~\eqref{eq::rewrite::recolor}, we only need to verify confluence for caterpillar monomials with the maximal number of straight edges. This confluence is equivalent to the confluence of the original rewriting system $S$.
\end{proof}

\begin{corollary}
\label{cor::coloring::Koszul}
If the shuffle operad associated with a cyclic operad $\calP$ admits a quadratic Gr\"obner basis with caterpillar normal forms (or, equivalently, a quadratic terminating and confluent caterpillar rewriting system), then the corresponding $2$-colored shuffle operad $\shuffle(\Psi(\Theta_{\mathfrak{c}}(\calP)))$ also admits a quadratic terminating and confluent rewriting system. Consequently, the dioperad $\Theta_{\mathfrak{c}}(\calP)$ is Koszul.
\end{corollary}

\begin{proof}
By Theorem~\ref{thm::cycl->diop}, the $2$-colored operad $\Psi(\Theta_{\mathfrak{c}}(\calP))$ admits a terminating and confluent rewriting system that governs the quadratic relations, provided the cyclic operad possesses a quadratic Gr\"obner basis. The Koszulness of the dioperad $\Theta_{\mathfrak{c}}(\calP)$ then follows directly from Corollary~\ref{cor::Grobner->Koszul}.
\end{proof}

\begin{corollary}
    The dioperads $\Theta_{\mathfrak{c}}(\Lie)$ and $\Theta_{\mathfrak{c}}(\Ass)$ described in Examples~\ref{ex::Lie::dioperad} and~\ref{ex::Assoc::dioperad} admits a quadratic convergent rewriting systems and, therefore, are Koszul dioperads. 
\end{corollary}

\begin{remark}
It is worth noting that the statements of Proposition~\ref{prp::cyclic::coloring} and Theorem~\ref{thm::cycl->diop} can be readily generalized to the various cases of $\mathfrak{c}$ discussed above. In particular, for most standard dioperadic colorings $\mathfrak{c}$, the Koszulness of the dioperad $\Theta_{\mathfrak{c}}(\calP)$ is ensured by the existence of a quadratic (caterpillar) Gr\"obner basis for the underlying cyclic operad $\calP$. 
\end{remark}

In the subsequent section, we discuss various applications of this result. Our examples include the dioperads of Frobenius algebras (\S\ref{example::Frobenius}), Lie bialgebras (\S\ref{example::Lie::Bialg}), quasi-Lie bialgebras, and pseudo-Lie bialgebras (\S\ref{sec::PseudoLie}). Notably, this framework provides a straightforward proof for the Koszulness of the dioperad of quadratic Poisson structures.

\section{Examples}
\label{sec::examples}

In this section, we describe Gr\"obner bases, convergent rewriting systems, and Hilbert series for several classes of dioperads that arise naturally in various areas of mathematics.

\subsection{The Frobenius dioperad $\Frob$}
\label{example::Frobenius}
Recall that a \emph{non-unital Frobenius algebra} is a triple $(V,\cdot,\Delta)$, where $\cdot$ defines a commutative associative multiplication and $\Delta$ a cocommutative coassociative comultiplication, such that for all $a, b \in V$:
$$
\Delta(a \cdot b) = a^{(1)} \otimes (a^{(2)} \cdot b) = (a \cdot b^{(1)}) \otimes b^{(2)}.
$$
Here, we employ Sweedler's notation $\Delta(a) := a^{(1)} \otimes a^{(2)}$ for the comultiplication map.

The corresponding dioperad (properad) of Frobenius algebras $\Frob$ is generated by a symmetric multiplication of arity $(2,1)$ and a symmetric comultiplication of arity $(1,2)$. These are depicted as trees with black vertices:
$$
\multPic{1}{2} = \multPic{2}{1}; \qquad \comultPic{1}{2} = \comultPic{2}{1}.
$$
The quadratic relations imply that the space of quadratic compositions is one-dimensional for all arities where a composition is possible:
$$
\left\{
\begin{array}{c}
    \mLiePic{1}{2}{3} = \mLiePic{2}{3}{1} = \mLiePic{3}{1}{2}; \qquad  
    \mcoLLiePic{1}{2}{3} = \mcoLLiePic{2}{3}{1} = \mcoLLiePic{3}{1}{2};  
    \\
    \mLiecoLiePic{1}{2}{1}{2} = \mLiecoLiePic{1}{2}{2}{1} = \mLiecoLiePic{2}{1}{1}{2} = \mLiecoLiePic{2}{1}{2}{1} = 
    \begin{tikzpicture}[x=1mm,y=1mm, baseline=-0.5ex]
        \node[draw,circle,inner sep=1pt,fill=black] (A) at (0,3) {};
        \draw (0,2.47) -- (0,0.12);                  
        \draw (0.5,3.5) -- (2.2,5.2);                
        \draw (-0.48,3.48) -- (-2.2,5.2);            
        \node[draw,circle,inner sep=1pt,fill=black] (B) at (0,-0.8) {};
        \draw (-0.39,-1.2) -- (-2.2,-3.5);           
        \draw (0.39,-1.2) -- (2.2,-3.5);             
        \node at (2.8,5.7) {\scriptsize $2$};
        \node at (-2.8,5.7) {\scriptsize $1$};
        \node at (-2.7,-5.2) {\scriptsize $1$};
        \node at (2.7,-5.2) {\scriptsize $2$};
    \end{tikzpicture}
\end{array}
\right\}
$$
Consequently, it follows that $\dim \Frob(m,n)=1$ for all $m,n \geq 1$.

\begin{remark}
The Frobenius dioperad is isomorphic to the coloring $\Theta_{\mathfrak{c}}(\Com)$ for $\mathfrak{c} := \bZ_{>0} \times \bZ_{>0}$.
In particular, since the commutative operad admits a quadratic Gr\"obner basis with caterpillar normal forms, the Frobenius dioperad likewise admits a quadratic, convergent rewriting system. Below, we explain that this rewriting system originates from an appropriate quadratic Gr\"obner basis.
\end{remark}

The $2$-colored shuffle operad $\shuffle(\Psi(\Frob))$ possesses six binary generators with distinct input/output colorings, which we order as follows:
$$
\begin{tikzpicture}[scale=0.15pt, baseline=-0.5ex]
    \draw (0,-0.55) -- (0,-2.5);
    \draw (0.5,0.5) -- (2.2,2.2);
    \draw (-0.48,0.48) -- (-2.2,2.2);
    \node[circle,draw,inner sep=1.5pt] (A) at (0,0)   {$\scriptscriptstyle \alpha$};
    \node at (-2.7,2.8) {$\scriptscriptstyle 1$};
    \node at (2.7,2.8) {$\scriptscriptstyle 2$};
\end{tikzpicture}
<  
\begin{tikzpicture}[scale=0.15pt, baseline=-0.5ex]
    \draw[dotted] (0,-0.55) -- (0,-2.5);
    \draw (0.5,0.5) -- (2.2,2.2);
    \draw[dotted] (-0.48,0.48) -- (-2.2,2.2);
    \node[circle,draw,inner sep=1.5pt] (A) at (0,0)   {$\scriptscriptstyle \beta$};
    \node at (-2.7,2.8) {$\scriptscriptstyle 1$};
    \node at (2.7,2.8) {$\scriptscriptstyle 2$};        
\end{tikzpicture}
<
\begin{tikzpicture}[scale=0.15pt, baseline=-0.5ex]
    \draw (0,-0.55) -- (0,-2.5);
    \draw (0.5,0.5) -- (2.2,2.2);
    \draw[dotted] (-0.48,0.48) -- (-2.2,2.2);
    \node[circle,draw,inner sep=1.5pt] (A) at (0,0)   {$\scriptscriptstyle \gamma$};
    \node at (-2.7,2.8) {$\scriptscriptstyle 1$};
    \node at (2.7,2.8) {$\scriptscriptstyle 2$};
\end{tikzpicture}
<
\begin{tikzpicture}[scale=0.15pt, baseline=-0.5ex]
    \draw[dotted] (0,-0.55) -- (0,-2.5);
    \draw[dotted] (0.5,0.5) -- (2.2,2.2);
    \draw (-0.48,0.48) -- (-2.2,2.2);
    \node[circle,draw,inner sep=1.5pt] (A) at (0,0)   {$\scriptscriptstyle \delta$};
    \node at (-2.7,2.8) {$\scriptscriptstyle 1$};
    \node at (2.7,2.8) {$\scriptscriptstyle 2$};
\end{tikzpicture}
< 
\begin{tikzpicture}[scale=0.15pt, baseline=-0.5ex]
    \draw (0,-0.55) -- (0,-2.5);
    \draw[dotted] (0.5,0.5) -- (2.2,2.2);
    \draw (-0.48,0.48) -- (-2.2,2.2);
    \node[circle,draw,inner sep=1.5pt] (A) at (0,0)   {$\scriptscriptstyle \varepsilon$};
    \node at (-2.7,2.8) {$\scriptscriptstyle 1$};
    \node at (2.7,2.8) {$\scriptscriptstyle 2$};
\end{tikzpicture}
<
\begin{tikzpicture}[scale=0.15pt, baseline=-0.5ex]
    \draw[dotted] (0,-0.55) -- (0,-2.5);
    \draw[dotted] (0.5,0.5) -- (2.2,2.2);
    \draw[dotted] (-0.48,0.48) -- (-2.2,2.2);
    \node[circle,draw,inner sep=1.5pt] (A) at (0,0)   {$\scriptscriptstyle \zeta$};
    \node at (-2.7,2.8) {$\scriptscriptstyle 1$};
    \node at (2.7,2.8) {$\scriptscriptstyle 2$};
\end{tikzpicture}
$$
The relations reflect the property that a composition of two operations, if it exists, is uniquely determined by the coloring of the inputs and outputs. For the path-lexicographical ordering, the set of normal forms of degree $2$ is similarly determined by the input/output coloring. These normal forms consist of right combs (right-growing trees), where internal edges are straight whenever such a coloring is admissible:
\begin{multline}
\label{eq::Com::normal::forms}
    \begin{tikzpicture}[x=1.5mm,y=1.5mm, xscale=-1, baseline=-0.5ex]	
        \node[draw,circle,inner sep=1pt] (A) at (0,0) {$\scriptscriptstyle \alpha$};
        \draw (0,-0.49) -- (0,-3.0);
        \draw (0.49,0.49) -- (1.9,1.9);
        \draw (-0.5,0.5) -- (-1.9,1.9);	
        \node[draw,circle,inner sep=1pt] (B) at (-2.3,2.3) {$\scriptscriptstyle \alpha$};
        \draw (-1.8,2.8) -- (0,4.9);
        \draw (-2.8,2.9) -- (-4.6,4.9);	
        \node at (2.7,2.3) {\scriptsize $1$};
        \node at (0.4,5.3) {\scriptsize $2$};
        \node at (-5.1,5.3) {\scriptsize $3$};
    \end{tikzpicture}, \
    \begin{tikzpicture}[x=1.5mm,y=1.5mm, xscale=-1, baseline=-0.5ex]	
        \node[draw,circle,inner sep=1pt] (A) at (0,0) {$\scriptscriptstyle \alpha$};
        \draw (0,-0.49) -- (0,-3.0);
        \draw (0.49,0.49) -- (1.9,1.9);
        \draw (-0.5,0.5) -- (-1.9,1.9);	
        \node[draw,circle,inner sep=1pt] (B) at (-2.3,2.3) {$\scriptscriptstyle \epsilon$};
        \draw (-1.8,2.8) -- (0,4.9);
        \draw[dotted] (-2.8,2.9) -- (-4.6,4.9);	
        \node at (2.7,2.3) {\scriptsize $1$};
        \node at (0.4,5.3) {\scriptsize $2$};
        \node at (-5.1,5.3) {\scriptsize $3$};
    \end{tikzpicture}, \
    \begin{tikzpicture}[x=1.5mm,y=1.5mm, xscale=-1, baseline=-0.5ex]	
        \node[draw,circle,inner sep=1pt] (A) at (0,0) {$\scriptscriptstyle \alpha$};
        \draw (0,-0.49) -- (0,-3.0);
        \draw (0.49,0.49) -- (1.9,1.9);
        \draw (-0.5,0.5) -- (-1.9,1.9);	
        \node[draw,circle,inner sep=1pt] (B) at (-2.3,2.3) {$\scriptscriptstyle \gamma$};
        \draw[dotted] (-1.8,2.8) -- (0,4.9);
        \draw (-2.8,2.9) -- (-4.6,4.9);	
        \node at (2.7,2.3) {\scriptsize $1$};
        \node at (0.4,5.3) {\scriptsize $2$};
        \node at (-5.1,5.3) {\scriptsize $3$};
    \end{tikzpicture}, \
    \begin{tikzpicture}[x=1.5mm,y=1.5mm, xscale=-1, baseline=-0.5ex]	
        \node[draw,circle,inner sep=1pt] (A) at (0,0) {$\scriptscriptstyle \gamma$};
        \draw (0,-0.49) -- (0,-3.0);
        \draw[dotted] (0.49,0.49) -- (1.9,1.9);
        \draw (-0.5,0.5) -- (-1.9,1.9);	
        \node[draw,circle,inner sep=1pt] (B) at (-2.3,2.3) {$\scriptscriptstyle \alpha$};
        \draw (-1.8,2.8) -- (0,4.9);
        \draw (-2.8,2.9) -- (-4.6,4.9);	
        \node at (2.7,2.3) {\scriptsize $1$};
        \node at (0.4,5.3) {\scriptsize $2$};
        \node at (-5.1,5.3) {\scriptsize $3$};
    \end{tikzpicture}, \
    \begin{tikzpicture}[x=1.5mm,y=1.5mm, xscale=-1, baseline=-0.5ex]	
        \node[draw,circle,inner sep=1pt] (A) at (0,0) {$\scriptscriptstyle \gamma$};
        \draw (0,-0.49) -- (0,-3.0);
        \draw[dotted] (0.49,0.49) -- (1.9,1.9);
        \draw (-0.5,0.5) -- (-1.9,1.9);	
        \node[draw,circle,inner sep=1pt] (B) at (-2.3,2.3) {$\scriptscriptstyle \varepsilon$};
        \draw (-1.8,2.8) -- (0,4.9);
        \draw[dotted] (-2.8,2.9) -- (-4.6,4.9);	
        \node at (2.7,2.3) {\scriptsize $1$};
        \node at (0.4,5.3) {\scriptsize $2$};
        \node at (-5.1,5.3) {\scriptsize $3$};
    \end{tikzpicture}, \
    \begin{tikzpicture}[x=1.5mm,y=1.5mm, xscale=-1, baseline=-0.5ex]	
        \node[draw,circle,inner sep=1pt] (A) at (0,0) {$\scriptscriptstyle \gamma$};
        \draw (0,-0.49) -- (0,-3.0);
        \draw[dotted] (0.49,0.49) -- (1.9,1.9);
        \draw (-0.5,0.5) -- (-1.9,1.9);	
        \node[draw,circle,inner sep=1pt] (B) at (-2.3,2.3) {$\scriptscriptstyle \gamma$};
        \draw[dotted] (-1.8,2.8) -- (0,4.9);
        \draw (-2.8,2.9) -- (-4.6,4.9);	
        \node at (2.7,2.3) {\scriptsize $1$};
        \node at (0.4,5.3) {\scriptsize $2$};
        \node at (-5.1,5.3) {\scriptsize $3$};
    \end{tikzpicture}, \ 
    \begin{tikzpicture}[x=1.5mm,y=1.5mm, xscale=-1, baseline=-0.5ex]	
        \node[draw,circle,inner sep=1pt] (A) at (0,0) {$\scriptscriptstyle \varepsilon$};
        \draw (0,-0.49) -- (0,-3.0);
        \draw (0.49,0.49) -- (1.9,1.9);
        \draw[dotted] (-0.5,0.5) -- (-1.9,1.9);	
        \node[draw,circle,inner sep=1pt] (B) at (-2.3,2.3) {$\scriptscriptstyle \zeta$};
        \draw[dotted] (-1.8,2.8) -- (0,4.9);
        \draw[dotted] (-2.8,2.9) -- (-4.6,4.9);	
        \node at (2.7,2.3) {\scriptsize $1$};
        \node at (0.4,5.3) {\scriptsize $2$};
        \node at (-5.1,5.3) {\scriptsize $3$};
    \end{tikzpicture}; \
    \qquad \\ \quad
    \begin{tikzpicture}[x=1.5mm,y=1.5mm, xscale=-1, baseline=-0.5ex]	
        \node[draw,circle,inner sep=1pt] (A) at (0,0) {$\scriptscriptstyle \delta$};
        \draw[dotted] (0,-0.49) -- (0,-3.0);
        \draw (0.49,0.49) -- (1.9,1.9);
        \draw[dotted] (-0.5,0.5) -- (-1.9,1.9);	
        \node[draw,circle,inner sep=1pt] (B) at (-2.3,2.3) {$\scriptscriptstyle \delta$};
        \draw (-1.8,2.8) -- (0,4.9);
        \draw[dotted] (-2.8,2.9) -- (-4.6,4.9);	
        \node at (2.7,2.3) {\scriptsize $1$};
        \node at (0.4,5.3) {\scriptsize $2$};
        \node at (-5.1,5.3) {\scriptsize $3$};
    \end{tikzpicture}, \
    \begin{tikzpicture}[x=1.5mm,y=1.5mm, xscale=-1, baseline=-0.5ex]	
        \node[draw,circle,inner sep=1pt] (A) at (0,0) {$\scriptscriptstyle \delta$};
        \draw[dotted] (0,-0.49) -- (0,-3.0);
        \draw (0.49,0.49) -- (1.9,1.9);
        \draw[dotted] (-0.5,0.5) -- (-1.9,1.9);	
        \node[draw,circle,inner sep=1pt] (B) at (-2.3,2.3) {$\scriptscriptstyle \beta$};
        \draw[dotted] (-1.8,2.8) -- (0,4.9);
        \draw (-2.8,2.9) -- (-4.6,4.9);	
        \node at (2.7,2.3) {\scriptsize $1$};
        \node at (0.4,5.3) {\scriptsize $2$};
        \node at (-5.1,5.3) {\scriptsize $3$};
    \end{tikzpicture}, \
    \begin{tikzpicture}[x=1.5mm,y=1.5mm, xscale=-1, baseline=-0.5ex]	
        \node[draw,circle,inner sep=1pt] (A) at (0,0) {$\scriptscriptstyle \beta$};
        \draw[dotted] (0,-0.49) -- (0,-3.0);
        \draw[dotted] (0.49,0.49) -- (1.9,1.9);
        \draw (-0.5,0.5) -- (-1.9,1.9);	
        \node[draw,circle,inner sep=1pt] (B) at (-2.3,2.3) {$\scriptscriptstyle \alpha$};
        \draw (-1.8,2.8) -- (0,4.9);
        \draw (-2.8,2.9) -- (-4.6,4.9);	
        \node at (2.7,2.3) {\scriptsize $1$};
        \node at (0.4,5.3) {\scriptsize $2$};
        \node at (-5.1,5.3) {\scriptsize $3$};
    \end{tikzpicture}, \
    \begin{tikzpicture}[x=1.5mm,y=1.5mm, xscale=-1, baseline=-0.5ex]	
        \node[draw,circle,inner sep=1pt] (A) at (0,0) {$\scriptscriptstyle \delta$};
        \draw[dotted] (0,-0.49) -- (0,-3.0);
        \draw (0.49,0.49) -- (1.9,1.9);
        \draw[dotted] (-0.5,0.5) -- (-1.9,1.9);	
        \node[draw,circle,inner sep=1pt] (B) at (-2.3,2.3) {$\scriptscriptstyle \zeta$};
        \draw[dotted] (-1.8,2.8) -- (0,4.9);
        \draw[dotted] (-2.8,2.9) -- (-4.6,4.9);	
        \node at (2.7,2.3) {\scriptsize $1$};
        \node at (0.4,5.3) {\scriptsize $2$};
        \node at (-5.1,5.3) {\scriptsize $3$};
    \end{tikzpicture}, \
    \begin{tikzpicture}[x=1.5mm,y=1.5mm, xscale=-1, baseline=-0.5ex]	
        \node[draw,circle,inner sep=1pt] (A) at (0,0) {$\scriptscriptstyle \beta$};
        \draw[dotted] (0,-0.49) -- (0,-3.0);
        \draw[dotted] (0.49,0.49) -- (1.9,1.9);
        \draw (-0.5,0.5) -- (-1.9,1.9);	
        \node[draw,circle,inner sep=1pt] (B) at (-2.3,2.3) {$\scriptscriptstyle \varepsilon$};
        \draw (-1.8,2.8) -- (0,4.9);
        \draw[dotted] (-2.8,2.9) -- (-4.6,4.9);	
        \node at (2.7,2.3) {\scriptsize $1$};
        \node at (0.4,5.3) {\scriptsize $2$};
        \node at (-5.1,5.3) {\scriptsize $3$};
    \end{tikzpicture}, \
    \begin{tikzpicture}[x=1.5mm,y=1.5mm, xscale=-1, baseline=-0.5ex]	
        \node[draw,circle,inner sep=1pt] (A) at (0,0) {$\scriptscriptstyle \beta$};
        \draw[dotted] (0,-0.49) -- (0,-3.0);
        \draw[dotted] (0.49,0.49) -- (1.9,1.9);
        \draw (-0.5,0.5) -- (-1.9,1.9);	
        \node[draw,circle,inner sep=1pt] (B) at (-2.3,2.3) {$\scriptscriptstyle \gamma$};
        \draw[dotted] (-1.8,2.8) -- (0,4.9);
        \draw (-2.8,2.9) -- (-4.6,4.9);	
        \node at (2.7,2.3) {\scriptsize $1$};
        \node at (0.4,5.3) {\scriptsize $2$};
        \node at (-5.1,5.3) {\scriptsize $3$};
    \end{tikzpicture}, \      
    \begin{tikzpicture}[x=1.5mm,y=1.5mm, xscale=-1, baseline=-0.5ex]	
        \node[draw,circle,inner sep=1pt] (A) at (0,0) {$\scriptscriptstyle \zeta$};
        \draw[dotted] (0,-0.49) -- (0,-3.0);
        \draw[dotted] (0.49,0.49) -- (1.9,1.9);
        \draw[dotted] (-0.5,0.5) -- (-1.9,1.9);	
        \node[draw,circle,inner sep=1pt] (B) at (-2.3,2.3) {$\scriptscriptstyle \zeta$};
        \draw[dotted] (-1.8,2.8) -- (0,4.9);
        \draw[dotted] (-2.8,2.9) -- (-4.6,4.9);	
        \node at (2.7,2.3) {\scriptsize $1$};
        \node at (0.4,5.3) {\scriptsize $2$};
        \node at (-5.1,5.3) {\scriptsize $3$};
    \end{tikzpicture}.
\end{multline}

Note that while there are $8$ possible ways to color the inputs and output of a binary corolla, and $16$ for a corolla with $3$ inputs, we observe only $6$ binary generators and $14$ quadratic compositions. The "forbidden" operations correspond to cases where all inputs share one color, but the output is colored differently. (A more comprehensive discussion of the remaining generators and operations is provided in \S\ref{sec::PseudoLie}, within the context of quasi- and pseudo-Lie bialgebras.)

It follows that for any permitted coloring, there exists exactly one normal form (a tree monomial that is not divisible by the leading term of any relation). This normal form is a right-comb, in which all internal edges are colored straight whenever possible. This normal form is non-zero as it represents the unique non-zero operation in $\Frob(m,n)$.

\subsection{The dioperad $\LieBi$ of Lie bialgebras}
\label{example::Lie::Bialg}
Recall that a \textit{Lie-bialgebra} is a vector space $V$ together with two operations 
$$[\,,]:V\wedge V\rightarrow V \text{ and } \delta:V\rightarrow V\wedge V.$$ 
Where $[\,,]$ is a Lie-bracket, $\delta$ is a Lie co-bracket  and the operations satisfy the relation
	\begin{equation*}
		\delta([a,b])=(\ad_a\otimes 1 + 1\otimes\ad_a)\delta(b)-(\ad_b\otimes 1 + 1\otimes\ad_b)\delta(a)
	\end{equation*}
	for any $a,b\in V$ and $\ad_a(b)=[a,b]$.
    
The corresponding dioperad $\LB$ of Lie bialgebras has the following pictorial description.
 $\LB$ is generated by an  $\bS\times\bS^{\op}$-bimodule $E=\{E(m,n)\}_{m,n\geq 1}$ with
	all $E(m,n)=0$ except
	$$
	\begin{array}{c}
		{
			E(2,1):=
\mbox{span}\left\langle
			\liePic{1}{2}
			=(-1)
			\liePic{2}{1}
			\right\rangle
		},
\quad
        {
			E(1,2):=
			\mbox{span}\left\langle
			\coliePic{1}{2} = (-1)
            \coliePic{2}{1}
			\right\rangle
		}
	\end{array}
	$$
	by the ideal generated by the following relations
	\begin{equation}
		\label{R for LieB}
		\left\{
		\begin{array}{c}
			\LLiePic{1}{2}{3} + 	\LLiePic{2}{3}{1} + \LLiePic{3}{1}{2} \quad = \quad
			\coLLiePic{1}{2}{3} + 	\coLLiePic{2}{3}{1} + \coLLiePic{3}{1}{2} = 0; \\ 
			-
			\begin{tikzpicture}[x=1mm,y=1mm]
				\node[draw,circle,inner sep=1pt,fill=white] (A) at (0,3) {};
				\draw (0,2.47) -- (0,0.12);                  
				\draw (0.5,3.5) -- (2.2,5.2);                
				\draw (-0.48,3.48) -- (-2.2,5.2);            
				\node[draw,circle,inner sep=1pt,fill=white] (B) at (0,-0.8) {};
				\draw (-0.39,-1.2) -- (-2.2,-3.5);           
				\draw (0.39,-1.2) -- (2.2,-3.5);             
				\node at (2.8,5.7) {\scriptsize $2$};
				\node at (-2.8,5.7) {\scriptsize $1$};
				\node at (-2.7,-5.2) {\scriptsize $1$};
				\node at (2.7,-5.2) {\scriptsize $2$};
			\end{tikzpicture}
			- \LiecoLiePic{1}{2}{1}{2} - 
            \LiecoLiePic{2}{1}{1}{2} + 
            \LiecoLiePic{2}{1}{2}{1} - 
            \LiecoLiePic{1}{2}{2}{1} = 0
		\end{array}
		\right\}
	\end{equation}

The dioperad $\LieBi$ is Koszul dual to the dioperad $\Frob$ of Frobenius algebras (see e.g.~\cite{Markl_Voronov}).
This follows that $\LieBi$ admits a quadratic Gr\"obner basis with respect to the opposite ordering of monomials.
In particular, the quadratic right-comb monomials~\eqref{eq::Com::normal::forms} assemble the set of leading terms of this quadratic Gr\"obner basis.

\begin{proposition}
\label{prp::Liebi::genseries}
The generating series of the dioperad of Lie bialgebras is given by:
$$
\rchi_{\LieBi}(u,v):= \sum_{m,n\geq 1} \dim\LieBi(m,n) \frac{u^m}{m!} \frac{v^n}{n!} = 
-\int \ln\left(\frac{1-u+v+ \sqrt{1-2(u+v)+(u-v)^2}}{2}\right) \mathrm{d}v.
$$ 
In particular, the dimensions of the spaces $\LieBi(m,n)$ are:
$$
\dim\LieBi(m,n) = (n+m-2)! \binom{n+m-2}{m-1} = \frac{(m+n-2)!^2}{(m-1)!(n-1)!}.
$$
\end{proposition}
\begin{proof}
The dioperad of Lie bialgebras is Koszul, and its Koszul dual is the dioperad of Frobenius algebras $\Frob$. The generating series for $\Frob$ is particularly simple because for every $m \geq 1$ and $n \geq 1$, there exists exactly one operation (up to a scalar) with $m$ inputs and $n$ outputs. Consequently, the generating series is:
$$
\rchi_{\Frob}(x,y):= \sum_{m,n\geq 1} \dim\Frob(m,n) \frac{x^m}{m!} \frac{y^n}{n!} = \sum_{m,n\geq 1} \frac{x^m}{m!} \frac{y^n}{n!} = (e^x-1)(e^y-1). 
$$
The generating series of the associated colored operad $\Psi(\Frob)$ consists of two components:
$$
\overline{\rchi_{\Psi(\Frob)}} = \left(\frac{\partial \rchi_{\Frob}(x,y)}{\partial y}, \frac{\partial \rchi_{\Frob}(x,y)}{\partial x} \right) = \left(e^{x+y}-e^y, e^{x+y}-e^x\right).
$$
Since all generators of $\Psi(\Frob)$ are binary, the $q$-grading is uniquely determined by the arity of the operations. Thus, we have the following functional equation for the generating series of the Koszul dual:
$$
\overline{\rchi_{\Psi(\LieBi)}(-x,-y)} \circ \overline{\rchi_{\Psi(\Frob)}(-x,-y)} = (x,y).
$$
Setting $u = \frac{\partial \rchi_{\LieBi}}{\partial v}$ and $v = \frac{\partial \rchi_{\LieBi}}{\partial u}$, we solve the functional equation system:
$$
(u,v) = - \overline{\rchi_{\Psi(\Frob)}(-x,-y)} \ \Leftrightarrow \
\begin{cases}
    u = e^{-y} - e^{-x-y} \\
    v = e^{-x} - e^{-x-y}
\end{cases}
\ \Leftrightarrow \ 
\begin{cases}
    u-v = e^{-y} - e^{-x} \\
    v = e^{-x} - e^{-x}e^{-y}
\end{cases}
$$
Substituting $e^{-y} = u-v + e^{-x}$ into the second equation:
$$
v = e^{-x} - e^{-x}(u-v + e^{-x}) \implies e^{-2x} - e^{-x}(1-u+v) + v = 0.
$$
By the quadratic formula:
$$
e^{-x} = \frac{1-u+v \pm \sqrt{(1-u+v)^2 - 4v}}{2}.
$$
Since we are working with power series where $x \to 0$ as $u, v \to 0$, we must have $e^{-x} \to 1$. Only the positive root satisfies this boundary condition. Thus, we obtain:
$$
x = -\ln\left(\frac{1- u+v + \sqrt{1-2(u+v)+(u-v)^2}}{2}\right).
$$
From the functional equation, we identify $x = \frac{\partial\rchi_{\LieBi}(u,v)}{\partial v}$.

Let $z = 1 - e^{-x}$. Then we have the following equality:
$$
\frac{z}{u}= \frac{1-e^{-x}}{u} = e^{y} = \frac{e^{-x}}{e^{-x-y}} = \frac{1-z}{1-v-z} = \Bigl(1- \frac{v}{1-z}\Bigr)^{-1}.
$$
Let $H(z) := -\ln(1-z) =x$. We apply the Lagrange Inversion Theorem (LIT) with respect to $z$, treating $v$ as a parametr:
\begin{multline*}
[u^m] \frac{\partial\rchi_{\LieBi}(u,v)}{\partial v} = [u^m] x(u,v) = [u^m] H(z) \stackrel{\text{LIT}}{=} 
\\ \stackrel{\text{LIT}}{=}
\frac{1}{m} [z^{m-1}] \left( H'(z) \Bigl(\frac{z}{u}\Bigr)^m \right)  
= \frac{1}{m} [z^{m-1}] \left( \frac{1}{1-z} \cdot \left( 1- \frac{v}{1-z} \right)^{-m} \right).
\end{multline*}
Now we can recall that $v$ is also a formal variable and compute the coefficient near $v^n$:
\begin{multline*}
 [v^n][u^m] x(u,v) = [v^n] \left( \frac{1}{m} [z^{m-1}]  \left( \frac{1}{1-z} \cdot \left( 1- \frac{v}{1-z} \right)^{-m} \right) \right) = \\
= \frac{1}{m}[z^{m-1}]\Bigl(\frac{1}{1-z} \cdot 
 [v^n]\left( 1- \frac{v}{1-z} \right)^{-m} 
 \Bigr) 
 = \frac{1}{m}[z^{m-1}]\left(\frac{1}{1-z} \cdot \left( (1-z)^{-n} \binom{n+m-1}{m-1} \right) \right) = 
\\
 = \frac{1}{m} \binom{n+m-1}{m} [z^{m-1}](1-z)^{-n-1} 
 = \frac{1}{m} \binom{n+m-1}{m-1} \binom{n+m-1}{m-1} = \frac{(n+m-1)! \binom{n+m-1}{m-1}}{m!n!}
\end{multline*}
Consequently, we have:
$$
\dim \LieBi(m,n) = \dim \Psi(\LieBi)^{\strt}(m,n-1) = 
(m+n-2)! \binom{m+n-2}{m-1} =  \frac{(m+n-2)!^2}{(m-1)!(n-1)!}.
$$
We refer to~\cite[\S6.2]{Stanley} for details on the Lagrange Inversion Theorem.
\end{proof}

\subsection{Quasi-Lie bialgebras and Pseudo-Lie bialgebras}
\label{sec::PseudoLie}
While considering different algebraic structures related to quantum groups, Drinfeld also defined different generalizations of the notion of Lie bialgebras by adding extra structures to them (see e.g.~\cite{Drinfeld_Quasi}).
Let us recall some of them
\begin{definition}
\begin{itemize}[itemsep=0pt, topsep=0pt]
\item A \emph{pseudo-Lie bialgebra} is a 5-tuple 
 $(V;[-,-],\delta,\phi,\eta)$ where
$$
\delta\in\mathrm{Hom}(V,\Lambda^2 V), \  [-,-]\in\mathrm{Hom}(\Lambda^2 V,V) \ \text{ and } \phi\in\mathrm{Hom}(\Bbbk,\Lambda^3 V), \   \eta\in\mathrm{Hom}(\Lambda^3 V,\Bbbk)$$ 
such that 
	\begin{gather*}
		\delta[a,b]=(\ad_a\otimes 1 + 1\otimes\ad_a)\delta(b)-(\ad_b\otimes 1 + 1\otimes\ad_b)\delta(a) + \phi^{(1)}\otimes\phi^{(2)} (\eta(\phi^{(3)}\otimes a\otimes b));\\
		\dfrac{1}{2}\mathrm{Alt}_3(\delta\otimes\mathrm{id})(\delta(a))=[a\otimes1\otimes1+1\otimes1\otimes1+1\otimes1\otimes a,\phi];\\
        \begin{array}{rl}
		[[a,b],c]+[[c,a],b]+ [[b,c],a] & =   \\ =\eta(a,b,\delta^{(1)}(c))\delta^{(2)}(c)+& \eta(c,a,\delta^{(1)}(b))\delta^{(2)}(b)+\eta(b,c,\delta^{(1)}(a))\delta^{(2)}(a);
        \end{array}
        \\
\mathrm{Alt}_4(\delta\otimes\mathrm{id}\otimes\mathrm{id})(\phi) = 0;\\
	\begin{array}{rl}	\eta([x,y],z,w)+\eta([x,z],y,w)& +\eta([x,w],y,z) + \\
		 \ \ \ \qquad +\eta([y,z],x,w) & +\eta([y,w],x,z)+\eta([z,w],x,y) = 0.
    \end{array}     
	\end{gather*}
	where $\mathrm{Alt}_k:V^{\otimes k}\rightarrow V^{\otimes k}$ denotes the operator $\sum_{\sigma\in\mathbb{S}_k}\mathrm{sgn}(\sigma)\sigma$.
\item    A \emph{quasi-Lie bialgebra} is a pseudo-Lie algebra where $\eta=0$.
\item When $\eta=0$ and $\phi=0$, the definition reduces to that of a Lie bi-algebra $(V,[\ttt,\ttt],\delta)$.
\end{itemize}
\end{definition}

The corresponding dioperad $\mathcal{PL}ieb_n$ is generated by the following skew-symmetric operations:
      \begin{equation*}
   	E(m,n)=
   	\begin{cases}
   		\begin{tikzpicture}[baseline=-0.3cm, scale =0.6]
   			\node[invisible] (c) at (0,0) {};
   			\node[invisible] at (0,-0.4) {}
   			edge [-] (c);
   			\node[invisible] at (0.4,-0.4) {}
   			edge [-] (c);
   			\node[invisible] at (-0.4,-0.4) {}
   			edge [-] (c);
   		\end{tikzpicture}
   		& (m,n)=(0,3)\\
   		\begin{tikzpicture}[baseline=-0.1cm,scale =0.6]
   			\node[invisible] (c) at (0,0) {};
   			\node[invisible] at (0,0.4) {}
   			edge [-] (c);
   			\node[invisible] at (0.4,-0.4) {}
   			edge [-] (c);
   			\node[invisible] at (-0.4,-0.4) {}
   			edge [-] (c);
   		\end{tikzpicture}
   		& (m,n)=(1,2)\\
   		\begin{tikzpicture}[baseline=0.1cm, scale = 0.5]
   			\node[invisible] (c) at (0,0) {};
   			\node[invisible] at (0,-0.4) {}
   			edge [-] (c);
   			\node[invisible] at (0.4,0.4) {}
   			edge [-] (c);
   			\node[invisible] at (-0.4,0.4) {}
   			edge [-] (c);
   		\end{tikzpicture}
   		& (m,n)=(2,1)\\
   		\begin{tikzpicture}[baseline=0.1cm, scale =0.7]
   			\node[invisible] (c) at (0,0) {};
   			\node[invisible] at (0,0.4) {}
   			edge [-] (c);
   			\node[invisible] at (0.4,0.4) {}
   			edge [-] (c);
   			\node[invisible] at (-0.4,0.4) {}
   			edge [-] (c);
   		\end{tikzpicture}
   		& (m,n)=(3,0)
   	\end{cases}
   \end{equation*}
   modulo the ideal generated by
   \begin{equation*}
   	\mathcal{R}(m,n)=
   	\begin{cases}
   		\begin{tikzpicture}[baseline=-0.5cm]
   			\node[invisible] (c) at (0,0) {};
   			\node[invisible] at (0,-0.25) {}
   			edge [-] (c);
   			\node[invisible] at (0.25,-0.25) {}
   			edge [-] (c);
   			\node[invisible] (l) at (-0.25,-0.25) {}
   			edge [-] (c);
   			\node[invisible] at (-0.15,-0.55) {}
   			edge [-] (l);
   			\node[invisible] at (-0.35,-0.55) {}
   			edge [-] (l);
   			\node[] at (-0.35,-0.7) {$\scriptstyle 1$};
   			\node[] at (-0.15,-0.7) {$\scriptstyle 2$};
   			\node[] at (0,-0.4) {$\scriptstyle 3$};
   			\node[] at (0.25,-0.4) {$\scriptstyle 4$};
   		\end{tikzpicture}
   		+
   		\begin{tikzpicture}[baseline=-0.5cm]
   			\node[invisible] (c) at (0,0) {};
   			\node[invisible] at (0,-0.25) {}
   			edge [-] (c);
   			\node[invisible] at (0.25,-0.25) {}
   			edge [-] (c);
   			\node[invisible] (l) at (-0.25,-0.25) {}
   			edge [-] (c);
   			\node[invisible] at (-0.15,-0.55) {}
   			edge [-] (l);
   			\node[invisible] at (-0.35,-0.55) {}
   			edge [-] (l);
   			\node[] at (-0.35,-0.7) {$\scriptstyle 1$};
   			\node[] at (-0.15,-0.7) {$\scriptstyle 3$};
   			\node[] at (0,-0.4) {$\scriptstyle 2$};
   			\node[] at (0.25,-0.4) {$\scriptstyle 4$};
   		\end{tikzpicture}
   		+
   		\begin{tikzpicture}[baseline=-0.5cm]
   			\node[invisible] (c) at (0,0) {};
   			\node[invisible] at (0,-0.25) {}
   			edge [-] (c);
   			\node[invisible] at (0.25,-0.25) {}
   			edge [-] (c);
   			\node[invisible] (l) at (-0.25,-0.25) {}
   			edge [-] (c);
   			\node[invisible] at (-0.15,-0.55) {}
   			edge [-] (l);
   			\node[invisible] at (-0.35,-0.55) {}
   			edge [-] (l);
   			\node[] at (-0.35,-0.7) {$\scriptstyle 1$};
   			\node[] at (-0.15,-0.7) {$\scriptstyle 4$};
   			\node[] at (0,-0.4) {$\scriptstyle 2$};
   			\node[] at (0.25,-0.4) {$\scriptstyle 3$};
   		\end{tikzpicture}
   		+
   		\begin{tikzpicture}[baseline=-0.5cm]
   			\node[invisible] (c) at (0,0) {};
   			\node[invisible] at (0,-0.25) {}
   			edge [-] (c);
   			\node[invisible] at (0.25,-0.25) {}
   			edge [-] (c);
   			\node[invisible] (l) at (-0.25,-0.25) {}
   			edge [-] (c);
   			\node[invisible] at (-0.15,-0.55) {}
   			edge [-] (l);
   			\node[invisible] at (-0.35,-0.55) {}
   			edge [-] (l);
   			\node[] at (-0.35,-0.7) {$\scriptstyle 2$};
   			\node[] at (-0.15,-0.7) {$\scriptstyle 3$};
   			\node[] at (0,-0.4) {$\scriptstyle 1$};
   			\node[] at (0.25,-0.4) {$\scriptstyle 4$};
   		\end{tikzpicture}
   		+
   		\begin{tikzpicture}[baseline=-0.5cm]
   			\node[invisible] (c) at (0,0) {};
   			\node[invisible] at (0,-0.25) {}
   			edge [-] (c);
   			\node[invisible] at (0.25,-0.25) {}
   			edge [-] (c);
   			\node[invisible] (l) at (-0.25,-0.25) {}
   			edge [-] (c);
   			\node[invisible] at (-0.15,-0.55) {}
   			edge [-] (l);
   			\node[invisible] at (-0.35,-0.55) {}
   			edge [-] (l);
   			\node[] at (-0.35,-0.7) {$\scriptstyle 2$};
   			\node[] at (-0.15,-0.7) {$\scriptstyle 4$};
   			\node[] at (0,-0.4) {$\scriptstyle 1$};
   			\node[] at (0.25,-0.4) {$\scriptstyle 3$};
   		\end{tikzpicture}
   		+
   		\begin{tikzpicture}[baseline=-0.5cm]
   			\node[invisible] (c) at (0,0) {};
   			\node[invisible] at (0,-0.25) {}
   			edge [-] (c);
   			\node[invisible] at (0.25,-0.25) {}
   			edge [-] (c);
   			\node[invisible] (l) at (-0.25,-0.25) {}
   			edge [-] (c);
   			\node[invisible] at (-0.15,-0.55) {}
   			edge [-] (l);
   			\node[invisible] at (-0.35,-0.55) {}
   			edge [-] (l);
   			\node[] at (-0.35,-0.7) {$\scriptstyle 3$};
   			\node[] at (-0.15,-0.7) {$\scriptstyle 4$};
   			\node[] at (0,-0.4) {$\scriptstyle 1$};
   			\node[] at (0.25,-0.4) {$\scriptstyle 2$};
   		\end{tikzpicture},
   		& (m,n)=(0,4)\\
   		\begin{tikzpicture}[baseline=-0.4cm]
   			\node[invisible] (c) at (0,0) {};
   			\node[invisible] at (0,0.25) {}
   			edge [-] (c);
   			\node[invisible] at (0.25,-0.25) {}
   			edge [-] (c);
   			\node[invisible] (l) at (-0.25,-0.25) {}
   			edge [-] (c);
   			\node[invisible] at (-0.5,-0.5) {}
   			edge [-] (l);
   			\node[invisible] at (0,-0.5) {}
   			edge [-] (l);
   			\node[] at (-0.5,-0.65) {$\scriptstyle 1$};
   			\node[] at (0,-0.65) {$\scriptstyle 2$};
   			\node[] at (0.25,-0.4) {$\scriptstyle 3$};
   		\end{tikzpicture}
   		+
   		\begin{tikzpicture}[baseline=-0.4cm]
   			\node[invisible] (c) at (0,0) {};
   			\node[invisible] at (0,0.25) {}
   			edge [-] (c);
   			\node[invisible] at (0.25,-0.25) {}
   			edge [-] (c);
   			\node[invisible] (l) at (-0.25,-0.25) {}
   			edge [-] (c);
   			\node[invisible] at (-0.5,-0.5) {}
   			edge [-] (l);
   			\node[invisible] at (0,-0.5) {}
   			edge [-] (l);
   			\node[] at (-0.5,-0.65) {$\scriptstyle 2$};
   			\node[] at (0,-0.65) {$\scriptstyle 3$};
   			\node[] at (0.25,-0.4) {$\scriptstyle 1$};
   		\end{tikzpicture}
   		+
   		\begin{tikzpicture}[baseline=-0.4cm]
   			\node[invisible] (c) at (0,0) {};
   			\node[invisible] at (0,0.25) {}
   			edge [-] (c);
   			\node[invisible] at (0.25,-0.25) {}
   			edge [-] (c);
   			\node[invisible] (l) at (-0.25,-0.25) {}
   			edge [-] (c);
   			\node[invisible] at (-0.5,-0.5) {}
   			edge [-] (l);
   			\node[invisible] at (0,-0.5) {}
   			edge [-] (l);
   			\node[] at (-0.5,-0.65) {$\scriptstyle 3$};
   			\node[] at (0,-0.65) {$\scriptstyle 1$};
   			\node[] at (0.25,-0.4) {$\scriptstyle 3$};
   		\end{tikzpicture}
   		+
   		\begin{tikzpicture}[baseline=-0.4cm]
   			\node[invisible] (r) at (0,0) {};
   			\node[invisible] at (0.25,-0.25) {}
   			edge [-] (r);
   			\node[invisible] at (0,-0.25) {}
   			edge [-] (r);
   			\node[invisible] (l) at (-0.25,-0.25) {}
   			edge [-] (r);
   			\node[invisible] at (-0.5,0) {}
   			edge [-] (l);
   			\node[invisible] at (-0.25,-0.5) {}
   			edge [-] (l);
   			\node[] at (-0.25,-0.65) {$\scriptstyle 1$};
   			\node[] at (0,-0.4) {$\scriptstyle 2$};
   			\node[] at (0.25,-0.4) {$\scriptstyle 3$};
   		\end{tikzpicture}
   		+
   		\begin{tikzpicture}[baseline=-0.4cm]
   			\node[invisible] (r) at (0,0) {};
   			\node[invisible] at (0.25,-0.25) {}
   			edge [-] (r);
   			\node[invisible] at (0,-0.25) {}
   			edge [-] (r);
   			\node[invisible] (l) at (-0.25,-0.25) {}
   			edge [-] (r);
   			\node[invisible] at (-0.5,0) {}
   			edge [-] (l);
   			\node[invisible] at (-0.25,-0.5) {}
   			edge [-] (l);
   			\node[] at (-0.25,-0.65) {$\scriptstyle 2$};
   			\node[] at (0,-0.4) {$\scriptstyle 3$};
   			\node[] at (0.25,-0.4) {$\scriptstyle 1$};
   		\end{tikzpicture}
   		+
   		\begin{tikzpicture}[baseline=-0.4cm]
   			\node[invisible] (r) at (0,0) {};
   			\node[invisible] at (0.25,-0.25) {}
   			edge [-] (r);
   			\node[invisible] at (0,-0.25) {}
   			edge [-] (r);
   			\node[invisible] (l) at (-0.25,-0.25) {}
   			edge [-] (r);
   			\node[invisible] at (-0.5,0) {}
   			edge [-] (l);
   			\node[invisible] at (-0.25,-0.5) {}
   			edge [-] (l);
   			\node[] at (-0.25,-0.65) {$\scriptstyle 3$};
   			\node[] at (0,-0.4) {$\scriptstyle 1$};
   			\node[] at (0.25,-0.4) {$\scriptstyle 2$};
   		\end{tikzpicture},
   		& (m,n)=(1,3)\\
   		\begin{tikzpicture}[baseline=-0.25cm]
   			\node[invisible] (t) at (0,0) {};
   			\node[invisible] at (-0.25,0.25) {}
   			edge [-] (t);
   			\node[invisible] at (0.25,0.25) {}
   			edge [-] (t);
   			\node[invisible] (b) at (0,-0.25) {}
   			edge [-] (t);
   			\node[invisible] at (-0.25,-0.5) {}
   			edge [-] (b);
   			\node[invisible] at (0.25,-0.5) {}
   			edge [-] (b);
   			\node[] at (-0.25,0.4) {$\scriptstyle 1$};
   			\node[] at (0.25,0.4) {$\scriptstyle 2$};
   			\node[] at (-0.25,-0.65) {$\scriptstyle 1$};
   			\node[] at (0.25,-0.65) {$\scriptstyle 2$};
   		\end{tikzpicture}
   		+
   		\begin{tikzpicture}[baseline=0.cm]
   			\node[invisible] (l) at (0,0) {};
   			\node[invisible] at (-0.25,0.25) {}
   			edge [-] (l);
   			\node[invisible] at (0,-0.25) {}
   			edge [-] (l);
   			\node[invisible] (r) at (0.25,0.25) {}
   			edge [-] (l);
   			\node[invisible] at (0.25,0.5) {}
   			edge [-] (r);
   			\node[invisible] at (0.5,0) {}
   			edge [-] (r);
   			\node[] at (-0.25,0.4) {$\scriptstyle 1$};
   			\node[] at (0.25,0.65) {$\scriptstyle 2$};
   			\node[] at (0,-0.4) {$\scriptstyle 1$};
   			\node[] at (0.5,-0.15) {$\scriptstyle 2$};
   		\end{tikzpicture}
   		+
   		\begin{tikzpicture}[baseline=0.cm]
   			\node[invisible] (l) at (0,0) {};
   			\node[invisible] at (-0.25,0.25) {}
   			edge [-] (l);
   			\node[invisible] at (0,-0.25) {}
   			edge [-] (l);
   			\node[invisible] (r) at (0.25,0.25) {}
   			edge [-] (l);
   			\node[invisible] at (0.25,0.5) {}
   			edge [-] (r);
   			\node[invisible] at (0.5,0) {}
   			edge [-] (r);
   			\node[] at (-0.25,0.4) {$\scriptstyle 2$};
   			\node[] at (0.25,0.65) {$\scriptstyle 1$};
   			\node[] at (0,-0.4) {$\scriptstyle 1$};
   			\node[] at (0.5,-0.15) {$\scriptstyle 2$};
   		\end{tikzpicture}
   		+
   		\begin{tikzpicture}[baseline=0.cm]
   			\node[invisible] (l) at (0,0) {};
   			\node[invisible] at (-0.25,0.25) {}
   			edge [-] (l);
   			\node[invisible] at (0,-0.25) {}
   			edge [-] (l);
   			\node[invisible] (r) at (0.25,0.25) {}
   			edge [-] (l);
   			\node[invisible] at (0.25,0.5) {}
   			edge [-] (r);
   			\node[invisible] at (0.5,0) {}
   			edge [-] (r);
   			\node[] at (-0.25,0.4) {$\scriptstyle 1$};
   			\node[] at (0.25,0.65) {$\scriptstyle 2$};
   			\node[] at (0,-0.4) {$\scriptstyle 2$};
   			\node[] at (0.5,-0.15) {$\scriptstyle 1$};
   		\end{tikzpicture}
   		+
   		\begin{tikzpicture}[baseline=0.cm]
   			\node[invisible] (l) at (0,0) {};
   			\node[invisible] at (-0.25,0.25) {}
   			edge [-] (l);
   			\node[invisible] at (0,-0.25) {}
   			edge [-] (l);
   			\node[invisible] (r) at (0.25,0.25) {}
   			edge [-] (l);
   			\node[invisible] at (0.25,0.5) {}
   			edge [-] (r);
   			\node[invisible] at (0.5,0) {}
   			edge [-] (r);
   			\node[] at (-0.25,0.4) {$\scriptstyle 2$};
   			\node[] at (0.25,0.65) {$\scriptstyle 1$};
   			\node[] at (0,-0.4) {$\scriptstyle 2$};
   			\node[] at (0.5,-0.15) {$\scriptstyle 1$};
   		\end{tikzpicture}
   		+
   		\begin{tikzpicture}[baseline=0.cm]
   			\node[invisible] (l) at (0,0) {};
   			\node[invisible] at (-0.25,0.25) {}
   			edge [-] (l);
   			\node[invisible] at (0,0.25) {}
   			edge [-] (l);
   			\node[invisible] (r) at (0.25,0.25) {}
   			edge [-] (l);
   			\node[invisible] at (0.25,0) {}
   			edge [-] (r);
   			\node[invisible] at (0.5,0) {}
   			edge [-] (r);
   			\node[] at (-0.25,0.4) {$\scriptstyle 1$};
   			\node[] at (0,0.4) {$\scriptstyle 2$};
   			\node[] at (0.25,-0.15) {$\scriptstyle 1$};
   			\node[] at (0.5,-0.15) {$\scriptstyle 2$};
   		\end{tikzpicture},
   		& (m,n)=(2,2)\\
   		\begin{tikzpicture}[baseline=0.3cm]
   			\node[invisible] (c) at (0,0) {};
   			\node[invisible] at (0,-0.25) {}
   			edge [-] (c);
   			\node[invisible] (r) at (0.25,0.25) {}
   			edge [-] (c);
   			\node[invisible] (l) at (-0.25,0.25) {}
   			edge [-] (c);
   			\node[invisible] at (0.5,0.5) {}
   			edge [-] (r);
   			\node[invisible] at (0,0.5) {}
   			edge [-] (r);
   			\node[] at (-0.25,0.4) {$\scriptstyle 1$}; 
   			\node[] at (0,0.65) {$\scriptstyle 2$};
   			\node[] at (0.5,0.65) {$\scriptstyle 3$};
   		\end{tikzpicture}
   		+
   		\begin{tikzpicture}[baseline=0.3cm]
   			\node[invisible] (c) at (0,0) {};
   			\node[invisible] at (0,-0.25) {}
   			edge [-] (c);
   			\node[invisible] (r) at (0.25,0.25) {}
   			edge [-] (c);
   			\node[invisible] (l) at (-0.25,0.25) {}
   			edge [-] (c);
   			\node[invisible] at (0.5,0.5) {}
   			edge [-] (r);
   			\node[invisible] at (0,0.5) {}
   			edge [-] (r);
   			\node[] at (-0.25,0.4) {$\scriptstyle 2$}; 
   			\node[] at (0,0.65) {$\scriptstyle 3$};
   			\node[] at (0.5,0.65) {$\scriptstyle 1$};
   		\end{tikzpicture}
   		+
   		\begin{tikzpicture}[baseline=0.3cm]
   			\node[invisible] (c) at (0,0) {};
   			\node[invisible] at (0,-0.25) {}
   			edge [-] (c);
   			\node[invisible] (r) at (0.25,0.25) {}
   			edge [-] (c);
   			\node[invisible] (l) at (-0.25,0.25) {}
   			edge [-] (c);
   			\node[invisible] at (0.5,0.5) {}
   			edge [-] (r);
   			\node[invisible] at (0,0.5) {}
   			edge [-] (r);
   			\node[] at (-0.25,0.4) {$\scriptstyle 3$}; 
   			\node[] at (0,0.65) {$\scriptstyle 1$};
   			\node[] at (0.5,0.65) {$\scriptstyle 2$};
   		\end{tikzpicture}
   		+
   		\begin{tikzpicture}[baseline=0.3cm]
   			\node[invisible] (l) at (0,0) {};
   			\node[invisible] at (-0.25,0.25) {}
   			edge [-] (l);
   			\node[invisible] at (0,0.25) {}
   			edge [-] (l);
   			\node[invisible] (r) at (0.25,0.25) {}
   			edge [-] (l);
   			\node[invisible] at (0.5,0) {}
   			edge [-] (r);
   			\node[invisible] at (0.25,0.5) {}
   			edge [-] (r);
   			\node[] at (-0.25,0.4) {$\scriptstyle 1$};
   			\node[] at (0,0.4) {$\scriptstyle 2$};
   			\node[] at (0.25,0.65) {$\scriptstyle 3$};
   		\end{tikzpicture}
   		+
   		\begin{tikzpicture}[baseline=0.3cm]
   			\node[invisible] (l) at (0,0) {};
   			\node[invisible] at (-0.25,0.25) {}
   			edge [-] (l);
   			\node[invisible] at (0,0.25) {}
   			edge [-] (l);
   			\node[invisible] (r) at (0.25,0.25) {}
   			edge [-] (l);
   			\node[invisible] at (0.5,0) {}
   			edge [-] (r);
   			\node[invisible] at (0.25,0.5) {}
   			edge [-] (r);
   			\node[] at (-0.25,0.4) {$\scriptstyle 2$};
   			\node[] at (0,0.4) {$\scriptstyle 3$};
   			\node[] at (0.25,0.65) {$\scriptstyle 1$};
   		\end{tikzpicture}
   		+
   		\begin{tikzpicture}[baseline=0.3cm]
   			\node[invisible] (l) at (0,0) {};
   			\node[invisible] at (-0.25,0.25) {}
   			edge [-] (l);
   			\node[invisible] at (0,0.25) {}
   			edge [-] (l);
   			\node[invisible] (r) at (0.25,0.25) {}
   			edge [-] (l);
   			\node[invisible] at (0.5,0) {}
   			edge [-] (r);
   			\node[invisible] at (0.25,0.5) {}
   			edge [-] (r);
   			\node[] at (-0.25,0.4) {$\scriptstyle 3$};
   			\node[] at (0,0.4) {$\scriptstyle 1$};
   			\node[] at (0.25,0.65) {$\scriptstyle 2$};
   		\end{tikzpicture},
   		& (m,n)=(3,1)\\
   		\begin{tikzpicture}[baseline=0.3cm]
   			\node[invisible] (c) at (0,0) {};
   			\node[invisible] at (0,0.25) {}
   			edge [-] (c);
   			\node[invisible] (r) at (0.25,0.25) {}
   			edge [-] (c);
   			\node[invisible] at (-0.25,0.25) {}
   			edge [-] (c);
   			\node[invisible] at (0.15,0.55) {}
   			edge [-] (r);
   			\node[invisible] at (0.35,0.55) {}
   			edge [-] (r);
   			\node[] at (-0.25,0.4) {$\scriptstyle 1$};
   			\node[] at (0,0.4) {$\scriptstyle 2$};
   			\node[] at (0.15,0.7) {$\scriptstyle 3$};
   			\node[] at (0.35,0.7) {$\scriptstyle 4$};
   		\end{tikzpicture}
   		+
   		\begin{tikzpicture}[baseline=0.3cm]
   			\node[invisible] (c) at (0,0) {};
   			\node[invisible] at (0,0.25) {}
   			edge [-] (c);
   			\node[invisible] (r) at (0.25,0.25) {}
   			edge [-] (c);
   			\node[invisible] at (-0.25,0.25) {}
   			edge [-] (c);
   			\node[invisible] at (0.15,0.55) {}
   			edge [-] (r);
   			\node[invisible] at (0.35,0.55) {}
   			edge [-] (r);
   			\node[] at (-0.25,0.4) {$\scriptstyle 1$};
   			\node[] at (0,0.4) {$\scriptstyle 3$};
   			\node[] at (0.15,0.7) {$\scriptstyle 2$};
   			\node[] at (0.35,0.7) {$\scriptstyle 4$};
   		\end{tikzpicture}
   		+
   		\begin{tikzpicture}[baseline=0.3cm]
   			\node[invisible] (c) at (0,0) {};
   			\node[invisible] at (0,0.25) {}
   			edge [-] (c);
   			\node[invisible] (r) at (0.25,0.25) {}
   			edge [-] (c);
   			\node[invisible] at (-0.25,0.25) {}
   			edge [-] (c);
   			\node[invisible] at (0.15,0.55) {}
   			edge [-] (r);
   			\node[invisible] at (0.35,0.55) {}
   			edge [-] (r);
   			\node[] at (-0.25,0.4) {$\scriptstyle 1$};
   			\node[] at (0,0.4) {$\scriptstyle 4$};
   			\node[] at (0.15,0.7) {$\scriptstyle 2$};
   			\node[] at (0.35,0.7) {$\scriptstyle 3$};
   		\end{tikzpicture}
   		+
   		\begin{tikzpicture}[baseline=0.3cm]
   			\node[invisible] (c) at (0,0) {};
   			\node[invisible] at (0,0.25) {}
   			edge [-] (c);
   			\node[invisible] (r) at (0.25,0.25) {}
   			edge [-] (c);
   			\node[invisible] at (-0.25,0.25) {}
   			edge [-] (c);
   			\node[invisible] at (0.15,0.55) {}
   			edge [-] (r);
   			\node[invisible] at (0.35,0.55) {}
   			edge [-] (r);
   			\node[] at (-0.25,0.4) {$\scriptstyle 2$};
   			\node[] at (0,0.4) {$\scriptstyle 3$};
   			\node[] at (0.15,0.7) {$\scriptstyle 1$};
   			\node[] at (0.35,0.7) {$\scriptstyle 4$};
   		\end{tikzpicture}
   		+
   		\begin{tikzpicture}[baseline=0.3cm]
   			\node[invisible] (c) at (0,0) {};
   			\node[invisible] at (0,0.25) {}
   			edge [-] (c);
   			\node[invisible] (r) at (0.25,0.25) {}
   			edge [-] (c);
   			\node[invisible] at (-0.25,0.25) {}
   			edge [-] (c);
   			\node[invisible] at (0.15,0.55) {}
   			edge [-] (r);
   			\node[invisible] at (0.35,0.55) {}
   			edge [-] (r);
   			\node[] at (-0.25,0.4) {$\scriptstyle 2$};
   			\node[] at (0,0.4) {$\scriptstyle 4$};
   			\node[] at (0.15,0.7) {$\scriptstyle 1$};
   			\node[] at (0.35,0.7) {$\scriptstyle 3$};
   		\end{tikzpicture}
   		+
   		\begin{tikzpicture}[baseline=0.3cm]
   			\node[invisible] (c) at (0,0) {};
   			\node[invisible] at (0,0.25) {}
   			edge [-] (c);
   			\node[invisible] (r) at (0.25,0.25) {}
   			edge [-] (c);
   			\node[invisible] at (-0.25,0.25) {}
   			edge [-] (c);
   			\node[invisible] at (0.15,0.55) {}
   			edge [-] (r);
   			\node[invisible] at (0.35,0.55) {}
   			edge [-] (r);
   			\node[] at (-0.25,0.4) {$\scriptstyle 3$};
   			\node[] at (0,0.4) {$\scriptstyle 4$};
   			\node[] at (0.15,0.7) {$\scriptstyle 1$};
   			\node[] at (0.35,0.7) {$\scriptstyle 2$};
   		\end{tikzpicture},
   		& (m,n)=(4,0)
   	\end{cases}
   \end{equation*}

Similarly, one may define the dioperad (or properad) $\mathcal{QL}ieb$ of Quasi-Lie bialgebras. The Koszulness of the corresponding properad (and consequently the dioperads) was established in~\cite{Granacker}, following the methodology used for the Koszul property of the properad of Lie bialgebras. In particular, it can be shown that each space of $(m,n)$-ary operations in the Koszul-dual dioperad is one-dimensional whenever it is non-empty.

Let us adapt the framework developed in this paper to these examples.

\begin{proposition}
\begin{itemize}[itemsep=0pt, topsep=0pt]
\item The following isomorphisms of dioperads hold:
\begin{equation}
\label{eq::pseudolie::coloring}
\mathcal{PL}ieb^{!} \simeq \Theta_{\bZ_{\geq 0}\times\bZ_{\geq 0}}(\Com), \quad \mathcal{QL}ieb^{!} = \Theta_{\bZ_{\geq 0}\times\bZ_{>0}}(\Com). 
\end{equation}
In particular, both dioperads admit a quadratic, convergent, caterpillar rewriting system.
\item Specifically, the $2$-colored shuffle operad $\Psi(\mathcal{PL}ieb)$ is generated by eight binary generators (one for each possible coloring), subject to sixteen quadratic relations (one for each possible input/output coloring). 
\item The rewriting system $S:= (\tau_i \to f_i)$, where $\tau_i$ is a right-growing shuffle ternary tree with an internal edge colored straight, is both terminating and confluent.
\end{itemize}
\end{proposition}

\begin{proof}
First, a direct comparison of the generators and relations for $\mathcal{PL}ieb^{!}$ and $\mathcal{QL}ieb^{!}$ with those arising from the different colorings of the commutative operad shows Isomorphism~\eqref{eq::pseudolie::coloring}.
Second, we generalize the arguments we provided for Lie bialgebras; it is sufficient to construct the Gr\"obner basis for the Koszul-dual 2-colored shuffle operad, which is a straightforward verification.
\end{proof}

\subsection{The dioperad $\LieTriang$ of triangular Lie bialgebras}
\label{example::Lie::Triang}

The following algebraic concepts were introduced by V.\,Drinfeld in~\cite{Drinfeld_Triangular} as a foundational component of his seminal work on quantum groups.

\begin{definition}
A Lie bialgebra $(\mathfrak{g}, [\ttt, \ttt], \delta)$ is called a \emph{coboundary Lie bialgebra} if there exists an element $r \in \mathfrak{g} \otimes \mathfrak{g}$ such that the Lie cobracket $\delta: \mathfrak{g} \to \mathfrak{g} \otimes \mathfrak{g}$ is defined by the following coboundary condition:
\begin{equation*}
\forall x \in \mathfrak{g} \quad    \delta(x) = \text{ad}_x(r) = [x \otimes 1 + 1 \otimes x, r].
\end{equation*}
\end{definition}

\begin{definition}
A coboundary Lie bialgebra $(\mathfrak{g}, [\ttt, \ttt], \delta, r)$ is said to be \emph{triangular} if the defining $r$-matrix is skew-symmetric (i.e., $r \in \Lambda^2 \mathfrak{g}$) and satisfies the \emph{Classical Yang-Baxter Equation (CYBE):} 
    \begin{equation}
        [[r, r]] = [r_{12}, r_{13}] + [r_{12}, r_{23}] + [r_{13}, r_{23}] = 0
    \end{equation}
    where $[[ \cdot, \cdot ]]$ denotes the algebraic Schouten bracket on $\mathfrak{g} \otimes \mathfrak{g} \otimes \mathfrak{g}$.
\end{definition}

The dioperad $\LieTriang$, whose representations correspond to triangular Lie bialgebras, is generated by two skew-symmetric operations in arities $(2,1)$ and $(0,2)$, respectively. These generators represent the Lie bracket and the $r$-matrix:
$$
 \liePic{1}{2}
=(-1)
\liePic{2}{1}: = [\ttt,\ttt]; \qquad 
\RLie{1}{2} = (-1)\RLie{2}{1} := r .  
$$
The structure is governed by the following relations, representing the Jacobi identity and the CYBE:

\begin{equation}
\label{eq::triang::LieBi}
\left\{
\LLiePic{1}{2}{3} + 	\LLiePic{2}{3}{1} + \LLiePic{3}{1}{2} = 0; \quad \quad
\RLieRel{1}{2}{3} +\RLieRel{2}{3}{1} + \RLieRel{3}{1}{2} =0
\right\}
\end{equation}

Furthermore, the associated $2$-colored shuffle operad $\Psi(\LieTriang)$ is generated by one unary generator and three binary generators, the latter of which arise from the various colorings of the Lie bracket:
\begin{equation*}
\begin{array}{c}
\begin{tikzpicture}[scale=0.15pt, baseline=-0.5ex]
    \draw (0,-0.55) -- (0,-2.5);
    \draw[dotted] (0,0.5) -- (0,2.2);
    \node[circle,draw,inner sep=1.5pt, fill = black] (A) at (0,0) {}; 
    \node at (0,2.8) {$\scriptscriptstyle 1$};
\end{tikzpicture}:=
\RLie{0}{1}
=(-1) \RLie{1}{0}; \\
\begin{tikzpicture}[scale=0.15pt, baseline=-0.5ex]
    \draw (0,-0.55) -- (0,-2.5);
    \draw (0.5,0.5) -- (2.2,2.2);
    \draw (-0.48,0.48) -- (-2.2,2.2);
    \node[circle,draw,inner sep=1.5pt, fill=white] (A) at (0,0) {};  
    \node at (-2.7,2.8) {$\scriptscriptstyle 1$};
    \node at (2.7,2.8) {$\scriptscriptstyle 2$};
    \node at (0,-3.2) {$\scriptscriptstyle 0$};    
\end{tikzpicture}
= (-1) \begin{tikzpicture}[scale=0.15pt, baseline=-0.5ex]
    \draw (0,-0.55) -- (0,-2.5);
    \draw (0.5,0.5) -- (2.2,2.2);
    \draw (-0.48,0.48) -- (-2.2,2.2);
    \node[circle,draw,inner sep=1.5pt, fill=white] (A) at (0,0) {};  
    \node at (-2.7,2.8) {$\scriptscriptstyle 2$};
    \node at (2.7,2.8) {$\scriptscriptstyle 1$};
    \node at (0,-3.2) {$\scriptscriptstyle 0$};
\end{tikzpicture},
 \\
\begin{tikzpicture}[scale=0.15pt, baseline=-0.5ex]
    \draw[dotted] (0,-0.55) -- (0,-2.5);
    \draw (0.5,0.5) -- (2.2,2.2);
    \draw[dotted] (-0.48,0.48) -- (-2.2,2.2);
    \node[circle,draw,inner sep=1.5pt, fill = white] (A) at (0,0)   {};
    \node at (-2.7,2.8) {$\scriptscriptstyle 1$};
    \node at (2.7,2.8) {$\scriptscriptstyle 2$}; 
    \node at (0,-3.2) {$\scriptscriptstyle 0$};    
\end{tikzpicture}:= 
\begin{tikzpicture}[scale=0.15pt, baseline=-0.5ex]
    \draw (0,-0.55) -- (0,-2.5);
    \draw (0.5,0.5) -- (2.2,2.2);
    \draw (-0.48,0.48) -- (-2.2,2.2);
    \node[circle,draw,inner sep=1.5pt, fill=white] (A) at (0,0) {};  
    \node at (-2.7,2.8) {$\scriptscriptstyle 0$};
    \node at (2.7,2.8) {$\scriptscriptstyle 2$};
        \node at (0,-3.2) {$\scriptscriptstyle 1$};
\end{tikzpicture}
= (-1) 
\begin{tikzpicture}[scale=0.15pt, baseline=-0.5ex]
    \draw (0,-0.55) -- (0,-2.5);
    \draw (0.5,0.5) -- (2.2,2.2);
    \draw (-0.48,0.48) -- (-2.2,2.2);
    \node[circle,draw,inner sep=1.5pt, fill=white] (A) at (0,0) {};  
    \node at (-2.7,2.8) {$\scriptscriptstyle 2$};
    \node at (2.7,2.8) {$\scriptscriptstyle 0$};
        \node at (0,-3.2) {$\scriptscriptstyle 1$};
\end{tikzpicture},
\\
\begin{tikzpicture}[scale=0.15pt, baseline=-0.5ex]
    \draw[dotted] (0,-0.55) -- (0,-2.5);
    \draw[dotted] (0.5,0.5) -- (2.2,2.2);
    \draw (-0.48,0.48) -- (-2.2,2.2);
    \node[circle,draw,inner sep=1.5pt, fill = white] (A) at (0,0) {}; 
    \node at (-2.7,2.8) {$\scriptscriptstyle 1$};
    \node at (2.7,2.8) {$\scriptscriptstyle 2$};
       \node at (0,-3.2) {$\scriptscriptstyle 0$};
\end{tikzpicture}
:=(-1)
\begin{tikzpicture}[scale=0.15pt, baseline=-0.5ex]
    \draw (0,-0.55) -- (0,-2.5);
    \draw (0.5,0.5) -- (2.2,2.2);
    \draw (-0.48,0.48) -- (-2.2,2.2);
    \node[circle,draw,inner sep=1.5pt, fill=white] (A) at (0,0) {};  
    \node at (-2.7,2.8) {$\scriptscriptstyle 0$};
    \node at (2.7,2.8) {$\scriptscriptstyle 1$};
        \node at (0,-3.2) {$\scriptscriptstyle 2$};
\end{tikzpicture}
=  
\begin{tikzpicture}[scale=0.15pt, baseline=-0.5ex]
    \draw (0,-0.55) -- (0,-2.5);
    \draw (0.5,0.5) -- (2.2,2.2);
    \draw (-0.48,0.48) -- (-2.2,2.2);
    \node[circle,draw,inner sep=1.5pt, fill=white] (A) at (0,0) {};  
    \node at (-2.7,2.8) {$\scriptscriptstyle 1$};
    \node at (2.7,2.8) {$\scriptscriptstyle 0$};
        \node at (0,-3.2) {$\scriptscriptstyle 2$};
\end{tikzpicture}.
\end{array}
\end{equation*}

\begin{theorem} 
\label{thm::triangular::Grobner}
The shuffle relations
\begin{gather}
\label{eq::Jacobi::triang}
\begin{array}{ccc}
{
\begin{tikzpicture}[x=1.2mm,y=1.2mm, xscale=-0.7]	
	\node[draw,circle,inner sep=1pt] (A) at (0,0) {};
	\draw (0,-0.49) -- (0,-3.0);
	\draw (0.49,0.49) -- (1.9,1.9);
	\draw (-0.5,0.5) -- (-1.9,1.9);	
	\node[draw,circle,inner sep=1pt] (B) at (-2.3,2.3) {};
	\draw (-1.8,2.8) -- (0,4.9);
	\draw (-2.8,2.9) -- (-4.6,4.9);	
	\node at (2.7,2.3) {\scriptsize $1$};
	\node at (0.4,5.3) {\scriptsize $2$};
	\node at (-5.1,5.3) {\scriptsize $3$};
\end{tikzpicture}
=
\begin{tikzpicture}[x=1.2mm,y=1.2mm,xscale=0.7]	
	\node[draw,circle,inner sep=1pt] (A) at (0,0) {};
	\draw (0,-0.49) -- (0,-3.0);
	\draw (0.49,0.49) -- (1.9,1.9);
	\draw (-0.5,0.5) -- (-1.9,1.9);	
	\node[draw,circle,inner sep=1pt] (B) at (-2.3,2.3) {};
	\draw (-1.8,2.8) -- (0,4.9);
	\draw (-2.8,2.9) -- (-4.6,4.9);	
	\node at (2.7,2.3) {\scriptsize $3$};
	\node at (0.4,5.3) {\scriptsize $2$};
	\node at (-5.1,5.3) {\scriptsize $1$};
\end{tikzpicture}
- 
\begin{tikzpicture}[x=1.2mm,y=1.2mm, xscale=0.7]	
	\node[draw,circle,inner sep=1pt] (A) at (0,0) {};
	\draw (0,-0.49) -- (0,-3.0);
	\draw (0.49,0.49) -- (1.9,1.9);
	\draw (-0.5,0.5) -- (-1.9,1.9);	
	\node[draw,circle,inner sep=1pt] (B) at (-2.3,2.3) {};
	\draw (-1.8,2.8) -- (0,4.9);
	\draw (-2.8,2.9) -- (-4.6,4.9);	
	\node at (2.7,2.3) {\scriptsize $2$};
	\node at (0.4,5.3) {\scriptsize $3$};
	\node at (-5.1,5.3) {\scriptsize $1$};
\end{tikzpicture}
}, & \quad &
{
\begin{tikzpicture}[x=1.2mm,y=1.2mm, xscale=-0.7]	
	\node[draw,circle,inner sep=1pt] (A) at (0,0) {};
	\draw[dotted] (0,-0.49) -- (0,-3.0);
	\draw (0.49,0.49) -- (1.9,1.9);
	\draw[dotted] (-0.5,0.5) -- (-1.9,1.9);	
	\node[draw,circle,inner sep=1pt] (B) at (-2.3,2.3) {};
	\draw (-1.8,2.8) -- (0,4.9);
	\draw[dotted] (-2.8,2.9) -- (-4.6,4.9);	
	\node at (2.7,2.3) {\scriptsize $1$};
	\node at (0.4,5.3) {\scriptsize $2$};
	\node at (-5.1,5.3) {\scriptsize $3$};
\end{tikzpicture}
=
\begin{tikzpicture}[x=1.2mm,y=1.2mm,xscale=0.7]	
	\node[draw,circle,inner sep=1pt] (A) at (0,0) {};
	\draw[dotted] (0,-0.49) -- (0,-3.0);
	\draw[dotted] (0.49,0.49) -- (1.9,1.9);
	\draw (-0.5,0.5) -- (-1.9,1.9);	
	\node[draw,circle,inner sep=1pt] (B) at (-2.3,2.3) {};
	\draw (-1.8,2.8) -- (0,4.9);
	\draw (-2.8,2.9) -- (-4.6,4.9);	
	\node at (2.7,2.3) {\scriptsize $3$};
	\node at (0.4,5.3) {\scriptsize $2$};
	\node at (-5.1,5.3) {\scriptsize $1$};
\end{tikzpicture}
- 
\begin{tikzpicture}[x=1.2mm,y=1.2mm, xscale=0.7]	
	\node[draw,circle,inner sep=1pt] (A) at (0,0) {};
	\draw[dotted] (0,-0.49) -- (0,-3.0);
	\draw (0.49,0.49) -- (1.9,1.9);
	\draw[dotted] (-0.5,0.5) -- (-1.9,1.9);	
	\node[draw,circle,inner sep=1pt] (B) at (-2.3,2.3) {};
	\draw[dotted] (-1.8,2.8) -- (0,4.9);
	\draw (-2.8,2.9) -- (-4.6,4.9);	
	\node at (2.7,2.3) {\scriptsize $2$};
	\node at (0.4,5.3) {\scriptsize $3$};
	\node at (-5.1,5.3) {\scriptsize $1$};
\end{tikzpicture}
},
\\
{
\begin{tikzpicture}[x=1.2mm,y=1.2mm, xscale=-0.7]	
	\node[draw,circle,inner sep=1pt] (A) at (0,0) {};
	\draw[dotted] (0,-0.49) -- (0,-3.0);
	\draw (0.49,0.49) -- (1.9,1.9);
	\draw[dotted] (-0.5,0.5) -- (-1.9,1.9);	
	\node[draw,circle,inner sep=1pt] (B) at (-2.3,2.3) {};
	\draw[dotted] (-1.8,2.8) -- (0,4.9);
	\draw (-2.8,2.9) -- (-4.6,4.9);	
	\node at (2.7,2.3) {\scriptsize $1$};
	\node at (0.4,5.3) {\scriptsize $2$};
	\node at (-5.1,5.3) {\scriptsize $3$};
\end{tikzpicture}
=
\begin{tikzpicture}[x=1.2mm,y=1.2mm,xscale=0.7]	
	\node[draw,circle,inner sep=1pt] (A) at (0,0) {};
	\draw[dotted] (0,-0.49) -- (0,-3.0);
	\draw (0.49,0.49) -- (1.9,1.9);
	\draw[dotted] (-0.5,0.5) -- (-1.9,1.9);	
	\node[draw,circle,inner sep=1pt] (B) at (-2.3,2.3) {};
	\draw[dotted] (-1.8,2.8) -- (0,4.9);
	\draw (-2.8,2.9) -- (-4.6,4.9);	
	\node at (2.7,2.3) {\scriptsize $3$};
	\node at (0.4,5.3) {\scriptsize $2$};
	\node at (-5.1,5.3) {\scriptsize $1$};
\end{tikzpicture}
- 
\begin{tikzpicture}[x=1.2mm,y=1.2mm,xscale=0.7]	
	\node[draw,circle,inner sep=1pt] (A) at (0,0) {};
	\draw[dotted] (0,-0.49) -- (0,-3.0);
	\draw[dotted] (0.49,0.49) -- (1.9,1.9);
	\draw (-0.5,0.5) -- (-1.9,1.9);	
	\node[draw,circle,inner sep=1pt] (B) at (-2.3,2.3) {};
	\draw (-1.8,2.8) -- (0,4.9);
	\draw (-2.8,2.9) -- (-4.6,4.9);	
	\node at (2.7,2.3) {\scriptsize $2$};
	\node at (0.4,5.3) {\scriptsize $3$};
	\node at (-5.1,5.3) {\scriptsize $1$};
\end{tikzpicture}    
},
& \quad &
{
\begin{tikzpicture}[x=1.2mm,y=1.2mm, xscale=-0.7]	
	\node[draw,circle,inner sep=1pt] (A) at (0,0) {};
	\draw[dotted] (0,-0.49) -- (0,-3.0);
	\draw[dotted] (0.49,0.49) -- (1.9,1.9);
	\draw (-0.5,0.5) -- (-1.9,1.9);	
	\node[draw,circle,inner sep=1pt] (B) at (-2.3,2.3) {};
	\draw (-1.8,2.8) -- (0,4.9);
	\draw (-2.8,2.9) -- (-4.6,4.9);	
	\node at (2.7,2.3) {\scriptsize $1$};
	\node at (0.4,5.3) {\scriptsize $2$};
	\node at (-5.1,5.3) {\scriptsize $3$};
\end{tikzpicture}
=
\begin{tikzpicture}[x=1.2mm,y=1.2mm, xscale=0.7]	
	\node[draw,circle,inner sep=1pt] (A) at (0,0) {};
	\draw[dotted] (0,-0.49) -- (0,-3.0);
	\draw (0.49,0.49) -- (1.9,1.9);
	\draw[dotted] (-0.5,0.5) -- (-1.9,1.9);	
	\node[draw,circle,inner sep=1pt] (B) at (-2.3,2.3) {};
	\draw (-1.8,2.8) -- (0,4.9);
	\draw[dotted] (-2.8,2.9) -- (-4.6,4.9);	
	\node at (2.7,2.3) {\scriptsize $3$};
	\node at (0.4,5.3) {\scriptsize $2$};
	\node at (-5.1,5.3) {\scriptsize $1$};
\end{tikzpicture}
- 
\begin{tikzpicture}[x=1.2mm,y=1.2mm,xscale=0.7]	
	\node[draw,circle,inner sep=1pt] (A) at (0,0) {};
	\draw[dotted] (0,-0.49) -- (0,-3.0);
	\draw (0.49,0.49) -- (1.9,1.9);
	\draw[dotted] (-0.5,0.5) -- (-1.9,1.9);	
	\node[draw,circle,inner sep=1pt] (B) at (-2.3,2.3) {};
	\draw (-1.8,2.8) -- (0,4.9);
	\draw[dotted] (-2.8,2.9) -- (-4.6,4.9);	
	\node at (2.7,2.3) {\scriptsize $2$};
	\node at (0.4,5.3) {\scriptsize $3$};
	\node at (-5.1,5.3) {\scriptsize $1$};
\end{tikzpicture}.  
}
\end{array}
\\
\label{eq::Yang-Bakster::di}
\begin{tikzpicture}[x=1.2mm,y=1.2mm]	
	\node[draw,circle,inner sep=1pt, fill =black] (C) at (0,-3.3) {};    
	\node[draw,circle,inner sep=1pt] (A) at (0,0) {};
	\draw[dotted] (0,-0.49) -- (0,-3.0);
    \draw (0,-3.5) -- (0,-5.9);
	\draw (0.49,0.49) -- (1.9,1.9);
	\draw[dotted] (-0.5,0.5) -- (-2.9,2.9);	
	\node[draw,circle,inner sep=1pt,fill = black] (B) at (2.3,2.3) {};
	\draw[dotted] (2.3,2.8) -- (2.3,5.2);
	\node at (2.3,5.8) {\scriptsize $2$};
	\node at (-3.8,3.8) {\scriptsize $1$};
\end{tikzpicture}  =
\begin{tikzpicture}[x=1.2mm,y=1.2mm]	
	\node[draw,circle,inner sep=1pt] (A) at (0,0) {};
	\draw (0,-0.49) -- (0,-3.0);
	\draw (0.49,0.49) -- (1.9,1.9);
	\draw (-0.5,0.5) -- (-1.9,1.9);	
	\node[draw,circle,inner sep=1pt, fill = black] (B) at (-2.3,2.3) {};
	\node[draw,circle,inner sep=1pt, fill =black] (C) at (2.3,2.3) {};
	\draw[dotted] (-2.3,2.8) -- (-2.3,4.9);
	\draw[dotted] (2.3,2.8) -- (2.3,4.9);	
	\node at (-2.3,5.6) {\scriptsize $1$};
	\node at (2.3,5.6) {\scriptsize $2$};
\end{tikzpicture}
  -
\begin{tikzpicture}[x=1.2mm,y=1.2mm,xscale=-1]	
	\node[draw,circle,inner sep=1pt, fill =black] (C) at (0,-3.3) {};    
	\node[draw,circle,inner sep=1pt] (A) at (0,0) {};
	\draw[dotted] (0,-0.49) -- (0,-3.0);
    \draw (0,-3.5) -- (0,-5.9);
	\draw (0.49,0.49) -- (1.9,1.9);
	\draw[dotted] (-0.5,0.5) -- (-2.9,2.9);	
	\node[draw,circle,inner sep=1pt,fill = black] (B) at (2.3,2.3) {};
	\draw[dotted] (2.3,2.8) -- (2.3,5.2);
	\node at (2.3,5.8) {\scriptsize $1$};
	\node at (-3.8,3.8) {\scriptsize $2$};
\end{tikzpicture} 
\end{gather}
associated with the relations in~\eqref{eq::triang::LieBi} constitute a Gr\"obner basis with respect to the reverse path-lexicographical ordering of monomials induced by the following ordering of generators:
$$
\begin{tikzpicture}[scale=0.15pt, baseline=-0.5ex]
    \draw (0,-0.55) -- (0,-2.5);
    \draw (0.5,0.5) -- (2.2,2.2);
    \draw (-0.48,0.48) -- (-2.2,2.2);
    \node[circle,draw,inner sep=1.5pt, fill=white] (A) at (0,0) {};  
    \node at (-2.7,2.8) {$\scriptscriptstyle 1$};
    \node at (2.7,2.8) {$\scriptscriptstyle 2$};
\end{tikzpicture}
>
\begin{tikzpicture}[scale=0.15pt, baseline=-0.5ex]
    \draw[dotted] (0,-0.55) -- (0,-2.5);
    \draw (0.5,0.5) -- (2.2,2.2);
    \draw[dotted] (-0.48,0.48) -- (-2.2,2.2);
    \node[circle,draw,inner sep=1.5pt, fill = white] (A) at (0,0)   {};
    \node at (-2.7,2.8) {$\scriptscriptstyle 1$};
    \node at (2.7,2.8) {$\scriptscriptstyle 2$}; 
\end{tikzpicture}
>
\begin{tikzpicture}[scale=0.15pt, baseline=-0.5ex]
    \draw[dotted] (0,-0.55) -- (0,-2.5);
    \draw[dotted] (0.5,0.5) -- (2.2,2.2);
    \draw (-0.48,0.48) -- (-2.2,2.2);
    \node[circle,draw,inner sep=1.5pt, fill = white] (A) at (0,0) {}; 
    \node at (-2.7,2.8) {$\scriptscriptstyle 1$};
    \node at (2.7,2.8) {$\scriptscriptstyle 2$};
\end{tikzpicture}
>
\begin{tikzpicture}[scale=0.15pt, baseline=-0.5ex]
    \draw (0,-0.55) -- (0,-2.5);
    \draw[dotted] (0,0.5) -- (0,2.2);
    \node[circle,draw,inner sep=1.5pt, fill = black] (A) at (0,0) {}; 
    \node at (0,2.8) {$\scriptscriptstyle 1$};
\end{tikzpicture}
$$
(In particular, the left-hand sides of the relations in~\eqref{eq::Jacobi::triang}--\eqref{eq::Yang-Bakster::di} are the leading terms, which we have expressed as a rewriting system).
\end{theorem}

\begin{proof}
To verify that these relations form a Gr\"obner basis, we must show that all $S$-polynomials reduce to zero. 
Note that the only nontrivial compositions occur between the two leading terms of the (colored) Jacobi identities~\eqref{eq::Jacobi::triang} and a unique composition arising from the two Yang-Baxter equation~\eqref{eq::Yang-Bakster::di}. 

The vanishing of the $S$-polynomials for the compositions of the (colored) Jacobi identities follows from Theorem~\ref{thm::rewrite::system}. The reduction of the remaining $S$-polynomial is illustrated in the pictorial computation below. 
Namely, we compare two different reduction paths for the same monomial. For the reader's convenience, after each reduction, we encircle (in dotted red) the divisor to be reduced in the subsequent step, and we denote the corresponding reduction in brackets after the equality:
\begin{multline}
\label{eq::triang::reduction::1}
\begin{tikzpicture}[scale=0.22,xscale=0.55]
   	\node[int] (v0) at (0,-2) {};
    \node (l0) at (0,-3.8) {};
    \draw (v0) -- (l0);
	\node[ext] (v1) at (0,0) {};
    \draw[dotted] (v0) -- (v1);    
	\node[int] (v2) at (2,2) {};
    \draw (v1) -- (v2);
    \node (l1) at (-2,2) {$\scriptstyle 1$};
    \draw[dotted] (v1) -- (l1);
    \node[ext] (v3) at (2,4) {};
    \draw[dotted] (v2) -- (v3);
    \node (l2) at (0,6) {$\scriptstyle 2$};
    \draw[dotted] (v3) -- (l2);
    \node[int] (v4) at (4,6) {};
    \draw (v3) -- (v4);
    \node (l3) at (4,8.3) {$\scriptstyle 3$};
    \draw[dotted] (v4) -- (l3);
    \ovalthree{v2}{v3}{v4};
\end{tikzpicture}
= \RwordRLLRR123 - \RLwordRLrLR123 = \Bigl(\LRwordRLRLR123-\LRwordRLRLR132\Bigr) - \Bigl(\LwordLRRLR123 - \LRwordRLRRL123 \Bigr) = \\
= \Bigl( \RwordLRLRR123 - \RLwordRLRLR123 \Bigr) 
- \Bigl( \RwordLRLRR132 - \RLwordRLRLR132 \Bigr) - \LwordLRRLR123 
+ \Bigl( \LwordRLLRR123 - \LRwordRLLRR132 \Bigr) = \\
= \Bigl( \LwordLLRRR123 - \LwordLRLRR123 \Bigr)  
- \Bigl( \cbl{\LwordRLLRR123} - \LwordRLRLR123 \Bigr) 
- \Bigl( \LwordLLRRR132 - {\color{purple}{\LwordLRLRR132}} \Bigr)  
+ \Bigl( \LwordRLLRR132 - \cbrown{\LwordRLRLR132} \Bigr) 
 - \LwordLRRLR123 + \\ + \cbl{\LwordRLLRR123} 
 -\Bigl( {\color{purple}{\LwordLRLRR132}} - \cbrown{\LwordRLRLR132} \Bigr) =  \LwordLLRRR123 - \LwordLRLRR123   
+ \LwordRLRLR123 
- \LwordLLRRR132 + \LwordRLLRR132 
 - \LwordLRRLR123.
\end{multline}
\begin{multline}
\label{eq::triang::reduction::2}
\begin{tikzpicture}[scale=0.22,xscale=0.55]
   	\node[int] (v0) at (0,-2) {};
    \node (l0) at (0,-3.8) {};
    \draw (v0) -- (l0);
	\node[ext] (v1) at (0,0) {};
    \draw[dotted] (v0) -- (v1);    
	\node[int] (v2) at (2,2) {};
    \draw (v1) -- (v2);
    \node (l1) at (-2,2) {$\scriptstyle 1$};
    \draw[dotted] (v1) -- (l1);
    \node[ext] (v3) at (2,4) {};
    \draw[dotted] (v2) -- (v3);
    \node (l2) at (0,6) {$\scriptstyle 2$};
    \draw[dotted] (v3) -- (l2);
    \node[int] (v4) at (4,6) {};
    \draw (v3) -- (v4);
    \node (l3) at (4,8.3) {$\scriptstyle 3$};
    \draw[dotted] (v4) -- (l3);
    \ovalthree{v0}{v1}{v2};
\end{tikzpicture}
= \RwordLRRLR123 - \LRwordRRLLR123
=
\Bigl(\RLwordLRLRR123 -\LwordLRRLR123 \Bigr) - 
\Bigl( \LRwordRLLRR123 - \LwordRLLRR132\Bigr) =  
\\
= \Bigl(\LwordLLRRR123- \LwordLLRRR132\Bigr) - \LwordLRRLR123
-\Bigl( \LwordLRLRR123 - \LwordRLRLR123 \Bigr) + \LwordRLLRR132.
\end{multline}
The right-hand sides of reductions~\eqref{eq::triang::reduction::1} and~\eqref{eq::triang::reduction::2} coincide, which implies that the corresponding $S$-polynomial is equal to zero.
\end{proof}

\begin{theorem}
\label{thm::Triang::Lieb::Anick}
Let $(\calQ_{\ldot},d)$ be the quasi-free dioperad generated by the following two collections of skew-symmetric generators:
$$
\begin{array}{rccl}
\forall \sigma\in\bS_{m}, m\geq 2 &
\ell_m:=
 \begin{tikzpicture}[scale =0.25,xscale=0.5]
 \node[ext] (v) at (0,0) {};
 \node (l0) at (0,-2) {};
 \node (l1) at (-6,2.2) {$\scriptscriptstyle 1$};
 \node (l2) at (-4,2.2) {$\scriptscriptstyle 2$};
 \node (l3) at (-1,2.2) {$\scriptscriptstyle \dots$};
 \node (l4) at (4,2.2) {$\scriptscriptstyle m$};
 \draw (v) -- (l0);
 \draw (v) -- (l1);
 \draw (v) -- (l3);
 \draw (v) -- (l4);
 \end{tikzpicture}
    =(-1)^{|\sigma|}
 \begin{tikzpicture}[scale =0.25,xscale=0.5]
 \node[ext] (v) at (0,0) {};
 \node (l0) at (0,-2) {};
 \node (l1) at (-6,2.2) {$\scriptscriptstyle \sigma(1)$};
 \node (l2) at (-4,2.2) {};
 \node (l3) at (-1,2.2) {$\scriptscriptstyle \dots$};
 \node (l4) at (4,2.2) {$\scriptscriptstyle \sigma(m)$};
 \draw (v) -- (l0);
 \draw (v) -- (l2);
 \draw (v) -- (l3);
 \draw (v) -- (l4);
 \end{tikzpicture}
& \in & \calQ(m,1), \\ 
\forall \tau\in\bS_n, n\geq 2 \ & 
r_n:=
 \begin{tikzpicture}[scale =0.2,xscale=0.5,yscale=-1]
 \node[int] (v) at (0,-1) {};
 \node (l1) at (-6,2.2) {$\scriptscriptstyle 1$};
 \node (l2) at (-4,2.2) {$\scriptscriptstyle 2$};
 \node (l3) at (-1,2.2) {$\scriptscriptstyle \dots$};
 \node (l4) at (4,2.2) {$\scriptscriptstyle n$};
 \draw (v) -- (l1);
 \draw (v) -- (l2);
 \draw (v) -- (l3);
 \draw (v) -- (l4);
 \end{tikzpicture} =
 (-1)^{|\tau|}
 \begin{tikzpicture}[scale =0.2,xscale=0.5,yscale=-1]
 \node[int] (v) at (0,-1) {};
 \node (l1) at (-6,2.2) {$\scriptscriptstyle \tau(1)$};
 \node (l2) at (-4,2.2) {};
 \node (l3) at (-1,2.2) {$\scriptscriptstyle \dots$};
 \node (l4) at (4,2.2) {$\scriptscriptstyle \tau(n)$};
 \draw (v) -- (l1);
 \draw (v) -- (l2);
 \draw (v) -- (l3);
 \draw (v) -- (l4);
 \end{tikzpicture} & \in & \calQ(0,n)
\end{array}
$$
(with degrees $2-m$ and $2-n$, respectively), where the differential is given by:
\begin{equation*}
\begin{array}{rcl}
d\left(  \begin{tikzpicture}[scale =0.25,xscale=0.5]
 \node[ext] (v) at (0,0) {};
 \node (l0) at (0,-2) {};
 \node (l1) at (-6,2.2) {$\scriptscriptstyle 1$};
 \node (l2) at (-4,2.2) {$\scriptscriptstyle 2$};
 \node (l3) at (-1,2.2) {$\scriptscriptstyle \dots$};
 \node (l4) at (4,2.2) {$\scriptscriptstyle m$};
 \draw (v) -- (l0);
 \draw (v) -- (l1);
 \draw (v) -- (l3);
 \draw (v) -- (l4);
 \end{tikzpicture}
 \right) & = &
 \sum_{\substack{A \subsetneq [m] \\ \# A \geq 2}} \ \pm
\begin{tikzpicture}[scale=0.25]
\node[ext] (v0) at (0,0) {};
\node[ext] (v1) at (-4,2) {};
\node (l0) at (0,-2) {};
\node (l1) at (-6,4) {};
\node (l2) at (-4,4.2) {$\scriptscriptstyle \dots$};
\node (l3) at (-2,4.2) {};
\node (l4) at (-2,2.2) {};
\node (l5) at (1,2.2) {$\scriptscriptstyle \dots$};
\node (l6) at (4,2.2) {};
\draw (l0) -- (v0);
\draw (v0) -- (v1); 
\draw (v1) -- (l1);
\draw (v1) -- (l2);
\draw (v1) -- (l3);
\draw (v0) -- (l4);
\draw (v0) -- (l5);
\draw (v0) -- (l6);
\draw [decorate, decoration={brace, amplitude=3pt}] (l1) -- (l3) 
        node [midway, yshift=0.4cm] {\scalebox{0.7}{$\scriptstyle A$}};
\draw [decorate, decoration={brace, amplitude=3pt}] (l4) -- (l6) 
        node [midway, yshift=0.4cm] {\scalebox{0.7}{$\scriptstyle [m]\setminus A$}};        
\end{tikzpicture},
\\
d\left(
 \begin{tikzpicture}[scale =0.2,xscale=0.5,yscale=-1]
 \node[int] (v) at (0,-1) {};
 \node (l1) at (-6,2.2) {$\scriptscriptstyle 1$};
 \node (l2) at (-4,2.2) {$\scriptscriptstyle 2$};
 \node (l3) at (-1,2.2) {$\scriptscriptstyle \dots$};
 \node (l4) at (4,2.2) {$\scriptscriptstyle n$};
 \draw (v) -- (l1);
 \draw (v) -- (l2);
 \draw (v) -- (l3);
 \draw (v) -- (l4);
 \end{tikzpicture}
\right) & = &
 \sum_{k\geq 2, [n]=\sqcup [n_\ldot],\atop { n_{0}=1, n_1,...,n_k\geq 1}}\ \pm  
\begin{tikzpicture}[baseline=-0.6ex, scale=0.6]
    \node[ext] (L) at (-2.5, -0.2) {};
    \node[bu] (B) at (-1.4, 0.5) {};
    \node[bu] (C) at (-0.8, 0.5) {};
    \node[bu] (D) at (0.3, 0.5) {};
    \node at (-0.3, 0.5) {\tiny $\dots$};
    
    \draw (D) -- (L);
    \draw (C) -- (L);
    \draw (B) -- (L);
    \draw (L) -- (-2.5, -0.8); 
    
    \draw (B) -- (-1.8, -0.8); \draw (B) -- (-1.3, -0.8); \node at (-1.5, -0.7) {\tiny $\dots$};
    \draw (C) -- (-1.1, -0.8); \draw (C) -- (-0.6, -0.8); \node at (-0.85, -0.7) {\tiny $\dots$};
    \draw (D) -- (0.1, -0.8); \draw (D) -- (0.5, -0.8); \node at (0.3, -0.7) {\tiny $\dots$};
    
    \node at (-0.5, 1.2) {$\overbrace{\hspace{1.5cm}}^{k}$};
    
    \node at (-2.5, -1.1) {\tiny $\scriptscriptstyle \underbrace{\ }_{n_0}$};
    \node at (-1.55, -1.1) {\tiny $\underbrace{\ }_{n_1}$};
    \node at (-0.85, -1.1) {\tiny $\underbrace{\ }_{n_2}$};
    \node at (0.3, -1.1) {\tiny $\underbrace{\ }_{n_k}$};
\end{tikzpicture}.
\end{array}
\end{equation*}
Then the dioperad morphism $(\calQ_{\ldot},d) \rightarrow \LieTriang$ that sends $\ell_2$ to the Lie bracket, $r_2$ to the copairing, and all other generators to zero is the minimal model of the dioperad $\LieTriang$.
\end{theorem}
\begin{proof}
To prove the theorem, we establish the following three conditions:

\textbf{First}, we verify that the aforementioned map $d$ is indeed a differential ($d^2=0$). This verification is straightforward to check pictorially; we refer to~\cite{Merkulov_Properad_Twisting} for the detailed verification.

\textbf{Second}, we show that the map $\calQ_{\ldot} \rightarrow \LieTriang$ is a surjective morphism of dioperads. Note that all generators of $\calQ_{\ldot}$ are negatively graded except for $\ell_2$ and $r_2$. The zero-th cohomology of $\calQ_{\ldot}$ is the free dioperad generated by $\ell_2$ and $r_2$ subject to the relations arising from the differential of $\ell_3$ (the Jacobi identity) and the differential of $r_3$ (the classical Yang-Baxter equation); thus, it coincides with $\LieTriang$.

\textbf{Third}, we demonstrate that all negatively graded cohomology of $\calQ_{\ldot}$ vanishes. To this end, we apply the Inclusion-Exclusion resolutions discussed in Definition~\ref{def::Anick::chains}. It suffices to show that the higher cohomology vanishes for the $2$-colored quasi-free shuffle operad $\shuffle(\Psi(\calQ_{\ldot}))$, as both functors $\Psi$ and $\shuffle$ are exact and do not affect the underlying vector spaces.

For the shuffle operad, we consider the decreasing filtration $F$ arising from the Gr\"obner theory of monomials. Specifically, we define:
$$
F_k(\Psi(\calQ_{\ldot})):= \mathsf{Span}\Bigl\langle 
\begin{array}{c}
\text{shuffle tree monomials with} \\
\text{at least $k$ internal edges growing to the left} 
\end{array}
\Bigr\rangle.
$$
The associated graded differential $d_{\gr F}$ may only create edges that grow to the right in the shuffle description. This yields the following pictorial description of the associated graded differential $d_{\gr F}$:
$$
\begin{array}{ccl}
\forall m\geq 2, \
d_{\gr F}(\Psi(\ell_m)) =
d_{\gr F}\left(  
\begin{tikzpicture}[scale =0.25,xscale=0.5]
 \node[ext] (v) at (0,0) {};
 \node (l0) at (0,-2) {};
 \node (l1) at (-6,2.2) {$\scriptscriptstyle 1$};
 \node (l2) at (-4,2.2) {$\scriptscriptstyle 2$};
 \node (l3) at (-1,2.2) {$\scriptscriptstyle \dots$};
 \node (l4) at (4,2.2) {$\scriptscriptstyle m$};
 \draw (v) -- (l0);
 \draw (v) -- (l1);
 \draw (v) -- (l3);
 \draw (v) -- (l4);
 \end{tikzpicture}
 \right) & = &
 \sum_{j=1}^{m-2} \ \pm
\begin{tikzpicture}[scale=0.25]
\node[ext] (v0) at (0,0) {};
\node[ext] (v1) at (4,2) {};
\node (l0) at (0,-2) {};
\node (l1) at (6,4) {$\scriptscriptstyle m$};
\node (l2) at (4,4.2) {$\scriptscriptstyle \dots$};
\node (l3) at (2,4.2) {$\scriptscriptstyle j+1$};
\node (l4) at (2,2.2) {$\scriptscriptstyle j$};
\node (l5) at (-1,2.2) {$\scriptscriptstyle \dots$};
\node (l6) at (-4,2.2) {$\scriptscriptstyle 1$};
\draw (l0) -- (v0);
\draw (v0) -- (v1); 
\draw (v1) -- (l1);
\draw (v1) -- (l2);
\draw (v1) -- (l3);
\draw (v0) -- (l4);
\draw (v0) -- (l5);
\draw (v0) -- (l6);
\end{tikzpicture}, \\
\forall n\geq 1, \ 
d_{\gr F}(\Psi(r_{n+1})) = d_{\gr F}\left( 
\begin{tikzpicture}[scale =0.25,xscale=0.5]
 \node[int] (v) at (0,0) {};
 \node (l0) at (0,-2) {};
 \node (l1) at (-6,2.2) {$\scriptscriptstyle 1$};
 \node (l2) at (-4,2.2) {$\scriptscriptstyle 2$};
 \node (l3) at (-1,2.2) {$\scriptscriptstyle \dots$};
 \node (l4) at (4,2.2) {$\scriptscriptstyle n$};
 \draw (v) -- (l0);
 \draw[dotted] (v) -- (l1);
 \draw[dotted] (v) -- (l3);
 \draw[dotted] (v) -- (l4);
 \end{tikzpicture}
\right)
& = & 
 \sum_{j=0}^{n-1} \ \pm
\begin{tikzpicture}[scale=0.25]
\node[int] (v0) at (0,0) {};
\node[ext] (v) at (3,1.5) {};
\node[int] (v1) at (5,3) {};
\node (l0) at (0,-2) {};
\node (l1) at (7,5.2) {$\scriptscriptstyle n$};
\node (l2) at (5,5.2) {$\scriptscriptstyle \dots$};
\node (l3) at (3,5.2) {$\scriptscriptstyle j+1$};
\node (l4) at (1.5,4.2) {$\scriptscriptstyle j$};
\node (l5) at (-1,2.2) {$\scriptscriptstyle j-1$};
\node (l) at (-3,2.2) {$\scriptscriptstyle \dots$};
\node (l6) at (-4,2.2) {$\scriptscriptstyle 1$};
\draw (l0) -- (v0);
\draw[dotted] (v0) -- (v); 
\draw (v) -- (v1);
\draw[dotted] (v1) -- (l1);
\draw[dotted] (v1) -- (l2);
\draw[dotted] (v1) -- (l3);
\draw[dotted] (v) -- (l4);
\draw[dotted] (v0) -- (l5);
\draw[dotted] (v0) -- (l6);
\end{tikzpicture}.
\end{array}
$$
Note that $\Psi(\ell_m)$ yields $m+1$ distinct generators depending on which input or output is chosen as the root. We illustrate only the case corresponding to the "straight" color; the others differ by assigning the "dotted" color to one of the inputs and the output. The right-hand side of the differential also varies according to the corresponding coloring.

Finally, one observes that these generators coincide with the generators of the corresponding inclusion-exclusion operad associated with the following shuffle monomials (where corresponding divisors are encircled in red):
\begin{equation*}
\label{eq::Anick::chains::Triang}
\begin{tikzpicture}[scale=0.2, yscale=1.4]
\node[ext] (v0) at (0,0) {};
\node[ext] (v1) at (2,1) {};
\node[ext] (v2) at (4,2) {};
\node (v3) at (6,3) {$\scriptstyle \ldots$};
\node[ext] (v4) at (8,4) {};
\node (u1) at (0,2) {$\scriptscriptstyle 2$};
\node (u2) at (5.5,5.5) {$\scriptscriptstyle m-1$};
\node (l0) at (0,-1.5) {};
\node (l1) at (-2,1) {$\scriptscriptstyle 1$};
\node (l2) at (0,4.2) {};
\node (l3) at (2,3) {};
\node (l5) at (10,5) {$\scriptscriptstyle m$};
\draw (l0) -- (v0);
\draw (v0) -- (l1);
\draw (v0) -- (v1);
\draw (v1) -- (u1);
\draw (v1) -- (v2);
\draw (v2) -- (l3);
\draw (v2) -- (v3);
\draw (v3) -- (v4);
\draw (v4) -- (u2);
\draw (v4) -- (l5);
\ovaltwo{v0}{v1};
\ovaltwo{v1}{v2};
\ovaltwo{v3}{v4};
\end{tikzpicture},
\qquad
\text{ for }j=1\dots m \ 
\begin{tikzpicture}[scale=0.2, xscale = 1]
\node[ext] (v0) at (0,0) {};
\node[ext] (v1) at (2,1) {};
\node (v2i) at (4,2) {};
\node (v2) at (4,2) {$\scriptscriptstyle \ldots$};
\node[ext] (v3) at (6,3) {};
\node[ext] (v4) at (8,4) {};
\node (v5) at (10,5) {$\scriptscriptstyle \ldots$};
\node[ext] (v6) at (12,6) {};
\node (u1) at (6,5) {};
\node (u2) at (14.5,7.5) {$\scriptscriptstyle m$};
\node (l0) at (0,-2) {};
\node (l1) at (-2,1) {$\scriptscriptstyle 1$};
\node (l2) at (0,2) {$\scriptscriptstyle 2$};
\node (l3) at (3,4) {$\scriptscriptstyle j$};
\node (l4) at (6,7) {};
\node (l5) at (9.5,7.5) {$\scriptscriptstyle m-1$};
\draw[dotted] (l0) -- (v0);
\draw (v0) -- (l1);
\draw[dotted] (v0) -- (v1);
\draw (v1) -- (l2);
\draw[dotted] (v1) -- (v2i);
\draw[dotted] (v2i) -- (v3);
\draw[dotted] (v3) -- (l3);
\draw (v3) -- (v4);
\draw (v4) -- (u1);
\draw (v4) -- (v5);
\draw (v5) -- (v6);
\draw (v6) -- (l5);
\draw (v6) -- (u2);
\ovaltwo{v0}{v1};
\ovaltwo{v3}{v4};
\ovaltwo{v5}{v6};
\end{tikzpicture},
\
\begin{tikzpicture}[scale=0.22, xscale =1.5, yscale=0.6]
    \node[int] (v0) at (0,-2) {};
    \node (l0) at (0,-4.8) {};
    \draw (v0) -- (l0);
	\node[ext] (v1) at (0,0) {};
    \draw[dotted] (v0) -- (v1);    
	\node[int] (v2) at (2,2) {};
    \draw (v1) -- (v2);
    \node (l1) at (-2,2) {$\scriptscriptstyle 1$};
    \draw[dotted] (v1) -- (l1);
    \node[ext] (v3) at (2,4) {};
    \draw[dotted] (v2) -- (v3);
    \node (l2) at (0,6) {$\scriptscriptstyle 2$};
    \draw[dotted] (v3) -- (l2);
    \node[int] (v4) at (4,6) {};
    \draw (v3) -- (v4);
    \node (l3) at (4,9.8) {};
    \draw[dotted] (v4) -- (l3);
    \node (l4) at (5,8.3) {$\scriptscriptstyle \ldots$};
    \node (l5) at (6,8) {};
    \node[int] (v5) at (7,10) {};
    \draw (l4) -- (v5);
    \node[ext] (v6) at (7,12) {};
    \draw[dotted] (v5) -- (v6);
    \node (l6) at (4.5,15) {$\scriptscriptstyle n-1$};
    \draw[dotted] (v6) -- (l6);
    \node[int] (v7) at (9,14) {};
    \node (l7) at (7,16) {$\scriptscriptstyle n$};
    \draw (v6) -- (v7);
    \draw[dotted] (v7) -- (l7);
    \ovalthree{v0}{v1}{v2};
\ovalthree{v2}{v3}{v4};
\ovalthree{v5}{v6}{v7};
\end{tikzpicture}.
\end{equation*}
Here, the first two diagrams correspond to different colorings of the inputs/output in the generator $\ell_m$, while the final diagram corresponds to the generator $\Psi(r_{n+1})$. Consequently, the higher homology vanishes for the associated graded differential $d_{\gr F}$, which completes the proof.
\end{proof}

\subsection{The dioperad $\calV^{(d)}$ of Tradler and Zeinalian}
\label{example::Tradler}

The representations of the dioperad $\calV^{(d)}$, introduced by Tradler and Zeinalian in \cite{Tradler_Zeinalian}, consist of an associative algebra $(A, \cdot)$ equipped with a symmetric and invariant element $c \in S^2 A$ of degree $d$. Writing $c = c^{(1)} \otimes c^{(2)}$ in Sweedler's notations, the symmetry and invariance conditions are expressed as:
\begin{gather*}
     c^{(1)} \otimes c^{(2)} =  (-1)^{|c^{(1)}||c^{(2)}|} c^{(2)} \otimes c^{(1)}, \\
     (a \cdot c^{(1)}) \otimes c^{(2)} = (-1)^{|a|(|c^{(1)}| + |c^{(2)}|)} c^{(1)} \otimes (c^{(2)}\cdot a) \quad \text{for all } a \in A,
\end{gather*}
respectively. 

From a diagrammatic perspective, the dioperad $\calV^{(d)}$ is generated by the following subspaces:
$$
\calV^{(d)}(2,1) := \Bbbk[\bS_2] = \mathsf{span}\left\langle 
\liePic{1}{2}, \liePic{2}{1}
\right\rangle; \qquad \calV^{(d)}(0,2) := \one_2[d] =
\mathsf{span}\left\langle 
\RLie{1}{2} = \RLie{2}{1}  
\right\rangle,
$$
subject to the quadratic relations:
\begin{equation}
\label{eq::ref::Tradler}  
\left\{
\LLiePic{1}{2}{3} = \LLiePicReverse{1}{2}{3};
\qquad
\mVPic{1}{2} = \mVPicT{1}{2} 
\right\}.
\end{equation}
The Koszul property of $\calV^{(d)}$ was originally verified in~\cite{Tradler}. Below, we provide a significantly simpler proof of this fact using the Gr\"obner basis theory introduced earlier.

\begin{notation}
The $2$-colored non-symmetric operad
$\VV^{(d)}$ is generated by three binary operations of degree $0$ and a single unary operation of degree $d$:
$$
\begin{array}{ccccccc}
\begin{tikzpicture}[scale=0.15pt, baseline=-0.5ex]
    \draw (0,-0.55) -- (0,-2.5);
    \draw (0.5,0.5) -- (2.2,2.2);
    \draw (-0.48,0.48) -- (-2.2,2.2);
    \node[ext] (A) at (0,0) {};  
\end{tikzpicture}
& , &
\begin{tikzpicture}[scale=0.15pt, baseline=-0.5ex]
    \draw[dotted] (0,-0.55) -- (0,-2.5);
    \draw (0.5,0.5) -- (2.2,2.2);
    \draw[dotted] (-0.48,0.48) -- (-2.2,2.2);
    \node[ext] (A) at (0,0) {};
\end{tikzpicture}
& , &
\begin{tikzpicture}[scale=0.15pt, baseline=-0.5ex]
    \draw[dotted] (0,-0.55) -- (0,-2.5);
    \draw[dotted] (0.5,0.5) -- (2.2,2.2);
    \draw (-0.48,0.48) -- (-2.2,2.2);
    \node[ext] (A) at (0,0) {}; 
\end{tikzpicture}
& , & 
\begin{tikzpicture}[scale=0.15pt, baseline=-0.5ex]
    \draw (0,-0.55) -- (0,-2.5);
    \draw[dotted] (0,0.5) -- (0,2.5);
    \node[int] (A) at (0,0) {}; 
\end{tikzpicture}
\end{array}
$$
subject to the following relations:
\begin{gather}
\label{eq::Tradler::ass}
\begin{tikzpicture}[scale=0.2]
\node[ext] (a) at (0,0) {};
\node[ext] (b) at (-1.5,2) {};
\node (l0) at (0,-2) {};
\node (l1) at (-3,4.3) {};
\node (l2) at (0, 4.3) {};
\node (l3) at (1.5,2.3) {};
\draw (a) -- (l0);
\draw (a) -- (b);
\draw (b) -- (l1);
\draw (b) -- (l2);
\draw (a) -- (l3);
\end{tikzpicture}
=
\begin{tikzpicture}[scale=0.2,xscale=-1]
\node[ext] (a) at (0,0) {};
\node[ext] (b) at (-1.5,2) {};
\node (l0) at (0,-2) {};
\node (l1) at (-3,4.3) {};
\node (l2) at (0, 4.3) {};
\node (l3) at (1.5,2.3) {};
\draw (a) -- (l0);
\draw (a) -- (b);
\draw (b) -- (l1);
\draw (b) -- (l2);
\draw (a) -- (l3);
\end{tikzpicture},
\quad 
\begin{tikzpicture}[scale=0.2]
\node[ext] (a) at (0,0) {};
\node[ext] (b) at (-1.5,2) {};
\node (l0) at (0,-2) {};
\node (l1) at (-3,4.3) {};
\node (l2) at (0, 4.3) {};
\node (l3) at (1.5,2.3) {};
\draw[dotted] (a) -- (l0);
\draw[dotted] (a) -- (b);
\draw[dotted] (b) -- (l1);
\draw (b) -- (l2);
\draw (a) -- (l3);
\end{tikzpicture}
=
\begin{tikzpicture}[scale=0.2,xscale=-1]
\node[ext] (a) at (0,0) {};
\node[ext] (b) at (-1.5,2) {};
\node (l0) at (0,-2) {};
\node (l1) at (-3,4.3) {};
\node (l2) at (0, 4.3) {};
\node (l3) at (1.5,2.3) {};
\draw[dotted] (a) -- (l0);
\draw (a) -- (b);
\draw (b) -- (l1);
\draw (b) -- (l2);
\draw[dotted] (a) -- (l3);
\end{tikzpicture},
\quad 
\begin{tikzpicture}[scale=0.2]
\node[ext] (a) at (0,0) {};
\node[ext] (b) at (-1.5,2) {};
\node (l0) at (0,-2) {};
\node (l1) at (-3,4.3) {};
\node (l2) at (0, 4.3) {};
\node (l3) at (1.5,2.3) {};
\draw[dotted] (a) -- (l0);
\draw[dotted] (a) -- (b);
\draw (b) -- (l1);
\draw[dotted] (b) -- (l2);
\draw (a) -- (l3);
\end{tikzpicture}
=
\begin{tikzpicture}[scale=0.2,xscale=-1]
\node[ext] (a) at (0,0) {};
\node[ext] (b) at (-1.5,2) {};
\node (l0) at (0,-2) {};
\node (l1) at (-3,4.3) {};
\node (l2) at (0, 4.3) {};
\node (l3) at (1.5,2.3) {};
\draw[dotted] (a) -- (l0);
\draw[dotted] (a) -- (b);
\draw (b) -- (l1);
\draw[dotted] (b) -- (l2);
\draw (a) -- (l3);
\end{tikzpicture},
\quad 
\begin{tikzpicture}[scale=0.2]
\node[ext] (a) at (0,0) {};
\node[ext] (b) at (-1.5,2) {};
\node (l0) at (0,-2) {};
\node (l1) at (-3,4.3) {};
\node (l2) at (0, 4.3) {};
\node (l3) at (1.5,2.3) {};
\draw[dotted] (a) -- (l0);
\draw (a) -- (b);
\draw (b) -- (l1);
\draw (b) -- (l2);
\draw[dotted] (a) -- (l3);
\end{tikzpicture}
=
\begin{tikzpicture}[scale=0.2,xscale=-1]
\node[ext] (a) at (0,0) {};
\node[ext] (b) at (-1.5,2) {};
\node (l0) at (0,-2) {};
\node (l1) at (-3,4.3) {};
\node (l2) at (0, 4.3) {};
\node (l3) at (1.5,2.3) {};
\draw[dotted] (a) -- (l0);
\draw[dotted] (a) -- (b);
\draw[dotted] (b) -- (l1);
\draw (b) -- (l2);
\draw (a) -- (l3);
\end{tikzpicture};
\\ 
\label{eq::Tradler::copairing}
\begin{tikzpicture}[scale=0.2]
\node[ext] (a) at (0,2) {};
\node[int] (b) at (0,0) {};
\node (l0) at (0,-2) {};
\node (l1) at (-1.5,4.2) {};
\node (l2) at (1.5, 4.2) {};
\draw (b) -- (l0);
\draw[dotted] (a) -- (b);
\draw[dotted] (a) -- (l1);
\draw (a) -- (l2);
\end{tikzpicture}=
\begin{tikzpicture}[scale=0.2]
\node[ext] (a) at (0,0) {};
\node[int] (b) at (-1.5,2) {};
\node (l0) at (0,-2) {};
\node (l1) at (-1.5,4.4) {};
\node (l2) at (1.5, 2.2) {};
\draw (a) -- (l0);
\draw (a) -- (b);
\draw[dotted] (b) -- (l1);
\draw (a) -- (l2);
\end{tikzpicture}, 
\quad 
\begin{tikzpicture}[scale=0.2,xscale=-1]
\node[ext] (a) at (0,2) {};
\node[int] (b) at (0,0) {};
\node (l0) at (0,-2) {};
\node (l1) at (-1.5,4.2) {};
\node (l2) at (1.5, 4.2) {};
\draw (b) -- (l0);
\draw[dotted] (a) -- (b);
\draw[dotted] (a) -- (l1);
\draw (a) -- (l2);
\end{tikzpicture}
=
\begin{tikzpicture}[scale=0.2,xscale=-1]
\node[ext] (a) at (0,0) {};
\node[int] (b) at (-1.5,2.2) {};
\node (l0) at (0,-2) {};
\node (l1) at (-1.5,4.4) {};
\node (l2) at (1.5, 2.2) {};
\draw (a) -- (l0);
\draw (a) -- (b);
\draw[dotted] (b) -- (l1);
\draw (a) -- (l2);
\end{tikzpicture}. 
\end{gather}
\end{notation}

\begin{proposition}
\label{prp::Sym::Tradler}
    The $2$-colored symmetric operad $\Psi(\calV^{(d)})$ and the symmetrization $\mathsf{Sym}(\VV^{(d)})$ of the non-symmetric $2$-colored operad $\VV^{(d)}$ are isomorphic.
\end{proposition}
\begin{proof}
    Note that the relations in~\eqref{eq::ref::Tradler} are planar (i.e., they do not permute the order of inputs or outputs as drawn in the plane). A direct comparison of these relations completes the proof. Specifically, relations~\eqref{eq::Tradler::ass} arise from selecting the input in the associativity relation, while relations~\eqref{eq::Tradler::copairing} correspond to the second relation in~\eqref{eq::ref::Tradler}.
\end{proof}

\begin{theorem}
\label{thm::Tradler::Grobner}
The rewriting system that maps the left-hand sides to the right-hand sides in relations~\eqref{eq::Tradler::ass} and~\eqref{eq::Tradler::copairing} is convergent.
\end{theorem}
\begin{proof}
Following the framework of quantum-monomial orderings introduced in~\cite{Dotsenko_QP-order}, we define an admissible monomial ordering such that the left-hand sides of the relations are the leading monomials. This choice automatically ensures the termination of the rewriting system. Consequently, these relations constitute a Gr\"obner basis with respect to this ordering (see~\cite{Dotsenko_Vallette_Nonsym} for a detailed treatment of Gr\"obner bases for non-symmetric operads).

Recall that monomials in the ring of quantum polynomials
$$\Bbbk\langle x,y,q \mid xq=qx, yq=qy, yx =q xy \rangle$$
admit a multiplication-compatible ordering defined by:
$$
x^{k_1} y^{l_1} q^{m_1} >  x^{k_2} y^{l_2} q^{m_2} \ \Leftrightarrow \ 
\begin{cases}
k_1 > k_2, \\
k_1 = k_2 \text{ and } l_1 > l_2, \\
k_1 = k_2, l_1 = l_2, \text{ and } m_1 > m_2.
\end{cases}
$$

To each tree monomial $T$ with $n$ leaves, we assign a collection of $n$ quantum monomials. The $i$-th monomial is the noncommutative word in variables $x$ and $y$ corresponding to the path from the root to the $i$-th input, where $x$ represents associative multiplication and $y$ represents the unary operation. To compare tree monomials $T_1$ and $T_2$, we compare their corresponding quantum monomials lexicographically. We omit the proof that this is an admissible ordering, as it follows the construction in~\cite{Dotsenko_QP-order}.

Confluence for compositions of the associativity relations follows from the standard confluence of associativity. Confluence for the remaining compositions is verified below for two of the three possible colorings of inputs (the third case being entirely analogous):
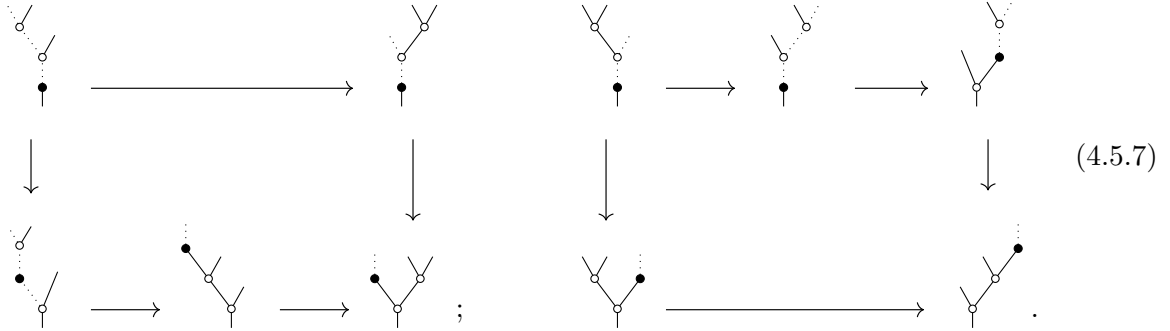
\begin{equation}
\label{eq::Tradler::S:polynom}
\begin{tikzcd}
\begin{tikzpicture}[scale=0.2]
\node[ext] (a) at (0,2) {};
\node[int] (b) at (0,0) {};
\node[ext] (c) at (-1.5, 4) {};
\node (l0) at (0,-2) {};
\node (l1) at (-3,6.2) {};
\node (l2) at (0, 6.2) {};
\node (l3) at (1.5,4.2) {};
\draw (b) -- (l0);
\draw[dotted] (b) -- (a);
\draw[dotted] (a) -- (c);
\draw (a) -- (l3);
\draw[dotted] (c) -- (l1);
\draw (c) -- (l2);
\end{tikzpicture}
\arrow[d] \arrow[rr]
&
&
\begin{tikzpicture}[scale=0.2]
\node[ext] (a) at (0,2) {};
\node[int] (b) at (0,0) {};
\node[ext] (c) at (1.5, 4) {};
\node (l0) at (0,-2) {};
\node (l3) at (3,6.2) {};
\node (l2) at (0, 6.2) {};
\node (l1) at (-1.5,4.2) {};
\draw (b) -- (l0);
\draw[dotted] (b) -- (a);
\draw (a) -- (c);
\draw[dotted] (a) -- (l1);
\draw (c) -- (l2);
\draw (c) -- (l3);
\end{tikzpicture}
\arrow[d]
& 
\begin{tikzpicture}[scale=0.2]
\node[ext] (a) at (0,2) {};
\node[int] (b) at (0,0) {};
\node[ext] (c) at (-1.5, 4) {};
\node (l0) at (0,-2) {};
\node (l1) at (-3,6.2) {};
\node (l2) at (0, 6.2) {};
\node (l3) at (1.5,4.2) {};
\draw (b) -- (l0);
\draw[dotted] (b) -- (a);
\draw (a) -- (c);
\draw[dotted] (a) -- (l3);
\draw (c) -- (l1);
\draw (c) -- (l2);
\end{tikzpicture}
\arrow[r] \arrow[d]
&
\begin{tikzpicture}[scale=0.2]
\node[ext] (a) at (0,2) {};
\node[int] (b) at (0,0) {};
\node[ext] (c) at (1.5, 4) {};
\node (l0) at (0,-2) {};
\node (l3) at (3,6.2) {};
\node (l2) at (0, 6.2) {};
\node (l1) at (-1.5,4.2) {};
\draw (b) -- (l0);
\draw[dotted] (b) -- (a);
\draw[dotted] (a) -- (c);
\draw (a) -- (l1);
\draw (c) -- (l2);
\draw[dotted] (c) -- (l3);
\end{tikzpicture}
\arrow[r]
&
\begin{tikzpicture}[scale=0.2]
\node[int] (a) at (1.5,2) {};
\node[ext] (b) at (0,0) {};
\node[ext] (c) at (1.5, 4.2) {};
\node (l0) at (0,-2) {};
\node (l3) at (3,6.4) {};
\node (l2) at (0, 6.4) {};
\node (l1) at (-1.5,3.2) {};
\draw (b) -- (l0);
\draw (b) -- (a);
\draw[dotted] (a) -- (c);
\draw[dotted] (c) -- (l3);
\draw (c) -- (l2);
\draw (b) -- (l1);
\end{tikzpicture}
\arrow[d]
\\
\begin{tikzpicture}[scale=0.2]
\node[int] (a) at (-1.5,2) {};
\node[ext] (b) at (0,0) {};
\node[ext] (c) at (-1.5, 4.2) {};
\node (l0) at (0,-2) {};
\node (l1) at (-3,6.2) {};
\node (l2) at (0, 6.2) {};
\node (l3) at (1.5,3.2) {};
\draw (b) -- (l0);
\draw[dotted] (b) -- (a);
\draw[dotted] (a) -- (c);
\draw (b) -- (l3);
\draw[dotted] (c) -- (l1);
\draw (c) -- (l2);
\end{tikzpicture}
\arrow[r]
&
\begin{tikzpicture}[scale=0.2]
\node[ext] (a) at (-1.5,2) {};
\node[ext] (b) at (0,0) {};
\node[int] (c) at (-3, 4) {};
\node (l0) at (0,-2) {};
\node (l1) at (-3,6.4) {};
\node (l2) at (0, 4.2) {};
\node (l3) at (1.5,2.2) {};
\draw (b) -- (l0);
\draw (b) -- (a);
\draw (a) -- (c);
\draw (b) -- (l3);
\draw[dotted] (c) -- (l1);
\draw (a) -- (l2);
\end{tikzpicture}
\arrow[r]
& 
\begin{tikzpicture}[scale=0.2]
\node[ext] (a) at (1.5,2) {};
\node[ext] (b) at (0,0) {};
\node[int] (c) at (-1.5, 2) {};
\node (l0) at (0,-2) {};
\node (l1) at (-1.5,4.4) {};
\node (l2) at (0, 4.2) {};
\node (l3) at (3,4.2) {};
\draw (b) -- (l0);
\draw (b) -- (a);
\draw (b) -- (c);
\draw (a) -- (l3);
\draw[dotted] (c) -- (l1);
\draw (a) -- (l2);
\end{tikzpicture};
&
\begin{tikzpicture}[scale=0.2,xscale=-1]
\node[ext] (a) at (1.5,2) {};
\node[ext] (b) at (0,0) {};
\node[int] (c) at (-1.5, 2) {};
\node (l0) at (0,-2) {};
\node (l1) at (-1.5,4.4) {};
\node (l2) at (0, 4.2) {};
\node (l3) at (3,4.2) {};
\draw (b) -- (l0);
\draw (b) -- (a);
\draw (b) -- (c);
\draw (a) -- (l3);
\draw[dotted] (c) -- (l1);
\draw (a) -- (l2);
\end{tikzpicture}
\arrow[rr]
&
&
\begin{tikzpicture}[scale=0.2,xscale=-1]
\node[ext] (a) at (-1.5,2) {};
\node[ext] (b) at (0,0) {};
\node[int] (c) at (-3, 4) {};
\node (l0) at (0,-2) {};
\node (l1) at (-3,6.4) {};
\node (l2) at (0, 4.2) {};
\node (l3) at (1.5,2.2) {};
\draw (b) -- (l0);
\draw (b) -- (a);
\draw (a) -- (c);
\draw (b) -- (l3);
\draw[dotted] (c) -- (l1);
\draw (a) -- (l2);
\end{tikzpicture}.
\end{tikzcd}
\end{equation}
\end{proof}

Thus, relation~\ref{eq::Tradler::copairing} defines a distributive law:
$$
\lambda: 
\left\langle 
\begin{tikzpicture}[scale=0.15pt, baseline=-0.5ex]
    \draw (0,-0.55) -- (0,-2.5);
    \draw[dotted] (0,0.5) -- (0,2.5);
    \node[int] (A) at (0,0) {}; 
\end{tikzpicture}
\right\rangle \circ \Psi(\mathsf{Assoc}) \longrightarrow \Psi(\mathsf{Assoc}) \circ \left\langle 
\begin{tikzpicture}[scale=0.15pt, baseline=-0.5ex]
    \draw (0,-0.55) -- (0,-2.5);
    \draw[dotted] (0,0.5) -- (0,2.5);
    \node[int] (A) at (0,0) {}; 
\end{tikzpicture}
\right\rangle
$$
between the $2$-colored operad associated with the non-symmetric associative operad (viewed as a dioperad) and the $2$-colored operad consisting of a single unary operation. We refer to~\cite{Markl::Distributive} for the original definition of distributive laws and to~\cite[\S8.6]{Loday::Vallette} for a refined exposition.

\begin{corollary}
The non-symmetric $2$-colored operad $\VV$ is Koszul. Furthermore, for each coloring of inputs/outputs, there exists a unique normal monomial in which all multiplications precede the copairing and grow to the right:
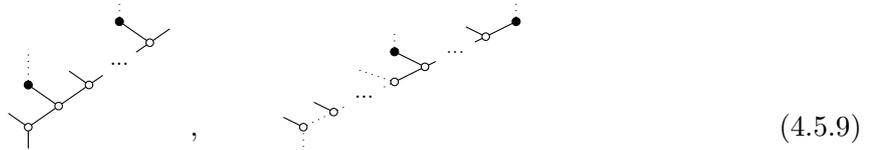
\begin{equation}
\label{eq::Tradler:normal::forms}
\begin{tikzpicture}[scale=0.2, yscale=1.4]
\node[ext] (v0) at (0,0) {};
\node[ext] (v1) at (2,1) {};
\node[ext] (v2) at (4,2) {};
\node (v3) at (6,3) {$\scriptstyle \ldots$};
\node[ext] (v4) at (8,4) {};
\node[int] (u1) at (0,2) {};
\node[int] (u2) at (6,5) {};
\node (l0) at (0,-1.5) {};
\node (l1) at (-2,1) {};
\node (l2) at (0,4.2) {};
\node (l3) at (2,3) {};
\node (l4) at (6,6.5) {};
\node (l5) at (10,5) {};
\draw (l0) -- (v0);
\draw (v0) -- (l1);
\draw (v0) -- (v1);
\draw (v1) -- (u1);
\draw[dotted] (u1) -- (l2);
\draw (v1) -- (v2);
\draw (v2) -- (l3);
\draw (v2) -- (v3);
\draw (v3) -- (v4);
\draw (v4) -- (u2);
\draw[dotted] (u2) -- (l4);
\draw (v4) -- (l5);
\end{tikzpicture},
\qquad
\begin{tikzpicture}[scale=0.2, xscale = 1]
\node[ext] (v0) at (0,0) {};
\node[ext] (v1) at (2,1) {};
\node (v2) at (4,2) {$\scriptstyle \ldots$};
\node[ext] (v3) at (6,3) {};
\node[ext] (v4) at (8,4) {};
\node (v5) at (10,5) {$\scriptscriptstyle \ldots$};
\node[ext] (v6) at (12,6) {};
\node[int] (u1) at (6,5) {};
\node[int] (u2) at (14,7) {};
\node (l0) at (0,-2) {};
\node (l1) at (-2,1) {};
\node (l2) at (0,2) {};
\node (l3) at (3,4) {};
\node (l4) at (6,7) {};
\node (l5) at (10,7) {};
\node (l6) at (14,9) {};
\draw[dotted] (l0) -- (v0);
\draw (v0) -- (l1);
\draw[dotted] (v0) -- (v1);
\draw (v1) -- (l2);
\draw[dotted] (v1) -- (v2);
\draw[dotted] (v2) -- (v3);
\draw[dotted] (v3) -- (l3);
\draw (v3) -- (v4);
\draw (v4) -- (u1);
\draw[dotted] (u1) -- (l4);
\draw (v4) -- (v5);
\draw (v5) -- (v6);
\draw (v6) -- (l5);
\draw (v6) -- (u2);
\draw[dotted] (u2) -- (l6);
\end{tikzpicture}
\end{equation}
Consequently, the dioperad $\calV^{(d)}$ is Koszul and satisfies
$$\dim \calV^{(d)}(m,n) = (m+n-1)!.$$ 
\end{corollary}
\begin{proof}
First, the existence of a quadratic Gr\"obner basis for $\VV$ (Theorem~\ref{thm::Tradler::Grobner}) implies the Koszulness of the non-symmetric operad $\VV$ (by Corollary~\ref{cor::Grobner->Koszul}). Second, the symmetrization functor $\mathsf{Sym}$ is an exact functor that preserves the Koszul property (see, e.g.,~\cite{Loday::Symmetrization} for the uncolored case). Thus, the $2$-colored operad $\Psi(\calV^{(d)})$ is Koszul, and by Corollary~\ref{cor::Psi::Koszul}, we conclude that the dioperad $\calV^{(d)}$ is Koszul.

The description of normal forms is straightforward: for each possible coloring of inputs and outputs, there is exactly one normal form. Since the number of such colorings is given by a binomial coefficient, we have:
\begin{multline*}
\dim \calV^{(d)}(m,n) = \dim\left( \Psi(\calV^{(d)})^{\strt}(m,n-1) \right) 
= \dim\Bigl(\mathsf{Sym}(\VV^{\strt}(m,n-1)\Bigr) = \\
= m! \cdot (n-1)! \cdot \dim\Bigl(\VV^{\strt}(m,n-1)\Bigr) = m! \cdot (n-1)! \binom{m+n-1}{m} = (m+n-1)!.
\end{multline*}
\end{proof}

It is also asked in~\cite{Tradler} whether the dioperad $\mathcal{W}^{(d)}$, obtained by replacing the symmetric element $c$ with a skew-symmetric one, is Koszul. The diagrammatic description of $\mathcal{W}^{(d)}$ is analogous to the symmetric case. Specifically, $\mathcal{W}^{(d)}$ is generated by the following subspaces:
$$
\begin{array}{c}
\mathcal{W}^{(d)}(2,1) := \Bbbk[\mathbb{S}_2] = \mathsf{span}\left\langle 
\liePic{1}{2}, \liePic{2}{1}
\right\rangle; \\
\mathcal{W}^{(d)}(0,2) := \sgn_2[d] =
\mathsf{span}\left\langle 
\RLie{1}{2} = - \RLie{2}{1}  
\right\rangle,
\end{array}
$$
subject to the quadratic relations:
\begin{equation}
\label{eq::ref::Tradler}  
\left\{
\LLiePic{1}{2}{3} = \LLiePicReverse{1}{2}{3};
\qquad
\mVPic{1}{2} = \mVPicT{1}{2} 
\right\}.
\end{equation}
We claim that in this case, one must introduce a sign in the relation~\eqref{eq::Tradler::copairing}, which fundamentally alters the algebraic structure. To demonstrate this, we define the $2$-colored non-symmetric operad $\mathbb{W}^{(d)}$, generated by three binary operations of degree $0$ and a single unary operation of degree $d$:
$$
\begin{array}{ccccccc}
\begin{tikzpicture}[scale=0.15pt, baseline=-0.5ex]
    \draw (0,-0.55) -- (0,-2.5);
    \draw (0.5,0.5) -- (2.2,2.2);
    \draw (-0.48,0.48) -- (-2.2,2.2);
    \node[ext] (A) at (0,0) {};  
\end{tikzpicture}
& , &
\begin{tikzpicture}[scale=0.15pt, baseline=-0.5ex]
    \draw[dotted] (0,-0.55) -- (0,-2.5);
    \draw (0.5,0.5) -- (2.2,2.2);
    \draw[dotted] (-0.48,0.48) -- (-2.2,2.2);
    \node[ext] (A) at (0,0) {};
\end{tikzpicture}
& , &
\begin{tikzpicture}[scale=0.15pt, baseline=-0.5ex]
    \draw[dotted] (0,-0.55) -- (0,-2.5);
    \draw[dotted] (0.5,0.5) -- (2.2,2.2);
    \draw (-0.48,0.48) -- (-2.2,2.2);
    \node[ext] (A) at (0,0) {}; 
\end{tikzpicture}
& , & 
\begin{tikzpicture}[scale=0.15pt, baseline=-0.5ex]
    \draw (0,-0.55) -- (0,-2.5);
    \draw[dotted] (0,0.5) -- (0,2.5);
    \node[int] (A) at (0,0) {}; 
\end{tikzpicture}
\end{array}
$$
subject to the following relations:
$$
\begin{array}{c}
\begin{tikzpicture}[scale=0.2]
\node[ext] (a) at (0,0) {};
\node[ext] (b) at (-1.5,2) {};
\node (l0) at (0,-2) {};
\node (l1) at (-3,4.3) {};
\node (l2) at (0, 4.3) {};
\node (l3) at (1.5,2.3) {};
\draw (a) -- (l0);
\draw (a) -- (b);
\draw (b) -- (l1);
\draw (b) -- (l2);
\draw (a) -- (l3);
\end{tikzpicture}
=
\begin{tikzpicture}[scale=0.2,xscale=-1]
\node[ext] (a) at (0,0) {};
\node[ext] (b) at (-1.5,2) {};
\node (l0) at (0,-2) {};
\node (l1) at (-3,4.3) {};
\node (l2) at (0, 4.3) {};
\node (l3) at (1.5,2.3) {};
\draw (a) -- (l0);
\draw (a) -- (b);
\draw (b) -- (l1);
\draw (b) -- (l2);
\draw (a) -- (l3);
\end{tikzpicture},
\quad 
\begin{tikzpicture}[scale=0.2]
\node[ext] (a) at (0,0) {};
\node[ext] (b) at (-1.5,2) {};
\node (l0) at (0,-2) {};
\node (l1) at (-3,4.3) {};
\node (l2) at (0, 4.3) {};
\node (l3) at (1.5,2.3) {};
\draw[dotted] (a) -- (l0);
\draw[dotted] (a) -- (b);
\draw[dotted] (b) -- (l1);
\draw (b) -- (l2);
\draw (a) -- (l3);
\end{tikzpicture}
=
\begin{tikzpicture}[scale=0.2,xscale=-1]
\node[ext] (a) at (0,0) {};
\node[ext] (b) at (-1.5,2) {};
\node (l0) at (0,-2) {};
\node (l1) at (-3,4.3) {};
\node (l2) at (0, 4.3) {};
\node (l3) at (1.5,2.3) {};
\draw[dotted] (a) -- (l0);
\draw (a) -- (b);
\draw (b) -- (l1);
\draw (b) -- (l2);
\draw[dotted] (a) -- (l3);
\end{tikzpicture},
\quad 
\begin{tikzpicture}[scale=0.2]
\node[ext] (a) at (0,0) {};
\node[ext] (b) at (-1.5,2) {};
\node (l0) at (0,-2) {};
\node (l1) at (-3,4.3) {};
\node (l2) at (0, 4.3) {};
\node (l3) at (1.5,2.3) {};
\draw[dotted] (a) -- (l0);
\draw[dotted] (a) -- (b);
\draw (b) -- (l1);
\draw[dotted] (b) -- (l2);
\draw (a) -- (l3);
\end{tikzpicture}
=
\begin{tikzpicture}[scale=0.2,xscale=-1]
\node[ext] (a) at (0,0) {};
\node[ext] (b) at (-1.5,2) {};
\node (l0) at (0,-2) {};
\node (l1) at (-3,4.3) {};
\node (l2) at (0, 4.3) {};
\node (l3) at (1.5,2.3) {};
\draw[dotted] (a) -- (l0);
\draw[dotted] (a) -- (b);
\draw (b) -- (l1);
\draw[dotted] (b) -- (l2);
\draw (a) -- (l3);
\end{tikzpicture},
\quad 
\begin{tikzpicture}[scale=0.2]
\node[ext] (a) at (0,0) {};
\node[ext] (b) at (-1.5,2) {};
\node (l0) at (0,-2) {};
\node (l1) at (-3,4.3) {};
\node (l2) at (0, 4.3) {};
\node (l3) at (1.5,2.3) {};
\draw[dotted] (a) -- (l0);
\draw (a) -- (b);
\draw (b) -- (l1);
\draw (b) -- (l2);
\draw[dotted] (a) -- (l3);
\end{tikzpicture}
=
\begin{tikzpicture}[scale=0.2,xscale=-1]
\node[ext] (a) at (0,0) {};
\node[ext] (b) at (-1.5,2) {};
\node (l0) at (0,-2) {};
\node (l1) at (-3,4.3) {};
\node (l2) at (0, 4.3) {};
\node (l3) at (1.5,2.3) {};
\draw[dotted] (a) -- (l0);
\draw[dotted] (a) -- (b);
\draw[dotted] (b) -- (l1);
\draw (b) -- (l2);
\draw (a) -- (l3);
\end{tikzpicture};
\\ 
\begin{tikzpicture}[scale=0.2]
\node[ext] (a) at (0,2) {};
\node[int] (b) at (0,0) {};
\node (l0) at (0,-2) {};
\node (l1) at (-1.5,4.2) {};
\node (l2) at (1.5, 4.2) {};
\draw (b) -- (l0);
\draw[dotted] (a) -- (b);
\draw[dotted] (a) -- (l1);
\draw (a) -- (l2);
\end{tikzpicture}= -
\begin{tikzpicture}[scale=0.2]
\node[ext] (a) at (0,0) {};
\node[int] (b) at (-1.5,2) {};
\node (l0) at (0,-2) {};
\node (l1) at (-1.5,4.4) {};
\node (l2) at (1.5, 2.2) {};
\draw (a) -- (l0);
\draw (a) -- (b);
\draw[dotted] (b) -- (l1);
\draw (a) -- (l2);
\end{tikzpicture}, 
\quad 
\begin{tikzpicture}[scale=0.2,xscale=-1]
\node[ext] (a) at (0,2) {};
\node[int] (b) at (0,0) {};
\node (l0) at (0,-2) {};
\node (l1) at (-1.5,4.2) {};
\node (l2) at (1.5, 4.2) {};
\draw (b) -- (l0);
\draw[dotted] (a) -- (b);
\draw[dotted] (a) -- (l1);
\draw (a) -- (l2);
\end{tikzpicture}
= -
\begin{tikzpicture}[scale=0.2,xscale=-1]
\node[ext] (a) at (0,0) {};
\node[int] (b) at (-1.5,2.2) {};
\node (l0) at (0,-2) {};
\node (l1) at (-1.5,4.4) {};
\node (l2) at (1.5, 2.2) {};
\draw (a) -- (l0);
\draw (a) -- (b);
\draw[dotted] (b) -- (l1);
\draw (a) -- (l2);
\end{tikzpicture}. 
\end{array}
$$

By adapting the arguments used in Proposition~\ref{prp::Sym::Tradler}, we find that the $2$-colored operad $\Psi(\mathcal{W}^{(d)})$ is isomorphic to the symmetrization of the $2$-colored non-symmetric operad $\mathbb{W}^{(d)}$. The following Proposition details why $\mathbb{W}^{(d)}$ is not Koszul, which implies that the dioperad $\mathcal{W}^{(d)}$ also fails to be Koszul.

\begin{proposition}
\begin{enumerate}[itemsep=0pt,topsep=0pt]
\item The ideal generated by the unary operation $\begin{tikzpicture}[scale=0.15pt, baseline=-0.5ex]
    \node[int] (A) at (0,0) {};     
    \draw (A) -- (0,-2.5);
    \draw[dotted] (A) -- (0,2.5);
\end{tikzpicture}$ in the $2$-colored non-symmetric operad $\mathbb{W}^{(d)}$ is finite-dimensional and possesses the following monomial basis:
\begin{equation}
\label{eq::u:ideal}
\left\{
\begin{tikzpicture}[scale=0.15pt, baseline=(current bounding box.center)]
    \node[int] (A) at (0,0) {};     
    \draw (A) -- (0,-2.5);
    \draw[dotted] (A) -- (0,2.5);
\end{tikzpicture}\ ,
\ \ 
\begin{tikzpicture}[scale=0.2, baseline=(current bounding box.center)]
\node[ext] (a) at (0,0) {};
\node[int] (b) at (-1.5,2.2) {};
\node (l0) at (0,-2) {};
\node (l1) at (-1.5,4.8) {};
\node (l2) at (1.5, 2.2) {};
\draw (a) -- (l0);
\draw (a) -- (b);
\draw[dotted] (b) -- (l1);
\draw (a) -- (l2);
\end{tikzpicture},
\ \
\begin{tikzpicture}[scale=0.2,xscale=-1, baseline=(current bounding box.center)]
\node[ext] (a) at (0,0) {};
\node[int] (b) at (-1.5,2.2) {};
\node (l0) at (0,-2) {};
\node (l1) at (-1.5,4.8) {};
\node (l2) at (1.5, 2.2) {};
\draw (a) -- (l0);
\draw (a) -- (b);
\draw[dotted] (b) -- (l1);
\draw (a) -- (l2);
\end{tikzpicture},
\ \ 
\begin{tikzpicture}[scale=0.2, baseline=(current bounding box.center)]
\node[ext] (a) at (0,0) {};
\node[int] (b) at (-1.5,2.2) {};
\node (l0) at (0,-2) {};
\node (l1) at (-1.5,4.8) {};
\node (l2) at (1.5, 2.2) {};
\draw[dotted] (a) -- (l0);
\draw (a) -- (b);
\draw[dotted] (b) -- (l1);
\draw[dotted] (a) -- (l2);
\end{tikzpicture},
\ \ 
\begin{tikzpicture}[scale=0.2, xscale=-1, baseline=(current bounding box.center)]
\node[ext] (a) at (0,0) {};
\node[int] (b) at (-1.5,2.2) {};
\node (l0) at (0,-2) {};
\node (l1) at (-1.5,4.8) {};
\node (l2) at (1.5, 2.2) {};
\draw[dotted] (a) -- (l0);
\draw (a) -- (b);
\draw[dotted] (b) -- (l1);
\draw[dotted] (a) -- (l2);
\end{tikzpicture},
\ \
\begin{tikzpicture}[scale=0.2, baseline=(current bounding box.center)]
\node[ext] (a) at (0,0) {};
\node[int] (b) at (-1.5,2.2) {};
\node[int] (c) at (1.5,2.2) {};
\node (l0) at (0,-2) {};
\node (l1) at (-1.5,4.8) {};
\node (l2) at (1.5, 4.8) {};
\draw (a) -- (l0);
\draw (a) -- (b);
\draw (a) -- (c);
\draw[dotted] (b) -- (l1);
\draw[dotted] (c) -- (l2);
\end{tikzpicture} 
\right\}
\end{equation} 
\item The dimensions of $\mathcal{W}^{(d)}(m,n)$ are given by:
$$\dim\mathcal{W}^{(d)}(m,n)=\begin{cases}
   m!, & \text{if } (n=1) \text{ and } (m\geq 1), \\
   2, & \text{if } (m=1, n=2) \text{ or } (m=0, n=3), \\
   1, & \text{if } (m=0, n=2), \\
   0, & \text{otherwise.}
\end{cases}
$$
\item Neither the non-symmetric operad $\mathbb{W}^{(d)}$ nor the dioperad $\mathcal{W}^{(d)}$ are Koszul.
\end{enumerate}
\end{proposition}

\begin{proof}
The principal difference between the relations for $\mathbb{V}^{(d)}$ and $\mathbb{W}^{(d)}$ lies in the sign. However, this sign is critical: the resulting rewriting system is no longer confluent. In the rewriting rule corresponding to~\eqref{eq::Tradler::S:polynom}, one changes the sign in three terms. This leads to the vanishing of all elements composed of one unary operation and two binary operations. Consequently, the only remaining operations containing at most one binary generator are those listed in~\eqref{eq::u:ideal}. The dimension counts follow directly from the explicit description of the symmetrization of $\mathbb{W}^{(d)}$.

To disprove the Koszul property, we consider the generating series:
$$\rchi_{\mathcal{W}}(x,y;q)= \sum_{m,n} \dim_q \mathcal{W}(m,n) \frac{x^{m}}{m!}\frac{y^n}{n!} = \frac{xy}{1-qx}  + \frac{q y^2}{2} + q^2 xy^2  + \frac{q^3 y^3}{3}. $$
Let us find the inverse series. I.e. let
$$
\begin{cases}
    u:=\frac{\partial \rchi_{\mathcal{W}}(x,y;-q)}{\partial y} = \frac{x}{1+qx} -qy + 2q^2xy -q^3y^2;\\
    v:=\frac{\partial \rchi_{\mathcal{W}}(x,y;-q)}{\partial x} = \frac{y}{(1+qx)^2} +q^2 y^2.
\end{cases}
$$
Consequently, we have 
$$
\begin{array}{c}
 u +q v = \frac{x}{1+qx} + \frac{qy}{(1+qx)^{2}} -qy(1-2qx) \ \Leftrightarrow 
 \ y = \frac{(1+qx)\Bigl((u+qv)(1+qx)-x\Bigr)}{1-(1-2x)(1+qx)^{2}}. 
 \end{array}
$$
Substituting $y$ into the second equation, we end up with an algebraic equation on $x,u,v$ of degree $7$ on the $x$-variable:
$$
\begin{array}{c}
v = \frac{\Bigl((u+qv)(1+qx)-x\Bigr)}{(1+qx)\Bigl(1-(1-2x)(1+qx)^{2}\Bigr)} + 
\frac{q^2(1+qx)^2\Bigl((u+qv)(1+qx)-x\Bigr)^2}{\Bigl(1-(1-2x)(1+qx)^{2}\Bigr)^2}
\ \Leftrightarrow \\ 
\Leftrightarrow \
\left[
\begin{array}{c}
 v(1+qx){\Bigl(1-(1-2x)(1+qx)^{2}\Bigr)^2} =
 \\ 
= {\Bigl(1-(1-2x)(1+qx)^{2}\Bigr)} \Bigl((u+qv)(1+qx)-x\Bigr) + 
\\ + {q^2(1+qx)^3\Bigl((u+qv)(1+qx)-x\Bigr)^2}
\end{array}
\right]
\end{array}
$$
This equation uniquely determines the coefficients of the expansion $x(u,v) = u+qv+qu^2+\dots$. Computer-aided calculations show that the coefficient of $u^2 v$ is not a polynomial in $q$ with non-negative coefficients. This contradicts the Hilbert series criterion for Koszul dioperads provided in Corollary~\ref{cor::sereis::diop::Koszul}.

Alternatively, one may observe that, as in the $\mathbb{V}^{(d)}$ case, the quadratic dual to the $2$-colored non-symmetric operad $\mathbb{W}^{(d)}$ is $\mathbb{W}^{(1-d)}$. However, their Hilbert series fail to satisfy the functional equation~\eqref{eq::generating::Koszul}.
\end{proof}

\subsection{The dioperad $\QP$ of quadratic Poisson structures}
\label{example::QP}
Recall that a Poisson structure on a flat space $V \cong \mathbb{R}^d$ is defined by a bivector field
$\pi := \sum_{i,j} \pi_{ij} \frac{\partial}{\partial x^i} \wedge \frac{\partial}{\partial x^j}$, 
where $x^1, \dots, x^d$ are coordinates on $V$, satisfying the Poisson identity:
\[
0 = [\pi, \pi]_{Sch} = \sum_{i,j,k,l} \left( \pi_{ij} \frac{\partial \pi_{kl}}{\partial x^i} \frac{\partial}{\partial x^j} \wedge \frac{\partial}{\partial x^k} \wedge \frac{\partial}{\partial x^l} - \pi_{ij} \frac{\partial \pi_{kl}}{\partial x^j} \frac{\partial}{\partial x^i} \wedge \frac{\partial}{\partial x^k} \wedge \frac{\partial}{\partial x^l} \right).
\]
Here, $[\cdot, \cdot]_{Sch}$ denotes the Schouten-Nijenhuis bracket on polyvector fields. A Poisson structure is called \textbf{quadratic} if the coefficients $\pi_{ij}$ are quadratic functions of the coordinates. In this case, the Poisson bivector $\pi$ may be viewed as an element of $S^2V \otimes \Lambda^2 V^{*}$.

The corresponding properad $\QP$ is generated by a single element that is skew-symmetric with respect to its inputs and symmetric with respect to its outputs:
\[
\qpgen{1}{2}{1}{2} = - \qpgen{2}{1}{1}{2} = \qpgen{1}{2}{2}{1} = -\qpgen{2}{1}{2}{1}.
\] 
This generator satisfies the following quadratic dioperadic relations:
\begin{equation}
    \qprel{1}{2}{3}{1}{2}{3} + \qprel{2}{3}{1}{1}{2}{3} + \qprel{3}{1}{2}{1}{2}{3} + \qprel{1}{2}{3}{2}{3}{1} + \qprel{2}{3}{1}{2}{1}{3} + \qprel{3}{1}{2}{2}{1}{3} + \qprel{1}{2}{3}{3}{1}{2} + \qprel{2}{3}{1}{3}{1}{2} + \qprel{3}{1}{2}{3}{1}{2} = 0.
\end{equation}

The deformation theory of the properad $\QP$ and its relationship with the Kontsevich graph complex were discussed in~\cite{Merkulov_Quadratic}. The Koszul property for this properad was recently proven in~\cite{Khor_QP}. In what follows, we provide an alternative perspective by showing that this dioperad admits a quadratic Gr\"obner basis.

The Koszul-dual dioperad $\QP^{!}$ also possesses a single generator of arity $(2,2)$, and all of its non-trivial quadratic compositions are proportional to one another. Specifically:
	$$
	\QP^{!}:= \left\langle \qpgen{1}{2}{1}{2} = (-1)^{\tau} \qpgen{$\sigma(1)$}{$\sigma(2)$}{$\tau(1)$}{$\tau(2)$} \left| 
	\qprel{1}{2}{3}{1}{2}{3} = (-1)^{\tau}\qprel{$\sigma(1)$}{$\sigma(2)$}{$\sigma(3)$}{$\tau(1)$}{$\tau(2)$}{$\tau(3)$} 
	\right.\right\rangle
	$$
where $\sigma, \tau \in S_3$. Consequently, we have:
\[
\dim(\QP^{!})(m,n) = \begin{cases}
    1, & \text{if } m=n, \\ 
    0, & \text{otherwise.}
\end{cases}
\]

Let $\Com_2$ be the suboperad of the commutative operad $\Com$ whose $n$-ary operations are non-zero only for odd $n$:
$$
\Com_2(n) := \begin{cases}
    \Com(n), & \text{if } n=2k+1; \\
    0, & \text{if } n=2k.
\end{cases}
$$
The operad $\Com_2$ is a cyclic operad generated by a single symmetric ternary operation. We refer to~\cite{Dotsenko_Veronese, Khor_single_generated} for a description of its Koszul dual and to~\cite{Etingof_M0n} for the connection between this operad and the real locus of the Deligne-Mumford compactification of the moduli space of marked points. Notably, $\Com_2$ is known to admit a quadratic Gr\"obner basis with caterpillar normal forms.

It follows directly from the definition that $\QP^{!}$ differs from the dioperad $\Theta_{\mathfrak{c}}(\Com_2)$ only by the twisting of outputs by the sign representation:
$$\QP^{!}(n,n) \simeq_{\bS_n \times \bS_n^{\op}} \one_n \otimes \Sgn_n \simeq_{\bS_n \times \bS_n^{\op}} \Theta_{\mathfrak{c}}(\Com_2)(n,n) \otimes \Sgn_n. $$

\begin{proposition}
    The dioperad $\QP$ admits a quadratic Gr\"obner basis.
\end{proposition}

\begin{proof}
    By a straightforward generalization of Theorem~\ref{thm::cycl->diop} and Corollary~\ref{cor::coloring::Koszul} to the coloring $\mathfrak{c} = \{(n,n) : n \in \mathbb{Z}_{>0}\}$, we conclude that $\Theta_{\mathfrak{c}}(\Com_2)$ admits a terminating and confluent quadratic rewriting system. This system is induced by the coloring of the quadratic Gr\"obner basis for the cyclic operad $\Com_2$. It is easily verified that this rewriting system corresponds to a Gr\"obner basis under the path-lexicographical ordering for any choice of ordering on the generators. The resulting set of normal forms consists of right-combs, (the coloring of the inputs and outputs uniquely determines the coloring of the internal edges).

    Note that for the dioperads $\QP^{!}$ and $\Theta_{\mathfrak{c}}(\Com_2)$ -- and, consequently, for their associated $2$-colored shuffle operads $\Psi(\QP^{!})$ and $\Psi(\Theta_{\mathfrak{c}}(\Com_2))$ -- the sets of generators and the leading terms of the quadratic relations coincide. Therefore, the set of monomials not divisible by the leading terms of the quadratic relations is identical for both colored shuffle operads and consists of right-normalized combs.
    Finally, since the dimensions of these operads coincide, the set of normal forms constitutes a basis in both cases. This establishes the existence of a quadratic Gr\"obner basis for $\QP^{!}$. Applying the opposite ordering on monomials yields the quadratic Gr\"obner basis for $\QP$.
\end{proof}

\end{document}